\documentclass[reqno,12pt,letterpaper]{amsart}
\usepackage{amsmath,amssymb,amsthm,graphicx,mathrsfs,url}
\usepackage[usenames,dvipsnames]{color}
\usepackage[colorlinks=true,linkcolor=Red,citecolor=Green]{hyperref}
\usepackage{floatrow}
\usepackage[usenames,dvipsnames]{xcolor}
\usepackage{tikz}
\usepackage{slashed}
\usepackage{graphicx}
\usepackage[colorlinks=true]{hyperref}
\usepackage{color}
\usepackage{amsmath,amsfonts}
\usepackage{calrsfs}
\usepackage{amsmath,environ}
\usepackage{floatrow}
\DeclareMathAlphabet{\pazocal}{OMS}{zplm}{m}{n}
\usepackage{autonum}

\numberwithin{equation}{section}
\usepackage{amsmath, amsthm}
    \usepackage{dsfont}
    \usepackage{fancybox}
    \usepackage{amssymb}

\newcommand{\bsu}{\boldsymbol{\sigma_1}}
\newcommand{\bss}{\boldsymbol{\sigma_\star}}
\newcommand{\bsd}{\boldsymbol{\sigma_2}}
\newcommand{\bst}{\boldsymbol{\sigma_3}}
\newcommand{\bsi}{\boldsymbol{\sigma_i}}
\newcommand{\bsj}{\boldsymbol{\sigma_j}}
\newcommand{\var}{{\vartheta}}
\newcommand{\tvar}{{\var_\sharp}
}
\newcommand{\tf}{{\tilde{f}}}
\newcommand{\dist}{{\operatorname{dist}}}
\newcommand{\Di}{\slashed{D}}
\newcommand{\douts}{{\hspace*{.7mm} . \hspace*{.7mm} . \hspace*{.7mm} . \hspace*{.7mm}}}

\newcommand{\R}{\mathbb{R}}
\newcommand{\N}{\mathbb{N}}
\newcommand{\p}{{\partial}}
\newcommand{\opsi}{{\overline{\psi}}}
\newcommand{\of}{{\overline{f}}}
\newcommand{\w}{\omega}
\newcommand{\dd}[2]{\dfrac{\partial #1}{\partial #2}}
\newcommand{\matrice}[1]{\left[ \begin{matrix}
#1
\end{matrix} \right]}
\newcommand{\Det}{{\operatorname{Det}}}
\newcommand{\ev}{{\operatorname{e}}}
\newcommand{\od}{{\operatorname{o}}}
\newcommand{\pp}{{\operatorname{pp}}}
\newcommand{\PP}{{\mathscr{P}}}
\newcommand{\DD}{{\pazocal{D}}}

\newcommand{\loc}{{\text{loc}}}
\newcommand{\epsi}{\varepsilon}

\newcommand{\te}{\theta}

\newcommand{\blr}[1]{\left\langle #1 \right\rangle}

\newcommand{\GG}{\pazocal{G}}
\newcommand{\DDD}{\mathscr{D}}
\newcommand{\WW}{\pazocal{W}}

\newcommand{\supp}{\mathrm{supp}}

\newcommand{\Z}{\mathbb{Z}}
\newcommand{\Dd}{\mathbb{D}}

\newcommand{\SSS}{\pazocal{S}}

\newcommand{\Ss}{\mathbb{S}}
\newcommand{\Id}{{\operatorname{Id}}}
\newcommand{\BB}{\pazocal{B}}
\newcommand{\VV}{{\pazocal{V}}}

\newcommand{\EE}{\pazocal{E}}

\newcommand{\II}{{\mathscr{I}}}
\newcommand{\EEE}{{\mathscr{E}}}
\newcommand{\MM}{\pazocal{M}}

\newcommand{\UU}{{\pazocal{U}}}
\newcommand{\LL}{{\pazocal{L}}}
\newcommand{\HH}{\pazocal{H}}

\newcommand{\vp}{{\varphi}}

\newcommand{\systeme}[1]{\left\{ \begin{matrix} #1 \end{matrix} \right.}
\newcommand{\ove}[1]{{\overline{#1}}}

\newcommand{\ess}{{\operatorname{ess}}}

\newcommand{\C}{\mathbb{C}}

\newcommand{\oz}{\ove{z}}
\newcommand{\ow}{\ove{\w}}
\newcommand{\oW}{\ove{W}}
\newcommand{\az}{\alpha}

\newcommand{\OO}{{\mathscr{O}}}
\newcommand{\ZZ}{{\pazocal{Z}}}

\newcommand{\KK}{\pazocal K}
\renewcommand{\Re}{\operatorname{Re}}
\renewcommand{\Im}{\operatorname{Im}}

\newcommand{\hf}{{\hat{f}}}

\newcommand{\ii}{\dfrac{1}{2\pi} \int_0^{2\pi}}

\newcommand{\tPi}{\widetilde{\Pi}}

\newcommand{\tD}{\widetilde{D}}

\newcommand{\de}{ \ \mathrel{\stackrel{\makebox[0pt]{\mbox{\normalfont\tiny def}}}{=}} \ }

\title{Defect modes for dislocated periodic media}

\author{A. Drouot, C. F. Fefferman and M. I. Weinstein}

\NewEnviron{equations}{%
\begin{equation}\begin{gathered}
  \BODY
\end{gathered}\end{equation}
}

\newtheorem{thm}{Theorem}
\newtheorem{cor}{Corollary}
\newtheorem{lem}{Lemma}[section]
\newtheorem{prop}{Proposition}[section]
\newtheorem{definition}[lem]{Definition}

\newtheorem{theorem}[thm]{Theorem}
\theoremstyle{definition}
\newtheorem{rmk}{Remark}[section]

\begin{document}
\vspace*{-1cm}
\maketitle
\begin{abstract} We study defect modes in a one-dimensional periodic medium with a dislocation. The model is a periodic Schr\"odinger operator on $\R$, perturbed by an adiabatic dislocation of amplitude  $\delta\ll 1$. If the periodic background admits a Dirac point -- a linear crossing of dispersion curves -- then the dislocated operator acquires a gap in its essential spectrum. For this model (and its honeycomb analog) Fefferman, Lee-Thorp and Weinstein \cite{FLTW1,FLTW2,FLTW3,FLTW4}  constructed defect modes with energies within this gap. The bifurcation of defect modes is associated with the discrete eigenmodes of an effective Dirac operator. 

We improve upon  this result: we show that \textit{all} the defect modes of the dislocated operator
arise from the eigenmodes of the  Dirac operator.  As a byproduct, we derive full expansions of the eigenpairs in powers of $\delta$. The self-contained proof relies on (a) resolvent estimates for the bulk operators; (b) scattering theory for highly oscillatory potentials \cite{Dr2,Dr3,Dr4}. This work significantly advances  the understanding of the topological stability of certain defect states, particularly the bulk-edge correspondence for \textit{continuous} dislocated systems. 
\end{abstract}

\section{Introduction}

The study of systems which exhibit energy localization at material interfaces (line defects in 2D, facets in 3D) has a long history with many fundamental advances and applications. Such models admit edge states, time-harmonic waves propagating \textit{along} rather than \textit{across} the interface.
Recent research focuses on topologically robust edge states, i.e. states that are stable against large local perturbations. 
Examples include the edge states of the quantum hall effect, e.g. \cite{AMU,KDP,Tho,Lau} and  topological insulators in condensed matter physics \cite{FK1,KM1}. 
Haldane and Raghu \cite{HR1,RH} proposed 
a photonic analog of such states. 
This inspired investigations in electronic physics, photonics, acoustics and mechanics; see e.g. \cite{KMY,YVW,WCJ,RZP,SG11,I15,M17,O18}.
 In many examples, topological robustness relates to bifurcations 
 from a {\it Dirac point}, a conical degeneracy in the band spectrum of the  periodic background.
 It appears that Shockley \cite{Sh} first studied this phenomena
 in a one-dimensional periodic setting.
  
The theoretical analysis of such edge states has been mainly carried out in  discrete systems, e.g. tight-binding models and their low energy approximations. In this article we advance the mathematical theory for a continuum model introduced and analyzed in work of Fefferman, Lee-Thorp and Weinstein  \cite{FLTW1,FLTW2}.
The unperturbed model is a continuous periodic Schr\"odinger operator $P_0$, with two dispersion curves  intersecting transversely at a \textit{one-dimensional Dirac point.} We consider an asymptotic dislocation $\PP_\delta$ of $P_0$, obtained by adding an adiabatic phase-shift at $-\infty$ relative to $+\infty$. The dislocation term is of order $\delta$ but its  support is the whole real line. This perturbation induces a  gap in the essential spectrum of $\PP_\delta$, with width $O(\delta)$ about the Dirac point energy. This model shares many features of defects induced by domain walls in higher dimensions. Fefferman, Lee-Thorp, Weinstein and Zhu  exploited this relation with success in the honeycomb lattice \cite{FLTW3,FLTW4,LWZ}.

In \cite{FLTW2}, the authors derive formal full expansions for defect states of $\PP_\delta$ via a multiscale expansion. The leading order term has a two-scale structure. It is a slowly varying combination of the Dirac point Bloch modes, with coefficients solving the eigenvalue problem for an emerging Dirac operator $\Di$. It is next \textit{rigorously} shown that $\PP_\delta$ admits a \textit{genuine} defect state equal to this leading term, modulo $O(\delta^2)$. 

The operator $\Di$ plays a similar role in the formal construction of photonic edge states in Raghu--Haldane \cite{RH}  and relates to the discrete model of a molecular chain introduced by
Su--Schrieffer--Heeger \cite{SSH}. It had appeared earlier in quantum field theory -- see Jackiw--Rebbi \cite{JR:76}. The operator $\Di$ has an odd number of eigenvalues. 
Theorem \ref{thm:2} implies that \textit{all} defect modes of $\PP_\delta$ are  seeded by the eigenpairs of $\Di$. 
Corollary \ref{cor:2} gives \textit{full} expansions of the defect states in powers of $\delta$ -- dramatically improving the error terms in \cite{FLTW2}.

The operator $\Di$ admits a topologically protected zero mode.
The results of this work are used in a crucial way to prove that the {\it topologically protected} character of this mode transfers to the defect mode of a {\it relaxed} $\PP_\delta$-spectral problem. This mode persists even in the non-asymptotic regime $\delta \sim 1$. For further discussion see \S\ref{perspec} and the forthcoming work \cite{Dr}.
 The present analysis as well as that in \cite{FLTW1,FLTW2} studies the $\PP_\delta$-spectral problem for small $\delta$ \textit{only}. Working in this asymptotic regime allows to describe the defect states in a much finer way than just an existence result when $\delta \sim 1$.

\subsection{Description of the model} The model of \cite{FLTW1,FLTW2} consists of a non-local perturbation of the \textit{unperturbed} periodic Schr\"odinger~operator
\begin{equations}
P_0 \de -\dd{^2}{x^2}+V(x) \ = \ D_x^2 +V, \ \ \ \ D_x \de \dfrac{1}{i} \dd{}{x}, 
\\
V \in C^\infty(\R,\R), \ \ V(x+1) = V(x).
\end{equations}
We assume moreover that $\SSS V =V$, where
\begin{equation}
\SSS u(x) \de u(x+1/2).
\label{Shalf}
\end{equation}
Hence, $V$ has period $1/2$ but is regarded as a $1$-periodic function with an additional symmetry. Since $V$ is $1$-periodic, the $L^2$-spectrum of $P_0$ is obtained from the family of operators $P_0(\xi)$ which are formally equal to $D_x^2+V$ but act on the space
\begin{equation}
L^2_\xi \de \left\{f\in L^2_{\rm loc}: f(x+1)=e^{i\xi}f(x)\right\}, \ \ \ \ \xi \in [0,2\pi].
\end{equation}
The parameter $\xi$ is the quasi-momentum. 
For each $\xi\in [0,2\pi]$, the  operator $P_0(\xi)$ has compact resolvent; its spectrum consists of discrete real eigenvalues, listed below with multiplicity:
\begin{equation}
\lambda_{0,1}(\xi) \leq \lambda_{0,2}(\xi) \leq \dots \leq \lambda_{0,j}(\xi) \leq \dots.
\end{equation}
The eigenvalue curves $\xi\mapsto\lambda_{0,j}(\xi)$, $j=1,2,\dots$ are the dispersion curves of $P_0$.

A {\it Dirac point} $(\xi_\star,E_\star)$ is a quasi-momentum / energy pair corresponding to a transverse intersection of dispersion curves of $P_0$; see \S\ref{sec:4.2} for a detailed description.  Due to $\SSS$-invariance, the Dirac points of $P_0$ with $\xi_\star = \pi$ are precisely given by
\begin{equation}
(\pi,E_\star) \ \ \  \text{ where } E_\star \text{ is a } L^2_\pi\text{-eigenvalue of } P_0.
\end{equation}
We can associate to each Dirac point $(\pi,E_\star)$ of $P_0$ a pair $(\phi_+^\star,\phi_-^\star)$ of eigenvectors of $P_0$ for the energy $E_\star$, uniquely characterized modulo a phase factor by
\begin{equation}\label{eq:2h}
(P_0-E_\star) \phi_+^\star = (P_0-E_\star) \phi_-^\star = 0, \ \ 
\SSS \phi_+^\star = i\phi_+^\star, \ \ \SSS \phi^\star_- = -i\phi^\star_- \ \text{ and } \ 
\ove{\phi_+^\star} = \phi_-^\star.
\end{equation}
We refer to the pair $(\phi_+^\star,\phi_-^\star)$ as a ``Dirac eigenbasis".

We now introduce perturbations of $P_0$ that break \eqref{Shalf}. Let $W$ such that
\begin{equation}
W \in C^\infty(\R,\R), \ \ W(x+1) = W(x), \ \ \SSS W = -W.
\end{equation}
Introduce the operators
\begin{equation}
 P_{\pm \delta} \de -\dd{^2}{x^2} + V \pm \delta W  = D_x^2 + V \pm \delta W.
\end{equation}
Observe that $P_{\delta}$ and $P_{-\delta}$ are conjugated:
 $\SSS P_{\delta}= P_{-\delta} \SSS$.
The addition of $\pm\delta W$ to $P_0 = D^2+V$ breaks the $\SSS$-invariance. This perturbation opens a spectral gap $\GG_\delta \subset \R$ about the Dirac energy for $P_{\pm \delta}$ as long as
\begin{equation}\label{eq:2d}
\var_\star \de \blr{\phi_+^\star,W\phi_-^\star} \neq 0.
\end{equation}
The gap $\GG_\delta$ is centered at $E_\star + O(\delta^2)$ and has width $2|\var_\star|\delta + O(\delta^2)$ -- see \cite[Proposition 4.1]{FLTW2} for the simple proof of this fact.

 \begin{center}
\begin{figure}
\caption{Left plot: Schematic of two dispersion curves of $P_0 = D_x^2+V$ which touch at a Dirac point $(\pi,E_\star)$. Right plot: Schematic of two dispersion curves of  $P_\delta = P_0+\delta W $. The perturbation $\delta W $ breaks $\SSS$-symmetry and a spectral gap, $\GG_\delta$, of width $2|\var_\star| \delta + O(\delta^2)$  appears. }\label{fig:1}
\begin{tikzpicture}
      \draw[thick,->] (-.5,-.3) -- (4.5,-.3);
      \draw[thick,->] (0,-.7) -- (0,4.5);
      \node at (4.2,-.6) {$\xi$};
      \node at (-.3,4.2) {$\lambda$};
      \draw[thick] (4,-.7) -- (4,4.5);
      \draw[thick,domain=0:4,smooth,variable=\x,blue] plot ({\x},{\x*\x/4});
      \draw[thick,domain=0:4,smooth,variable=\x,blue] plot ({\x},{(\x-4)*(\x-4)/4});
      \node at (2.95,1) {$(\pi,E_\star)$}; 
      \node[blue] at (2.7,3.3) {$\lambda_{0,j_\star+1}(\xi)$}; 
      \node[blue] at (.7,.7) {$\lambda_{0,j_\star}(\xi)$}; 
      \draw[thick,->] (-.5+6,-.3) -- (4.5+6,-.3);
      \draw[thick,->] (0+6,-.7) -- (0+6,4.5);
      \node at (4.2+6,-.6) {$\xi$};
      \node at (-.3+6,4.2) {$\lambda$};
      \draw[thick] (4+6,-.7) -- (4+6,4.5);
      \draw[thick,domain=0:4,smooth,variable=\x,blue,dashed] plot ({\x+6},{\x*\x/4});
      \draw[thick,domain=0:4,smooth,variable=\x,blue,dashed] plot ({\x+6},{(\x-4)*(\x-4)/4});
      \draw[thick,domain=0:4,smooth,variable=\x,red] plot ({\x+6},{1+((max((\x-4)*(\x-4)/4,\x*\x/4)-1)^2+1/2)^(1/2)});
      \draw[thick,domain=0:4,smooth,variable=\x,red] plot ({\x+6},{1.5-1.3*((min((\x-4)*(\x-4)/4,\x*\x/4)-1)^2+1/2)^(1/2)});

\draw[red,thick,->] (2+6,1) -- (2+6,1.7);
\draw[red,thick,->] (2+6,1) -- (2+6,.58);
      \node[red] at (2.9+6,1.1) {$\sim 2|\vartheta_\star|\delta$};
    \end{tikzpicture}\end{figure}
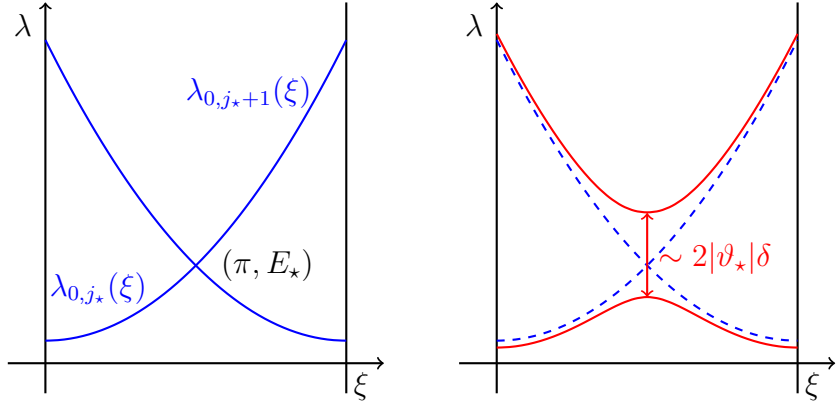
\end{center}

\vspace*{-9mm}

The operator $\PP_\delta$ studied in this article was introduced in  \cite{FLTW1,FLTW2}. It interpolates between $P_{-\delta}$ at $x=-\infty$ and $P_\delta$ at $x=+\infty$. Specifically, we define
\begin{equation}
\PP_\delta \de -\dd{^2}{x^2} + V + \delta \cdot \kappa(\delta \ \cdot) \cdot W  \ = \ D_x^2 + V +\delta \cdot \kappa(\delta \ \cdot) \cdot W.
\end{equation}
Above $\kappa(\delta \ \cdot)(x) \de \kappa(\delta x)$ and $\delta$ is a small parameter. The function $\kappa$ is a smooth  \textit{domain wall}: 
\begin{equation}\label{eq:8g}
 \text{for some } L>0, \  \  \kappa(y) = \systeme{ 1 \text{ for } y \geq L, \\ -1 \text{ for } y \leq -L }; \ \ \ \ \kappa \in C^\infty(\R,\R).
\end{equation}
The essential spectrum of $\PP_\delta$ is determined by the bulk operators $P_{\pm\delta}$: $\Sigma_\ess(\PP_\delta)=\Sigma_\ess(P_{\pm\delta})$. Hence $\PP_\delta$ inherits the same essential spectral gap $\GG_\delta$ as $P_{\pm \delta}$.

We focus on the eigenvalues of $\PP_\delta$ lying in $\GG_\delta$. This problem was recently studied by Fefferman, Lee-Thorp and Weinstein  \cite{FLTW1,FLTW2}. Their multiscale approach produces an effective Dirac operator which involves the parameter $\var_\star$ defined in \eqref{eq:2d}:
\begin{equations}
\Di \de \nu_\star \bst D_y + \bss \kappa,\ \ \ \  {\rm where} \\
\nu_\star \de 2\blr{\phi_+^\star,D_x\phi_+^\star} \neq 0, \ \ \ \ \bss \de \matrice{0 & \ove{\var_\star} \ \\ \var_\star & 0}  =  \Re(\var_\star)\bsu+\Im(\var_\star)\bsd.
\end{equations}

Some spectral properties of $\Di$ are particularly relevant: 
\begin{itemize}
\item the essential spectrum of $\Di$ equals $\R \setminus (-|\var_\star|,|\var_\star|)$;
\item the point spectrum $\{\var_j\}_{j=-N}^N$ is simple and symmetric about the origin;
\begin{equation}
-|\var_\star| < \var_{-N} < \dots < \var_{-1} < \var_0=0 < \var_1 < \dots < \var_N < |\var_\star|,\ \ \ \ \var_{-j} = -\var_j. 
\end{equation}
\end{itemize}
See \S\ref{sec:7.3} for proofs.
In \cite{FLTW2} it is shown that each eigenvalue $\var_j$ of $\Di$ seeds an eigenvalue of $\PP_\delta$, with energy $E_j(\delta)=E_\star + \var_j \delta + O(\delta^2)$. The corresponding eigenvector takes the form
 \begin{equation}
u_{\delta,j}(x) =  \az_+(\delta x) \phi_+^\star(x) + \az_-(\delta x) \phi_-^\star(x), \ \ (\Di-\var_j) \matrice{\az_+ \\ \az_-} = \matrice{0 \\ 0}.
\label{eq:1}
\end{equation}
Therefore  $\PP_\delta$ has at least one eigenvalue in $\GG_\delta$  for $\delta$ sufficiently small.
%

The present paper addresses the  questions:
\begin{itemize}
\item Are all eigenvalues of $\PP_\delta$ in $\GG_\delta$ of the form \eqref{eq:1}?
\item Do eigenpairs admit full expansions in powers of $\delta$? 
\end{itemize}

\subsection{Main results}
We make the following hypotheses: 
\begin{enumerate}
\item[(H1)]  $V$, $W \in C^\infty(\R,\R)$ are $1$-periodic and $\SSS V =V$, $\SSS W = -W$.
\item[(H2)] $(\pi,E_\star)$ is a Dirac point of $P_0$ with Dirac eigenbasis $(\phi_+^\star,\phi_-^\star)$; see \eqref{eq:2h}. 
\item[(H3)]  $\var_\star = \blr{\phi_+^\star,W \phi_-^\star}\neq 0$. 
\end{enumerate}

As mentioned above, when (H1)--(H3) are satisfied, the operator $\PP_\delta$ has no essential spectrum in an interval centered at $E_\star$, of width $2|\var_\star|\delta+O(\delta^2)$. 

\begin{theorem}\label{thm:2}  
Assume that $\operatorname{(H1)}$--$\operatorname{(H3)}$ hold. For any $\tvar \in (\var_N,|\var_\star|)$, there exist $\epsi > 0$ and $\delta_0 > 0$ such that
\begin{itemize}
\item[(A)]  The spectrum of $\PP_\delta$ in $[E_\star-\tvar\delta, E_\star+\delta \tvar]$ is contained in
\begin{equation}\label{eq:2n}
\bigcup_{j=-N}^N\ \II_{\delta,j}, \ \ \ \  \textrm{where} \ \ \II_{\delta,j} \de \big[ E_\star + (\var_j-\epsi) \delta, E_\star+ (\var_j+\epsi) \delta \big].
\end{equation}
\item[(B)]  The spectrum of $\PP_\delta$ in each $\II_{\delta,j}$ consists of at most one eigenvalue. 
\end{itemize}
\end{theorem}

Thanks to variational principles for selfadjoint operators, Theorem \ref{thm:2} implies much more accurate expressions for the eigenvalues $E_{\delta,j}$ and their eigenvectors,
 with respect to the small parameter, $\delta$. We introduce the function spaces:
\begin{equation}
X \de \left\{ v \in C^\infty(\R^2,\C), \ \ v(x+1,y) = v(x,y), \ \ v(x,y) = O\left(e^{-|y|}\right) \right\}.
\end{equation}

\begin{cor}\label{cor:2} Under the assumptions and notations of Theorem \ref{thm:2}, the operator $\PP_\delta$ has exactly one eigenvalue in each subinterval $\II_{\delta,j}$ of the spectral gap $\GG_\delta$. 

Moreover, for all $M \geq 1$ and $k \geq 0$, the associated eigenpair $(E_{\delta,j},u_{\delta,j})$~expands~as
\begin{equations}\label{eq:3r}
E_{\delta,j} = E_\star + \var_j \delta + a_2 \delta^2 + \douts + a_M \delta^M + O\left(\delta^{M+1}\right), \ \ \ \  a_m \in \R;
\\
u_{\delta,j}(x) = v_0(x,\delta x) + \delta v_1(x,\delta x) + \douts + \delta^M v_M(x,\delta x) + o_{H^k}\left(\delta^{M}\right), \ \ \ \ v_m \in X.
\end{equations}
The implicit constants involved in \eqref{eq:3r} depend on $M$ and $k$ and
\begin{itemize}
 \item $a_m \in \R$ and $v_m \in X$ are constructed following a recursive (multiscale) procedure; $H^k$ denotes the Sobolev space of order $k$. 
\item The function $v_0$ appearing in the leading order term satisfies
\begin{equation}\label{eq:1t}
v_0(x,y) = \az_+(y) \phi_+^\star(x) + \az_-(y) \phi^\star_-(x), \ \  (\Di-\var_j)\matrice{\az_+ \\ \az_-} = \matrice{0 \\ 0}.
\end{equation}
\end{itemize}
\end{cor}

Theorem \ref{thm:2} and Corollary \ref{cor:2} characterize the spectrum of $\PP_\delta$ in the gap $\GG_\delta$, arbitrarily close to its edges. These results demonstrate that after suitable rescaling, the spectrum of $\Di$ and $\PP_\delta$ are very similar; see Figure \ref{fig:2}. Figure \ref{fig:3} is a schematic of the ``spray of eigenvalues" of $\PP_\delta$ into the spectral gap $\GG_\delta$. The expression \eqref{eq:1t} relates the eigenfunctions of $\PP_\delta$ to those of $\Di$; to leading order in $\delta$,  the eigenfunction $u_{\delta,j}$ is a  modulation of the Dirac eigenbasis
$(\phi_+^\star,\phi_-^\star)$, with slowly varying amplitudes that form eigenvectors of $\Di$.  The terms $v_m(x,\delta x)$ exhibit a two-scale structure: they oscillate on a unit length scale and decay on the length scale $\delta^{-1}$.

 \begin{center}
\begin{figure}
\begin{tikzpicture}

      \draw[thick] (-5,0) -- (5,0);
      \fill[gray] (-5,-.1) rectangle (-2.5,.1);
    \draw[thick] (-2.5,-.2) --(-2.5,.2);
      \fill[gray] (5,-.1) rectangle (2.5,.1);
    \draw[thick] (2.5,-.2) --(2.5,.2);
        \node at (-2.5,-.5) {$-|\var_\star|$};
        \node at (2.5,-.5) {$|\var_\star|$};

        \node[thick,blue] at (-1.5,0) {$\bullet$};
        \node[blue] at (-1,0) {$\bullet$};
        \node[blue] at (0,0) {$\bullet$};
        \node at (0,-.5) {$0$};
        \node[blue] at (1,0) {$\bullet$};
        \node[blue] at (1.5,0) {$\bullet$};
        \node at (0,.6) {Spectrum of $\Di$};
      \draw[thick] (-5,0-2) -- (5,0-2);
      \fill[lightgray] (-2.5,-.1-2) rectangle (-2.3,.1-2);

      \fill[gray] (-5,-.1-2) rectangle (-2.5,.1-2);
    \draw[thick] (-2.5,-.2-2) --(-2.5,.2-2);
      \fill[gray] (5,-.1-2) rectangle (2.5,.1-2);
      \fill[lightgray] (2.5,-.1-2) rectangle (2.3,.1-2);
    \draw[thick] (2.5,-.2-2) --(2.5,.2-2);
        \node at (-2.5,-.5-2) {$E_\star-|\var_\star|\delta$};
        \node at (2.5,-.5-2) {$E_\star+|\var_\star|\delta$};

        \node[thick,red] at (-1.5,0-2) {$\bullet$};
        \node[red] at (-1,0-2) {$\bullet$};
        \node[red] at (0,0-2) {$\bullet$};
         \node at (0,-.5-2) {$E_\star$};
        \node[red] at (1,0-2) {$\bullet$};
        \node[red] at (1.5,0-2) {$\bullet$};
        \node at (0,.6-2) {Spectrum of $\PP_\delta$};
\end{tikzpicture}
{\caption{Schematic of eigenvalues of $\Di$ in $(-|\var_\star|,|\var_\star|)$ (top, blue dots) and eigenvalues of $\PP_\delta$  in the spectral gap containing $E_\star$ (bottom, red dots). An approximate rescaling equal to $\var \mapsto E_\star+\delta \var + O(\delta^2)$ maps the top to the bottom. Theorem \ref{thm:2} and Corollary \ref{cor:2} do not apply in the lighter gray region near the essential spectrum.}\label{fig:2}}
\end{figure}
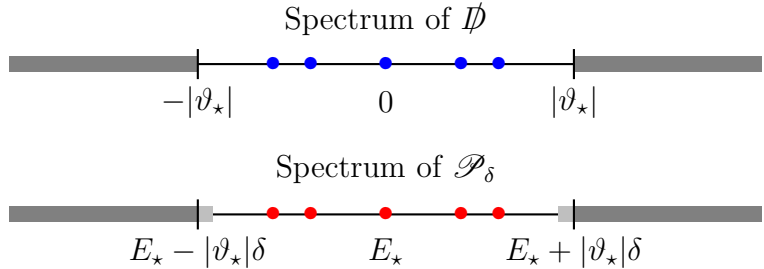
\end{center}

\vspace*{-8mm}

The \textit{formal} multiscale expansion procedure referred to in Corollary \ref{cor:2} is a part of the analysis  in \cite[\S4]{FLTW2}. There, this expansion is rigorously justified at first order via a decomposition of $\PP_\delta$ in terms of dominant quasi-momenta / energy components near $(\pi,E_\star)$; and sub-dominant contributions away from $(\pi,E_\star)$. The  Dirac operator $\Di$ emerges as controlling the point spectrum. This procedure would be formidable to implement at arbitrary order in $\delta$. Although \cite{FLTW2} produces \textit{some} modes of $\PP_\delta$, the method there cannot prove that these are \textit{all} the modes.

Our strategy incorporates the above picture of dominant and subdominant quasi-momenta / energies. It is simpler to implement and yields more information. Estimates on the bulk operators $P_{\pm \delta}$ -- rather than directly on $\PP_\delta$ -- play a key role. The full expansion provided by Corollary \ref{cor:2} follows with little efforts from Theorem \ref{thm:2} and the existence of arbitrarily accurate approximate eigenvectors. To emphasize that Theorem \ref{thm:2} is much deeper than Corollary \ref{cor:2}, we present the proof of Corollary \ref{cor:2} \textit{assuming} Theorem \ref{thm:2}  in \S\ref{sec:4}. The rest of the paper will be devoted to the proof of Theorem \ref{thm:2}.\footnote{The setting  \eqref{eq:8g} is taken for the sake of simplicity. In \cite{FLTW1,FLTW2} it is assumed that $\kappa(x)$ is smooth, approaches $\pm1$ as $x\to\pm\infty$, and that $\kappa'$ decays to zero as $x \rightarrow \pm\infty$ sufficiently rapidly. Many of the results presented here -- without changes in the proofs -- extend to periodic potentials with less regularity; and more general domain wall functions. For instance, assume that
\begin{equation}
\kappa(x) = \pm 1 + O\left( |x|^{-2} \right) \text{ near } \pm \infty; \ \ \ \ \kappa \in H^1_\loc; \ \ \ \ \exists \epsi \in (0,1), \ \ \kappa'(x) \exp\left( - \epsi\dfrac{ |\var_\star x| }{|\nu_\star|} \right) \in L^2.
\end{equation}
If $\Di$ has $M$ eigenvalues in $\big(-\var_\star(1-\epsi),\var_\star(1-\epsi) \big)$ then for $\delta$ sufficiently small, $\PP_\delta$ has at least $M$ eigenvalues in $\big( E_\star-\var_\star(1-\epsi) \delta,E_\star +\var_\star(1-\epsi) \delta \big)$. If in addition $\kappa \in C^2$ and $\kappa', \kappa''$ are $O(|x|^{-2})$ near $\pm \infty$ then $\PP_\delta$ has exactly $M$ eigenvalues in that interval. We leave this verification to the reader.}

\subsection{Overview and strategy} The litterature on edge states generated by small -- possibly non-local -- defects in periodic backgrounds is extensive; see e.g. \cite{DH,FK,Ol,Bor,BG1,PLAJ,Bo7,Bo8,HW,Ze}. Corollary \ref{cor:2} belongs to this class of results. Our approach -- exposed below -- relies on a novel, scattering-oriented framework.

\begin{itemize}

\item[(A)] To prove Theorem \ref{thm:2} and Corollary \ref{cor:2} we show:
\begin{enumerate}
\item[(a)] There is no point spectrum in $\GG_\delta\ \setminus\ \bigcup_{j=N}^{-N}\II_{\delta,j}$ --  the subintervals $\II_{\delta,j}\subset\GG_\delta$ are defined in \eqref{eq:2n}.
\item[(b)] Each interval $\II_{\delta,j}\subset\GG_\delta$ of $\PP_\delta$ contains at most one eigenvalue of $\PP_\delta$.
\item[(c)] Each interval $\II_{\delta,j}\subset\GG_\delta$  contains at least an eigenvalue of $\PP_\delta$; the associated eigenpair admits a full  expansion in powers of $\delta$.
\end{enumerate}

The main part of this paper concerns (a) and (b). Part (c) follows from (b); a variational principle for self-adjoint operators; and the multiscale expansion of \cite{FLTW2}. We show that (b) implies (c) in \S\ref{sec:4}.

\item[(B)] In Section \ref{sec:5.1} we construct a parametrix   $Q_\delta(\lambda)$ for $\PP_\delta-\lambda$ when $\lambda \in \Dd(E_\star,\tvar\delta)$, $\tvar \in (\var_N,|\var_\star|)$;  we recall that $\var_N>0$ denotes the largest eigenvalue of $\Di$. Specifically, there exists $Q_\delta(\lambda) : L^2 \rightarrow H^2$ such that for $\lambda \in \Dd(E_\star,\var_\sharp \delta)$,
\begin{equation}\label{eq:2f}
(\PP_\delta-\lambda)Q_\delta(\lambda) - \Id  \de K_\delta(\lambda) \  \text{ is a compact operator on } L^2.
\end{equation}
The family $\lambda \in \Dd(E_\star,\tvar\delta) \mapsto K_\delta(\lambda)$ depends holomorphically on $\lambda$. Analytic Fredholm theory shows that $(\Id+K_\delta(\lambda))^{-1}$ is meromorphic. We can then write
\begin{equation}\label{eq:4f}
(\PP_\delta-\lambda)^{-1} = Q_\delta(\lambda) \cdot \big(\Id+K_\delta(\lambda)\big)^{-1}.
\end{equation}
Thus each eigenvalue of $\PP_\delta$ in $[E_\star-\tvar \delta, E_\star+\tvar\delta]$ is a pole of of $(\Id+K_\delta(\lambda))^{-1}$.

\item[(C)] We then seek  ``zeros" of $\Id+K_\delta(\lambda)$ using holomorphic tools. This requires estimates of $K_\delta(\lambda)$ for $\delta$ near $ 0$. This happens to be related to an expansion of the resolvent difference 
\begin{equation}
(P_\delta-\lambda)^{-1} - (P_{-\delta}-\lambda)^{-1}.
\end{equation}

\item[(D)] A key ingredient is Theorem \ref{thm:1}. It shows that modulo rescaling and matrix-valued multiplications related to $\phi_+^\star$ and $\phi_-^\star$, the resolvent $(P_\delta-\lambda)^{-1}$ behaves like a $2 \times 2$ Fourier multiplier:
\begin{equations}\label{eq:4g}
\tvar \in (0,|\var_\star|), \ \delta \in (0,\delta_0), \ z \in \Dd(0,\tvar), \
 \lambda = E_\star+z \delta
\\
\Rightarrow  \ 
(P_\delta-\lambda)^{-1} = \dfrac{1}{\delta} \matrice{ \phi_+^\star \\ \phi_-^\star  }^\top     \UU_\delta \cdot \big( \Di_+-z \big)^{-1} \cdot  \UU_\delta^{-1}  \ove{\matrice{ \phi_+^\star \\ \phi_-^\star  }} + \OO_{L^2}(\delta^{1/3}),
\end{equations}
where $\UU_\delta f(x) = \delta^{1/2}f(\delta x)$ and $\Di_+=\nu_\star \bst D_x + \bss$.

\item[(E)] The expansion \eqref{eq:4g} allows us to identify the poles of $(\Id+K_\delta(\lambda))^{-1}$ with the eigenvalues of a matrix Schr\"odinger operator with highly oscillatory $2\times2$ potential $\MM(x/\delta,x)$ where $\MM(x,y)$ is one-periodic in $x$, compactly supported in $y$ and smooth in both variables. This class of problems
 arises in the study of waves in microstructures and has been studied extensively; see e.g. \cite{GW,BorG,DW,DVW,Dr3,Dr2,Dr4,DR}. In particular, the first author's work \cite{Dr2} shows that the eigenvalues of such operators are approached by those of their average. This homogenized operator happens to be $\Di{}^2$. 
This identifies the poles of $\left(\Id + K_\delta(\lambda)\right)^{-1}$ in the spectral gap $\GG_\delta$ with the eigenvalues of $\Di{}^2$ modulo negligible errors. Only a very weak form of the results of \cite{Dr2} is needed here; the presentation is self-contained.

\item[(F)] But there's a glitch. The operator $\Di{}^2$ has twice as many eigenvalues as predicted in Theorem \ref{thm:2} and Corollary \ref{cor:2}. This is due to the factorization \eqref{eq:4f} of $\PP_\delta-\lambda$: the parametrix $Q_\delta(\lambda)$ in \eqref{eq:4f} happens to cancel some of the poles of $(\Id+K_\delta(\lambda))^{-1}$.  This phenomena is however unstable and the parametrix is not unique:  we may for example perturb $Q_\delta(\lambda)$ by a rank-one operator and still respect \eqref{eq:2f}. We use these extra degrees of freedom to produce a {\it twisted parametrix},
 which  {\it deforms away} the spurious poles. With this final piece in place, we obtain Theorem \ref{thm:2}. 
\end{itemize}

Finally, we remark that the Dirac operator $\Di$ arises in other settings \cite{JR:76,SSH,HR1,RH,LWW}. Our detailed results on the spectrum of $\Di$, particularly the simplicity of eigenvalues established in Appendix \ref{app:3}, are of independent interest.

\subsection{Further perspectives}\label{perspec}  From the mathematical physics point of view, the most important advance of this work is the precise counting of eigenvalues of $\PP_\delta$ in the gap $\GG_\delta$. In a forthcoming work of the first author \cite{Dr}, the Hamiltonian $\PP_\delta$ is embedded in a one-parameter family of operators $t \in \R/(2\pi\Z) \mapsto \PP_\delta(t)$. Such systems were previously studied in \cite{Ko,Ko2,HK,DPR} -- see \cite{HK2,HK3} for two-dimensional extensions. Following the mathematical physics literature, we define:
\begin{itemize}
\item An edge index: the spectral flow of $t \in \R/(2\pi\Z) \mapsto \PP_\delta(t)$ at $E_\star$, i.e. the signed number of eigenvalues of $\PP_\delta(t)$ crossing $E_\star$ as $t$ runs through $\R/(2\pi\Z)$.
\item A bulk index: the Chern number of an eigenbundle associated to the asymptotic operator $P_\delta(t)$ for $\PP_\delta(t)$.
\end{itemize}
The bulk-edge correspondence is a general principle that formally asserts that these two indexes should be equal. Although many proofs are known for discrete models, e.g. tight-binding \cite{EGS,GP,PS,ASV,Ba,Sha,Br,GS,GT,ST} the problem is widely open for the continuum. Corollary \ref{cor:2} is a key tool in computing the edge index for the dislocation systems $\PP_\delta(t)$. This provides a second \textit{continuous} setting where the bulk-edge correspondence is shown to hold rigorously -- the first one being the Quantum Hall effect \cite{Taa}. In particular, the zero mode of $\Di$ seeds a \textit{topologically protected} mode of $\PP_\delta$: this state persists under large spatially localized deformations of $\PP_\delta(t)$. 

The bulk-edge correspondence has stimulated the search for accurate estimates on the number of eigenvalues of $\PP_\delta$ in the gap $\GG_\delta$. 
Previous techniques -- developed in \cite{FLTW2} -- could not produce the precise counting needed to compute the edge index. They eventually relied on the inverse function theorem, which \textit{needs} approximate defect states as starting points to construct exact ones.

\begin{center}
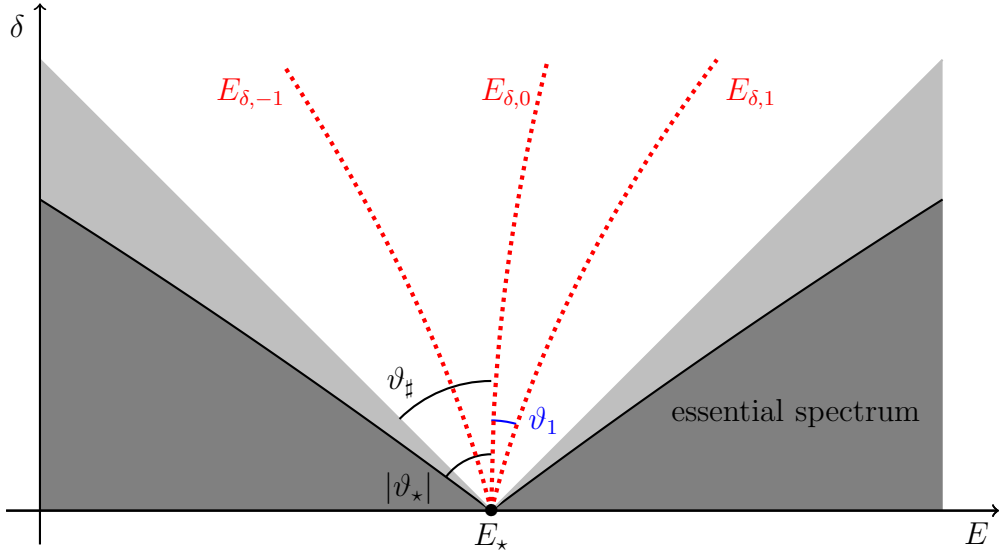
\begin{figure}
\begin{tikzpicture}
           \fill [lightgray, domain=4*1.5:0, variable=\x]
      (-4*1.5, 0)
      -- plot ({-\x}, {\x})
      -- (0, 0)
      -- cycle;
           \fill [lightgray, domain=4*1.5:0, variable=\x]
      (4*1.5, 0)
      -- plot ({\x}, {\x})
      -- (0, 0)
      -- cycle;

           \fill [gray, domain=4*1.5:0, variable=\x]
      (-4*1.5, 0)
      -- plot ({-\x}, {.75*\x - \x^2/100})
      -- (0, 0)
      -- cycle;
           \fill [gray, domain=4*1.5:0, variable=\x]
      (4*1.5, 0)
      -- plot ({\x}, {.75*\x- \x^2/100})
      -- (0, 0)
      -- cycle;

           \draw [thick,domain=4*1.5:0, variable=\x]
 plot ({\x}, {.75*\x- \x^2/100});
           \draw [thick,domain=4*1.5:0, variable=\x]
 plot ({-\x}, {.75*\x- \x^2/100});

      \draw[thick,lightgray] (0,0) -- (-4*1.5,4*1.5);
      \draw[thick,lightgray] (0,0) -- (4*1.5,4*1.5);

      \draw[thick] (0,1.15*1.5) arc (90:135:1.15*1.5);
      \draw[thick] (0,.5*1.5) arc (90:143:.5*1.5);
      \draw[thick,blue] (0,.8*1.5) arc (90:74:.8*1.5);

\node[blue] at (.7,1.2){$\vartheta_1$};
\node at (-1.2,1.7) {$\vartheta_\sharp$};
\node at (-1.1,.3) {$|\vartheta_\star|$};
\node at (2.7*1.5,1.3) {essential spectrum};
\node[red] at (2.3*1.5,3.7*1.5) {$E_{\delta,1}$};
\node[red] at (0.2,3.7*1.5) {$E_{\delta,0}$};
\node[red] at (-3.2,3.7*1.5) {$E_{\delta,-1}$};

   \draw[ultra thick,domain=0:.5,smooth,variable=\x,red,dotted] plot ({\x*1.5},{5.65*sqrt(\x)*1.5});
       \draw[ultra thick,domain=.25:2.25,smooth,variable=\x,red,dotted] plot ({(\x-.25)*1.5},{(4*sqrt(\x)-2)*1.5});
       \draw[ultra thick,domain=.25:2.09,smooth,variable=\x,red,dotted] plot ({(-\x+.25)*1.5},{(4*sqrt(\x)-2 + (\x-.25)*(\x-.25)/20)*1.5});

      \draw[thick,->] (-4.3*1.5,0) -- (4.5*1.5,0);
      \draw[thick,->] (-4*1.5,-.3*1.5) -- (-4*1.5,4.5*1.5);
      \draw[thick,->] (-4.3*1.5,0) -- (4.5*1.5,0);

\node at (-4*1.5-.3,4.3*1.5) {$\delta$};
\node at (4.3*1.5,-.3) {$E$};
\node at (0,-.35) {$E_\star$};
\node at (0,0) {$\bullet$};
\end{tikzpicture}
{\caption{Eigenvalues of $\Di$ seed eigenvalues of $\PP_\delta$ (red dotted curves) from the Dirac energy $E_\star$ into the gap $\GG_\delta$  as $\delta$ increases. The slopes of these curves at $\delta = 0$ equal the eigenvalues of $\Di$. The essential spectrum corresponds to the dark-gray region and the ``unexplored" zone $\tvar \delta < |E-E_\star| < \tvar \delta$ is the light-gray region.}\label{fig:3}}
\end{figure}
\end{center} 

\vspace*{-8mm}

The papers \cite{FLTW1,FLTW2} served as a model to construct edge states in continuous honeycomb structures with line defects \cite{FLTW3,FLTW4,LWZ}. This proved predictions of Raghu--Haldane \cite{RH}. We will similarly extend the techniques of this article to the honeycomb case. We expect that it will lead to topological results; see e.g. \cite{HR1} for predictions. 

We hope to adapt our methods to deal with failures of the {\it spectral no-fold condition} \cite[\S1.3]{FLTW3}. Although parity breaking ensures that a local $L^2_\xi$-spectral gap opens for $\xi$ near $\xi_\star$, in higher dimensions this may not induce a $L^2$-gap because of global aspects of dispersion surfaces.\footnote{The one-dimensional setting is really rigid: the underlying equation is a second-order ODE. This implies monotonicity of the dispersion curves -- see \cite[Theorem XIII.89]{RS}. In particular, the no-fold condition always holds in 1D.}
Fefferman, Lee-Thorp and Weinstein \cite[\S1.4]{FLTW3} conjectured that there should nonetheless exist long-lived edge states. 
Their multiscale expansion produces arbitrarily accurate quasimodes; this formally indicate existence of resonances near the Dirac energy. We refer to 
 \cite{SV,TZ,St1,St2,Ga1} for the relation between quasimodes and resonances in other settings; and to \cite{GeS} for their time-dependent dynamics.

We speculate that some of the results presented here have analogs in random situations where averaging effect arise -- such as the settings of \cite{BZ1,BZ2,BGu1,BG2,BG3,Dr4}. The bulk-edge correspondence, for instance, has been investigated in the random context for the Quantum Hall effect \cite{Taa}. In our situation the discussion has remained speculative: it is not clear what random model to consider.\\

\noindent {\bf Notations.} We will use the following notations:
\begin{itemize}
\item $\C^+=\{z\in\C: \Im z>0\}$ and $\Ss^1 = \{ z \in \C : |z| = 1\}$.
\item If $z \in \C$, $\bar{z}$ is its complex conjugate and $\Dd(z,r)=\{\zeta\in\C\ :\ |\zeta-z|<r\}$.
\item If $\HH$ is a Hilbert space and $\psi \in \HH$, we write $|\psi|_\HH$ for the norm of $\HH$; if $A : \HH \rightarrow \HH$ is a bounded operator of $\HH$, the operator norm of $A$ is
\begin{equation}
\| A \|_\HH \de \sup \left\{ |A\psi|_\HH : |\psi|_\HH=1 \ \right\}.
\end{equation}
\item If $A$ is an operator, $\Sigma(A)$ denotes the spectrum of $A$; $\Sigma_\ess(A)$ denotes the essential spectrum of $A$; $\Sigma_\pp(A)$ denotes the pure point spectrum of $A$.
\item If $\psi_\epsi \in \HH$ (resp. $A_\epsi : \HH \rightarrow \HH$ is a linear operator) and $f : \R \rightarrow \R$, we write $\psi_\epsi = O_\HH\big(f(\epsi)\big)$ $\big($resp. $A_\epsi = \OO_\HH(f(\epsi))\big)$ when there exists $C > 0$ such that $|\psi_\epsi|_\HH \leq Cf(\epsi)$ (resp. $\|A_\epsi\|_\HH  \leq C f(\epsi)$). 
\item $D_x$ is the operator $\frac{1}{i}\p_x$.
\item To a kernel distribution $K(x,y)$ taking values in $\R^2$, we associate the operator 
\[ (Kf)(x)=\int_\R K(x,y)f(y)dy.\] 
\item $L^2$ denotes the space of square-summable functions on $\R$ and $H^k$, for $k\in\R$,
are the classical Sobolev spaces. 
\item The space $L^2_\xi$ consists of $\xi$-quasiperiodic functions:
\begin{equation}
L^2_\xi \de \left\{ u\in L^2_\loc : \ u(x+1) = e^{i\xi} u(x) \right\}.
\end{equation}
It splits orthogonally as $L^2_\xi = L^2_{\xi,\ev} \oplus L^2_{\xi,\od}$ with
\begin{equations}
L^2_{\xi,\ev} \de \left\{ u \in L^2_\loc : \ u(x+1/2,\xi) = e^{i\xi/2} u(x)\right\}, 
\\
 L^2_{\xi,\od} \de \left\{ u \in L^2_\loc : \ u(x+1/2,\xi) = -e^{i\xi/2} u(x)\right\}.
\end{equations}
\item $V$, $W$ always denote potentials that are smooth, real-valued and such that $V(x+1/2) = V(x)$ and $W(x+1/2) = -W(x)$. 
\item $P_0$ is the operator $D_x^2\ + V$.
\item $P_{\pm\delta}$ is the operator $D_x^2 + V \pm \delta\cdot W$.
\item $\PP_\delta$ is the operator  $D_x^2 + V +\delta\cdot \kappa(\delta \ \cdot)\cdot W$
\item $(\pi,E_\star)$ is a Dirac point of $P_0$ with a Dirac eigenbasis $(\phi_+^\star,\phi_-^\star) \in L^2_{\pi,\ev} \times L^2_{\pi,\od}$.
 \item The Pauli matrices are
\begin{equation}\label{eq:1f}
\bsu = \matrice{0 & 1 \\ 1 & 0}, \ \ \bsd = \matrice{0 & -i \\ i & 0}, \ \  \bst = \matrice{1 & 0 \\ 0 & -1}.
\end{equation}
These matrices satisfy $\bsj^2 = \Id_2$ and $\bsi \bsj = -\bsj \bsi$ for $i \neq j$.
\item $\nu_\star = 2\blr{D_x\phi_+^\star,\phi_+^\star}$; $\var_\star = \blr{\phi_+^\star,W  \phi_-^\star}$; and
\begin{equation}
\bss \de \matrice{0 & \ove{\var_\star} \ \\ \var_\star & 0} = \Re(\var_\star) \bsu + \Im(\var_\star) \bsd.
\end{equation}
\item $\Di$ is is the operator $\nu_\star \bst D_x + \bss \kappa$; $\Di_\pm$ is the operator $\nu_\star \bst D_x \pm \bss$.
\item If $E,F \subset \C$, $\dist(E,F)$ denotes the Euclidean distance between $E$ and $F$.
\end{itemize}

\noindent{\bf Acknowledgments.} This research was supported in part by 
National Science  Foundation grants DMS-1800086 (AD), DMS-1265524 (CF), DMS-1412560, DMS-1620418 (MIW) and Simons Foundation Math + X Investigator Award \#376319 (MIW and AD). The authors thank A. B.  Watson for very stimulating discussions and S. Dyatlov for suggesting a more conceptual proof to Proposition \ref{prop:1b}.

\section{Preliminaries}

\subsection{Floquet--Bloch theory} We review here how to study periodic operators from their Floquet--Bloch representation. Basic references are \cite[Chapter XIII]{RS} and \cite{Ku}. We equip the space $L^2_\xi$ of $\xi$-quasiperiodic functions 
\begin{equation}
L^2_\xi \de \left\{ f \in L^2_\loc(\R) : u(x+1) = e^{i\xi} u(x) \right\} 
\end{equation}
with the Hermitian inner product $\left\langle f,g\right\rangle_{L^2_\xi} = \int_0^1 \ove{f}g$ 
and the norm $|f|_{L^2_\xi}^2 = \int_0^1 |f(x)|^2 dx$. In particular, $L^2_0$ is the space of periodic square-summable functions.

Assume that $T$ is a $1$-periodic operator, i.e. $f\in\textrm{dom}(T)$ implies $(Tf)(\cdot+1)=T \left(f(\cdot+1)\right)$.  Then $T$ maps $L^2_\xi$ to itself and we denote this operator by $T(\xi)$. If $f \in C_0^\infty(\R,\C)$, then $f$ has a $L^2_\xi$-decomposition, based on the Fourier inversion formula:
\begin{equation}
f(x) = \ii \tf(\xi,x)  d\xi, \ \ \tf(\xi,x) \de e^{i\xi x}\ \sum_{k \in \Z} \hf(\xi-2k\pi) e^{-2\pi i k x } \in L^2_\xi. 
\end{equation}
The action of $T$ on $L^2$ can be studied using the reconstruction formula 
\begin{equation}
Tf(x) = \ii \big(T(\xi)\tf(\xi,\cdot)\big)(x) d\xi.
\end{equation} 
This formula extends to all $f \in L^2(\R,\C)$ by density. We will write
\begin{equation}\label{eq:2j}
 T = \frac{1}{2\pi}\ \int_{[0,2\pi]}^\oplus d\xi\ T(\xi),
 \end{equation}
emphasizing that each $T(\xi)$ acts fiberwise on $L^2_\xi$.

If now $T=D_x^2 + \VV$ is a Schr\"odinger operator 
where $\VV \in C^\infty(\R,\R)$ is $1$-periodic then  the resolvent of $T(\xi)$ is selfadjoint and compact. The spectrum of $T(\xi)$ consists of point eigenvalues
$E_1(\xi)\le \dots \le E_j(\xi)\le\cdots$, listed with multiplicity  and 
 tending to $+\infty$. The curves $\xi\mapsto E_j(\xi)$ obtained by varying $\xi\in[0,2\pi]$ are called {\it dispersion curves} of the periodic operator $T$. The collection of $L^2_\xi$-spectra of $T(\xi)$ sweeps out the $L^2$-spectrum of $T$:
 \begin{equation}
 \Sigma_{L^2}(T) = \bigcup_{\xi \in [0,2\pi]} \Sigma_{L^2_\xi}(T) = \big\{ E_j(\xi) : j \geq 1, \ \xi \in [0,2\pi]\big\}.
 \end{equation}

We will use the following stability result for dispersion surfaces. It follows from \cite[Appendix A]{FW2}, a general stability result for selfadjoint eigenvalues problems.

\begin{lem}\label{lem:1m} Let $\VV$ and $\WW$ be two periodic potentials. Denote by $E_{\VV,\ell}(\xi)$ (resp. $E_{\VV+\WW,\ell}(\xi)$) the $\ell$-th $L^2_\xi$-eigenvalue of $D_x^2+\VV$ (resp. $D_x^2+\VV+\WW$). For any $\ell\ge1$, there exists a constant $C(\ell)$ such that 
\begin{equation}
 \big| E_{\VV+\WW,\ell}(\xi) - E_{\VV,\ell}(\xi) \big|\ \leq\ C(\ell) |\WW|_\infty.
\end{equation}
\end{lem}

\subsection{The spaces $L^2_{\xi,\ev}$ and $L^2_{\xi,\od}$}\label{sec:2.1} Fix $\xi \in [0,2\pi]$. The operator $\SSS$ is unitary on $L^2_\xi$ and $\SSS^2 = e^{i\xi} \cdot \Id_{L^2_\xi}$. Hence,  $\SSS$ has eigenvalues $\big\{e^{i\xi/2},-e^{i\xi/2}\big\}$. Since  $\SSS$ is a normal operator on $L^2_\xi$, there exists an orthogonal decomposition of $L^2_\xi$ in terms of eigenspaces of $\SSS$:
 \begin{equation}\label{eq:1b}
 L^2_\xi = L^2_{\xi,\ev} \oplus L^2_{\xi,\od},  \ \ \ \ L^2_{\xi,\ev} \de \ker\left(\SSS-e^{i\xi/2}\right), \ \ \ \ L^2_{\xi,\od} \de \ker\left(\SSS+e^{i\xi/2}\right).
\end{equation}
Functions in $L^2_{\xi,\ev}$ (resp. $L^2_{\xi,\od}$) admit even-indexed Fourier series representations (resp. odd-indexed Fourier series representations) -- see \cite[Definition 3.3]{FLTW2}. This explains the subscripts e/o. 

When $\xi \in (0,2\pi)$, these spaces are also related via complex conjugation: $\ove{L^2_{\xi,\ev}} = L^2_{2\pi-\xi,o}$. Indeed, if $f \in L^2_{\xi,\ev}$ then
\begin{equation}
\ove{f(x+1/2)} = e^{-i\xi/2} \ove{f(x)} = -e^{i(2\pi-\xi)/2}\ove{f(x)},
\end{equation}
which means that $\of \in L^2_{2\pi-\xi,o}$.

We recall that $P_0 = D_x^2 + V$, where $V \in C_0^\infty(\R,\R)$ is one-periodic and satisfies $\SSS V = V$. Because $P_0$ is $\SSS$-invariant, we can regard  $P_0$ as an operator acting on $L^2_{\xi,\ev}$ or $L^2_{\xi,\od}$. When $\xi \in (0,2\pi)$ the eigenvalues of $P_0(\xi)$ on  $L^2_{\xi,\ev}$ and on  $L^2_{\xi,\od}$ are simple. Indeed, since $V$ is real-valued,
\begin{equation}
(P_0-E)\psi = 0, \ \ \psi \in L^2_{\xi,\ev} \ \ \Rightarrow \ \ (P_0-E)\opsi = 0, \ \ \ \ \opsi \in L^2_{2\pi-\xi,o}.
\end{equation}
If $E$ were a $L^2_{\xi,\ev}$-eigenvalue of $P_0(\xi)$ of multiplicity at least $2$, then the solution space of the ODE $(P_0(\xi)-E)u=0$ would have dimension at least $4$. This is impossible because the dimension of this space is at most $2$.

\subsection{$P_0$ and Dirac points}\label{sec:4.2} We recall that $V \in C^\infty(\R,R)$ is $1$-periodic and satisfies $\SSS V = V$. The operator $P_0(\xi)$ denotes the operator $P_0 = D_x^2 +V$ acting in the space $L^2_\xi$.
Let  $\lambda_{0,j}(\xi)$ denote the $j$-th eigenvalue  (listed with multiplicity) of $P_0(\xi)$. We are interested in the situation
 where two dispersion curves intersect at a quasi-momentum / energy pair $(\pi,E_\star)$.
 
\begin{definition}\label{def:1}
We say that $P_0$ has a Dirac point at the  quasi-momentum / energy pair $(\pi,E_\star)$ if there exist an integer $j_\star$ and a constant $\nu_F>0$ such that
\begin{equations}
\lambda_{0,j_\star}(\xi) = E_\star - \nu_F \cdot |\xi-\pi| + O(\xi-\pi)^2, \\
\lambda_{0,j_\star+1}(\xi) = E_\star + \nu_F \cdot |\xi-\pi| + O(\xi-\pi)^2.
\end{equations}
\end{definition}

\begin{rmk} The expression ``Dirac point'' typically refers to the conical intersection of dispersion surfaces for Schr\"odinger operators  in two or more dimensions. They arise e.g. for honeycomb structures  -- see \cite{CV,BC,Le} and notably \cite{FW} which shows that they generically exist. In \cite[Appendix D]{FLTW2} the authors show that 1D Dirac points also arise generically in one-dimensional lattices with an extra symmetry. Watson \cite[Appendix B.1]{Watson:17} gives a proof that in one dimension every eigenvalue $\lambda_{0,j}(\pi)$ of $P_0(\pi)$ corresponds to a Dirac point $(\pi,\lambda_{0,j}(\pi))$; this was used in \cite{WW}. We provide a simpler proof in \cite[Appendix A.2]{Dr}.

 The notation $\nu_F$ is used since this constant plays the role of the Fermi velocity arising in mathematical theory of graphene; see, for example, \cite{CGP,FW,FLW6}.
\end{rmk}

\begin{prop}\label{prop:1w} Let $(\pi,E_\star)$ denote a Dirac point of $P_0$ in the sense of Definition \ref{def:1}. Then for each $\xi \in (0,2\pi)$, there exist $L^2_\xi$-normalized eigenpairs $(\lambda_+(\xi),\phi_+(\xi))$ and $(\lambda_-(\xi),\phi_-(\xi))$ of $P_0(\xi)$ with the following properties:
\begin{enumerate}
\item $\phi_+(\xi) \in L^2_{\xi,\ev}$ and $\phi_-(\xi) \in L^2_{\xi,\od}$.
\item If $\phi_+^\star \de \phi_+(\pi)$ and $\phi_-^\star \de \phi_-(\pi)$ then $\ove{\phi_+^\star} = \phi_-^\star$.
\item The maps $\xi \mapsto (\lambda_+(\xi),\phi_+(\xi))$ and $\xi \mapsto (\lambda_-(\xi),\phi_-(\xi))$ are real-analytic.
\item For $\xi$ in a neighborhood of $\pi$,
\begin{equation}
\begin{matrix}
\lambda_+(\xi) = E_\star + \nu_\star \cdot (\xi-\pi) + O(\xi-\pi)^2,  \\
\lambda_-(\xi) = E_\star - \nu_\star \cdot (\xi-\pi) + O(\xi-\pi)^2,
\end{matrix} \ \ \ \ \ \nu_\star \de 2\blr{\phi_+^\star,D_x\phi_+^\star}.\label{pm-disp}
\end{equation}
\end{enumerate}
Furthermore, if $(\mu_+(\xi),\psi_+(\xi))$ and $(\mu_-(\xi),\psi_-(\xi))$ are other normalized eigenpairs of $P_0(\xi)$ satisfying $(1)$--$(4)$ then 
\begin{equation}\label{eq:1g}
\textrm{for some}\  \w \in \Ss^1, \textrm{we have:}\ \ \w \cdot \phi_+^\star =  \psi_+(\pi), \ \ \ow \cdot \phi_-^\star = \psi_+(\pi).
\end{equation}
\end{prop}

\begin{proof} 1. Let $(\pi,E_\star)$ be a Dirac point of $P_0$. Because of the decomposition \eqref{eq:1b}, we can write
\begin{equation}
\ker(P_0(\pi)-E_\star) = \left( \ker(P_0(\pi)-E_\star) \cap L^2_{\xi,\ev} \right) \oplus \left( \ker(P_0(\pi)-E_\star) \cap L^2_{\xi,\od} \right).
\end{equation}
Since $(\pi,E_\star)$ is a Dirac point of $P_0$, the space on the LHS has dimension $2$. The subspaces on the RHS have dimension at most $1$ -- see the discussion at the end of \S\ref{sec:2.1}.  We deduce that 
\begin{equation}\label{eq:1d}
\dim \left( \ker(P_0(\pi)-E_\star) \cap L^2_{\xi,\ev} \right) = \dim \left( \ker(P_0(\pi)-E_\star) \cap L^2_{\xi,\od} \right) = 1.
\end{equation}
Let $\phi_+^\star$ be a normalized element of $\ker(P_0(\pi)-E_\star) \cap L^2_{\xi,\ev}$; define $\phi_-^\star = \ove{\phi_+^\star}$.

2. Because of Step 1, $E_\star$ is a $L^2_{\xi,\ev}$-eigenvalue of $P_0(\pi)$. For $\xi \in (0,2\pi)$ the $L^2_{\xi,\ev}$-eigenvalues of $P_0$ are simple; and $P_0(\xi)$ is real analytic in $\xi$ in the sense that
\begin{equation}
e^{-i\xi x} P_0(\xi) e^{i\xi x} = (D_x+\xi)^2+V : L^2_0 \rightarrow L^2_0
\end{equation}
is real-analytic in $\xi$. Therefore standard results on perturbation of eigenvalues show that for $\xi \in (0,2\pi)$, $(E_\star, \phi_+^\star)$ seeds a unique  $L^2_{\xi,\ev}$-eigenpair 
$(\lambda_+(\xi),\phi_+(\xi))$ of $P_0$ which is real-analytic in $\xi$. Similarly, $(E_\star, \phi_-^\star)$ seeds a unique  $L^2_{\xi,\od}$-eigenpair $(\lambda_-(\xi),\phi_-(\xi))$ of $P_0$ which is real-analytic in $\xi$. 

3. So far we have proved assertions (1)--(3) of Proposition \ref{prop:1w}. Because the eigenvalues of $P_0(\pi)$ have multiplicity at most $2$, Step 2 implies that
\begin{equation}
\{ \lambda_+(\xi),\lambda_-(\xi) \} = \{ \lambda_{0,j_\star}(\xi),\lambda_{0,j_\star+1}(\xi) \}.
\end{equation}
Since the maps $\xi \mapsto \lambda_\pm(\xi)$ are smooth, we deduce that
\begin{equations}
\lambda_+(\xi) = E_\star + \nu_\star (\xi-\pi) + O(\xi-\pi)^2,  \ \ \ \
\lambda_-(\xi) = E_\star - \nu_\star (\xi-\pi) + O(\xi-\pi)^2,
\end{equations}
where $\nu_\star$ is equal to either $\nu_F$ or $-\nu_F$. We now prove that $\nu_\star = 2\blr{\phi_+^\star,D_x\phi_+^\star}$. Let $p_+(\xi,x) = e^{-i\xi x}\phi_+(\xi,x)$. This forms a smooth family of functions in $L^2_0$ which satisfy the equation
\begin{equation}
\left((D_x+\xi)^2 + V - \lambda_+(\xi)\right) p_+(\xi) = 0.
\end{equation}
Taking the derivative with respect to $\xi$ yields
\begin{equation}\label{eq:3na}
\left((D_x+\xi)^2 + V - \lambda_+(\xi)\right) \p_\xi p_+(\xi) = \p_\xi  \lambda_+(\xi) \cdot p_+(\xi)  -  2(D_x+\xi)p_+(\xi).
\end{equation}
Setting $\xi=\pi$ and taking the scalar product of \eqref{eq:3na} with $p_+(\pi)$, we get
\begin{equation}
\blr{\left((D_x+\pi)^2 + V - \lambda_+(\pi)\right) \p_\xi p_+(\pi),p_+(\pi)} = \p_\xi \lambda_+(\pi)  -  2 \blr{(D_x+\xi)p_+(\pi),p_+(\pi)}.
\end{equation}
Observe that the LHS vanishes: $(D_x+\pi)^2 + V - \lambda_+(\pi)$ is selfadjoint and $((D_x+\pi)^2 + V - \lambda_+(\pi))p_+(\pi) = 0$. This yields (4):
\begin{equation}
 \nu_\star =  2\blr{ p_+(\pi),(D_x+\pi)p_+(\pi)} = 2\blr{ \phi^\star_+,D_x\phi^\star_+}.
\end{equation}

5. It remains to prove the characterization property \eqref{eq:1g}. Assume that $(\mu_+(\xi),\psi_+(\xi))$ and $(\mu_-(\xi),\psi_-(\xi))$ are any normalized eigenpairs of $P_0(\xi)$ satisfying (1)--(4). Then $\psi_+(\pi) \in \ker(P_0(\pi)-E_\star)$. Because of \eqref{eq:1d}, there exists $\w \in \Ss^1$ such that $\psi_+(\pi) = \w \cdot \phi_+^\star$. Because of (2), $\psi_-(\pi) = \ow \phi_+^\star$. This concludes the proof.
\end{proof}

We call the pair $(\phi_+^\star,\phi_-^\star)$ a ``Dirac eigenbasis". According to the characterization of Proposition \ref{prop:1w}, they are unique modulo a multiplicative factor in $\Ss^1$. In particular, $\nu_\star$ is invariantly defined: if we change $(\phi_+^\star,\phi_-^\star)$ to $(\w \cdot \phi_+^\star, \ow \cdot \phi_-^\star)$, with $\w \in \Ss^1$, then   $\nu_\star$ in \eqref{pm-disp} is unchanged.

\subsection{The operators $P_{\pm\delta}$}

Recall that $P_0=D_x^2+V$, where $V\in L^2_{0,\ev}$: $V(x+1/2) = V(x)$; and $D_x = \frac{1}{i}\p_x$. Let $0 \not\equiv W \in L^2_{0,\od}$: $W(x+1/2) = -W(x)$. We introduce the perturbed periodic Schr\"odinger operator
\begin{equation}
P_\delta \de D_x^2+V + \delta W.
\end{equation}
For $\delta\ne0$, $[\SSS,P_\delta] = - 2\delta \cdot W \ne0$: the perturbation $\delta W$ breaks the $\SSS$-invariance. 

For the remainder of the paper, we assume that hypothesis (H3) holds: if $(\phi_+^\star, \phi_-^\star)$ is a Dirac eigenbasis for $(\pi,E_\star)$ then
\begin{equation}\label{eq:1h}
\var_\star \de \blr{\phi_+^\star,W\phi_-^\star} \neq 0.
\end{equation}
It should be noted that $\var_\star$ is \textit{not} invariantly defined. Indeed, a change of Dirac eigenbasis $(\phi_+^\star, \phi_-^\star)$ to $(\w \cdot \phi_+^\star, \ow \cdot \phi_-^\star)$ -- with $\w \in \Ss^1$ -- transforms $\var_\star$ to $\w^2 \var_\star$. But the condition  \eqref{eq:1h} is invariant under this change.

\begin{prop}\label{prop:1p} With the above notations,
\begin{equations}
2\matrice{\blr{\phi_+^\star,D_x\phi_+^\star} & \blr{\phi_-^\star,D_x \phi_+^\star} \\ \blr{\phi_+^\star,D_x \phi_-^\star} & \blr{\phi_-^\star,D_x\phi_-^\star}}  = \nu_\star  \bst,
\\
\matrice{\blr{\phi_+^\star,W \phi_+^\star} & \blr{\phi_-^\star,W \phi_+^\star} \\ \blr{\phi_+^\star,W \phi_-^\star} & \blr{\phi_-^\star,W \phi_-^\star}} = \matrice{0 & \ove{\var_\star} \ \\ \var_\star & 0} \de \bss = \Re(\var_\star) \bsu + \Im(\var_\star) \bsd
\end{equations}
where the Pauli matrices $\bsj$ are displayed in   \eqref{eq:1f}.
\end{prop}

\begin{rmk} In \cite{FLTW2}, the potentials $V$ and $W$ are assumed to be even (though the proof does not really require it). This implies  $\var_\star \in \R$. In general,  $\var_\star$ may be complex; for instance if $V \equiv 0$, $E_\star = \pi^2$ and $W(x) = 2\sin(2\pi x)$ then $\phi_+^\star(x) = e^{i\pi x}$, $\phi_+^\star(x) = e^{i\pi x}$ and $\var_\star = -i \notin \R$. \end{rmk}

\begin{proof} In order to prove this proposition, it suffices to derive the following identities:
\begin{equations}
2\blr{\phi_-^\star,D_x\phi_-^\star} = -\nu_\star, \ \ \ \ \blr{\phi_-^\star,D_x\phi_+^\star} = \blr{\phi_+^\star,D_x\phi_-^\star} = 0,
\\
\blr{\phi_-^\star,W\phi_+^\star} = \var_\star, \ \ \ \ \blr{\phi_+^\star,W\phi_+^\star} = \blr{\phi_-^\star,D_x\phi_-^\star} = 0.
\label{idens}\end{equations}
The first identity in \eqref{idens} follows from $\phi_-^\star = \overline{\phi_+^\star}$ (Proposition \ref{prop:1w})
 and $\ove{D_x} = -D_x$:
\begin{equation}
2\blr{\phi_-^\star,D_x\phi_-^\star} = -2\blr{\ove{\phi_+^\star}, \ove{D_x\phi_+^\star}} =
 -2\blr{D_x\phi_+^\star, \phi_+^\star} = -2\blr{\phi_+^\star, D_x\phi_+^\star} = -\nu_\star.
\end{equation}
The second identity is a consequence of $L^2_{\pi,\ev} \perp L^2_{\pi,\od}$, and that $D_x$ acts on both $L^2_{\pi,\ev}$ and $L^2_{\pi,\od}$. Indeed, these facts imply that $D_x \phi_+^\star$ is orthogonal to $\phi_-^\star$. 
The third identity follows from $\phi_-^\star = \ove{\phi_+^\star}$ and that $W = \oW$:
\begin{equation}
\blr{\phi_-^\star,W\phi_+^\star} = \blr{\ove{\phi_+^\star},\ove{W\phi_-^\star}} = \blr{W\phi_-^\star,\phi_+^\star} = \ove{\blr{\phi_+^\star,W\phi_-^\star}} = \ove{\var_\star}.
\end{equation}
Finally, the last identity is a consequence of $L^2_{\pi,\ev} \perp L^2_{\pi,\od}$ and $W \in L^2_{0,\od}$. These facts imply $\phi_+^\star \perp W \phi_+^\star$ and $\phi_-^\star \perp W \phi_-^\star$. This completes the proof. 
\end{proof}

 \section{From quasimodes to eigenmodes}\label{sec:4}
 
In this section we deduce Corollary \ref{cor:2}  from Theorem \ref{thm:2} by using:
\begin{itemize}
\item The formal multiscale expansion of \cite[\S 4]{FLTW2};
\item A general principle for gapped self-adjoint problems; see Lemma \ref{lem:1r}.
\end{itemize}

\subsection{The multiscale approach of \cite{FLTW1,FLTW2}.}

We review the formal multiscale analysis of \cite{FLTW1,FLTW2}. Recall that $X$ is the function space of two-scale functions:
\begin{equation}
X = \left\{ v \in C^\infty(\R^2,\C), \ \ v(x+1,y) = v(x,y), \ \ |v(x,y)| = O\left(e^{-|y|}\right) \right\}.
\end{equation}
For each eigenvalue $\var_j$ of $\Di$, the procedure of \cite{FLTW1,FLTW2}  constructs recursively:
\begin{itemize}
\item a sequence of numbers $a_m \in \R$ such that $a_0 = E_\star$ and $a_1 = \var_j$;
\item a sequence of two-scale functions $v_m(x,y) \in X$, starting with: 
\begin{equation}
v_0(x,y) = \az_+(y) \phi_+^\star(x) + \az_-(y) \phi^\star_-(x), \ \  (\Di-\var_j)\matrice{\az_+ \\ \az_-} = \matrice{0 \\ 0},
\end{equation}
\end{itemize}
such that  the following property holds:  for any positive integer $M$,  define
\begin{equations}
E_{\delta,M} \de \sum_{m=0}^M a_m \delta^m = E_\star+\var_j\delta+O\left(\delta^2\right), \\
 v_{\delta,M}(x) \de \delta^{1/2} \sum_{m=0}^M \delta^m u_m(x,\delta x)\quad \textrm{with}\quad
  |v_{\delta,M}|_{L^2}=1.\label{v-app}
\end{equations}
Then,  for any $k \in \N$, there exists a constant $C_{M,k}$ such that
\begin{equation}\label{eq:3v}
\delta \in (0,1) \ \Rightarrow \ \big|(\PP_\delta-E_{\delta,M}) v_{\delta,M}\big|_{H^k} \leq C_{M,k}  \delta^M.
\end{equation}
The energy $E_{\delta,M}$  is then an approximate eigenvalue of $\PP_\delta$ with corresponding  
approximate eigenstate or {\it quasimode} $v_{\delta,M}$.

\subsection{Abstract  quasimode theory} We discuss how
sufficiently accurate \textit{approximate} eigenpairs imply existence of \textit{genuine} eigenpairs.

\begin{lem}\label{lem:1r} Let $T$ be a selfadjoint operator on a Hilbert space $\HH$, with domain ${\rm dom}(T)$.
Assume: 
\begin{itemize}
\item There exist $E$ and  $\epsilon_0>0$ such that
$[E-\epsilon_0,E+\epsilon_0] \cap \Sigma_\ess(T)$ is empty.
\item There exist $v\in {\rm dom}(T)$
and $\epsilon \in [0,\epsilon_0)$ such that 
\begin{equations}
|v|_\HH = 1\ \ {\rm and}\  \ |(T-E)v|_\HH \leq \epsilon.
\label{q-mode}\end{equations}
\end{itemize}
Then  $T$ has a eigenvalue $\lambda$ with $|\lambda-E| \leq \epsilon$. 

 Furthermore, if there is a constant $C>0$ satisfying $C\epsilon \leq \epsilon_0$ and $T$ has no other eigenvalue in the interval $[E-C\epsilon,E+C\epsilon]$ then $T$ has an eigenvector $u$ with \[ |v-u|_\HH \leq C^{-1}.\]
\end{lem}

The function $v$ is called an ($\epsilon$-precise) quasimode. Lemma \ref{lem:1r} shows that a quasimode energy $E$ lying in an essential spectral gap of a selfadjoint operator must be $\epsilon$-close to an eigenvalue. We postpone the proof of this well-known lemma to Appendix \ref{sec:B}.

In the next section, we will apply this lemma with parameters related to the choice of quasimode: $\epsilon_0$ of order $\delta$ and $\epsilon$ of order $\delta^M$, where $M \geq 2$. In particular, $C$ will be of order $\delta^{1-M}$ and $C^{-1}$ is of order $\delta^{M-1}$ -- hence small.

\subsection{The proof of Corollary \ref{cor:2} assuming Theorem \ref{thm:2}.}

Let $\epsi$ and $\delta_0$ be as in Theorem \ref{thm:2}. Fix $M \geq 2$, $j \in [-N,N]$ and $k = 0$. Let $E_{\delta,M}$, $v_{\delta,M}$ and $C_{M,0} = C_M$ be given by \eqref{v-app} and \eqref{eq:3v}.
   
We begin by applying the first part of Lemma \ref{lem:1r} with the choices:
  \begin{equation} 
  T=\PP_\delta,\quad E = E_{\delta,M},\quad \epsilon_0 = \epsi \delta/2,\ \epsilon = C_M \delta^M\quad {\rm and}\quad v = v_{\delta,M}.\label{ch}
  \end{equation}
  
  By \eqref{eq:3v}, for $\delta \in (0,1)$, we have:
\begin{equation}
|(T-E)v|_{L^2} \leq \epsilon.
\end{equation}
This verifies the quasimode hypothesis \eqref{q-mode} of Lemma \ref{lem:1r}

We next verify that the choices \eqref{ch} ensure that the spectral gap assumption  of Lemma \ref{lem:1r} is satisfied.
Since $E_{\delta,M} = E_\star+\delta\var_j + O(\delta^2)$, for $\delta$ sufficiently small, 
\begin{equation}
[E-\epsilon_0,E+\epsilon_0] = \left[E_{\delta,M}-\dfrac{\epsi \delta}{2}, E_{\delta,M} + \dfrac{\epsi \delta}{2}\right]  \subset \big[E_\star + \delta (\var_j-\epsi), E_\star + \delta (\var_j+\epsi)\big].
\end{equation}
In particular, Theorem \ref{thm:2} implies that
\begin{equation}
\Sigma_\ess(\PP_\delta) \cap [E-\epsilon_0,E+\epsilon_0] = \emptyset.
\end{equation}
The first conclusion of Lemma \ref{lem:1r} applies and shows that for any $M$
\[\textrm{$\PP_\delta$ has an eigenvalue $E_{\delta,j}\in[E-\epsilon, E_{\delta,M}+\epsilon] = [E_{\delta,M}-C_M \delta^M, E_{\delta,M}+C_M\delta^M]$ .}\]
 Now Theorem \ref{thm:2} asserts that 
 \[\textrm{$\PP_\delta$ has exactly one eigenvalue $E_{\delta,j}$ in $[E_{\delta,M}-C_M \delta^M, E_{\delta,M}+C_M \delta^M]$.}\]
  Since $M\ge1$ is arbitrary, this proves that $E_{\delta,j}$ has an expansion to arbitrary finite order in the small parameter $\delta$. This proves the first line of the expansion \eqref{eq:3r} in Corollary \ref{cor:2}.

We now turn to arbitrary finite order expansions of the eigenvector corresponding to $E_{\delta,j}$. We keep the choices \eqref{ch} made above. Let $C = \epsi/(2C_M\delta^{M-1}) = \epsi\delta/(2\epsilon)$. Since $E_{\delta,M} = E_\star+\delta\var_j + O(\delta^2)$,
\begin{equation}
[E-C\epsilon,E + C\epsilon] \subset \big[E_\star + \delta (\var_j-\epsi), E_\star + \delta (\var_j+\epsi)\big],
\end{equation}
as long as $\delta$ is sufficiently small. By Theorem \ref{thm:2}, $\PP_\delta$ has no other eigenvalues in the set $[E-C\epsilon,E + C\epsilon]$. The additional assumption of Lemma \ref{lem:1r} holds. Thus this lemma produces $u_{\delta,j}$ that satisfies
\begin{equation}\label{eq:1u}
\big|v_{\delta,M} - u_{\delta,j}\big|_{L^2} \leq  C^{-1} = \dfrac{2C_M \delta^{M-1}}{\epsi}.
\end{equation}
Since $M$ was arbitrary, we can replace $M$ by $M+1$ in \eqref{eq:1u} and conclude the proof of Corollary \ref{cor:2} for the case $k=0$.

In order to complete the proof of Corollary \ref{cor:2}, we must show that the $L^2$ norm in \eqref{eq:1u} can be replaced to a $H^k$-norm. Specifically, we show by induction on $k$ that for any $k\ge0$, there exists $C_{M,k}$ with
\begin{equation}\label{eq:1z}
|v_{\delta,M} - u_{\delta,j}|_{H^k} \leq  C_{k,M} \delta^M.
\end{equation}
Without loss of generality, we carry out the induction on $k$  even.  We have verified the case $k=0$ above. Assume now that \eqref{eq:1z} holds for some value of $k$ and let us show that it also holds for $k+2$. The main point is the elliptic estimate;  for any integer $s$ if $f \in H^s$ and $\PP_\delta f \in H^s$, then $f \in H^{s+2}$ and there is a  constant $A_s$ such that 
\begin{equation}
|f|_{H^{s+2}} \leq A_s \big( |f|_{H^s} + \left|\PP_\delta f\right|_{H^s} \big).
\end{equation}
The constant $A_s$ does not depend on $\delta \in (0,1)$ because the derivatives of the coefficients of the potential $V+\delta \kappa(\delta \ \cdot) W$ arising in $\PP_\delta$ are uniformly bounded in $\delta \in (0,1)$. Applying this estimate to $f = v_{\delta,M} - u_{\delta,j}$ with $s=k$  yields
\begin{equations}\label{eq:1x}
|v_{\delta,M} - u_{\delta,j}|_{H^{k+2}} \leq A \big( |v_{\delta,M} - u_{\delta,j}|_{H^k} + \left|\PP_\delta(v_{\delta,M} - u_{\delta,j})\right|_{H^k} \big)
\\
\leq A_k \big( |v_{\delta,M} - u_{\delta,j}|_{H^k} + \left|(\PP_\delta-E_{\delta,M})v_{\delta,M}\right|_{H^k} + |E_{\delta,M} - E_{\delta,j}| \cdot |u_{\delta,j}|_{H^k} \big).
\end{equations}
We used $(\PP_\delta  - E_\delta,j)u_{\delta,j}= 0$ above.
Applying and iterating the elliptic estimate for $u_{\delta,j}$ shows that there exists $B_k$ independent of $\delta$ such that $|u_{\delta,j}|_{H^k} \leq B_k$. Using this estimate together with the recursion hypothesis, \eqref{eq:3v}, and $|E_{\delta,M}-E_{\delta,j}| \leq C_M \delta^M$, we obtain from \eqref{eq:1x} that
\begin{equation}
|v_{\delta,M} - u_{\delta,j}|_{H^{k+2}} \leq A_k \cdot \big( C_{M,k} \delta^M + C_M\delta^M + B_k C_M \delta^M\big).
\end{equation}
This completes the induction and therewith the proof of Corollary \ref{cor:1}. 

\section{Resolvent of $P_\delta$ for small $\delta$ }

Recall that $V \in C^\infty(\R,\R) \cap L^2_{0,\ev}$ and $W \in C^\infty(\R,\R) \cap L^2_{0,\od}$. Let $j_\star \geq 1$ and consider the corresponding Dirac point $(\pi,E_\star) = (\pi,\lambda_{0,j_\star}(\pi))$ of $P_0 = D_x^2 + V$. Proposition \ref{prop:1w} yields a Dirac eigenbasis $(\phi_+^\star, \phi_-^\star)$; we set
\begin{equation}\label{eq:1l}
\nu_\star = 2\blr{\phi_+^\star,D_x\phi_+^\star}, \ \ \ \ \var_\star = \blr{\phi_+^\star,W\phi_-^\star}, \ \ \ \ \bss = \matrice{0 & \ove{\var_\star} \ \\ \var_\star & 0}.
\end{equation}
As mentioned above, we assume that $\var_\star \neq 0$. The operator $P_\delta=D_x^2+V+\delta W$ has a $L^2$-spectral gap near $E_\star$, of width $2|\var_\star|\delta + O(\delta^2)$. The main result of this section is an estimate on the resolvent of $P_\delta$ when the spectral parameter spans a slightly smaller gap. We introduce the operator $\Di_+$, which is the asymptotic limit of $\Di$ as $x \rightarrow + \infty$:
\begin{equation}
\Di_+ \de \nu_\star \bst D_x + \bss.
\end{equation}
We will also need the dilation operator $\UU_\delta$ defined by $\UU_\delta f(x) = \delta^{1/2} f(\delta x)$.

\begin{theorem}\label{thm:1} Let $\tvar \in (0,|\var_\star|)$. There exists $\delta_0 > 0$ such that for $\delta \in (0,\delta_0)$ and 
$\lambda \in \Dd(E_\star,\tvar \delta)$, the operator $P_\delta-\lambda$ is invertible $H^2 \rightarrow L^2$ and
\begin{equations}
(P_\delta-\lambda)^{-1} = S_\delta\left(\dfrac{\lambda-E_\star}{\delta}\right) +  \OO_{L^2}\left(\delta^{-1/3}\right), \ \ \ \
D_x (P_\delta-\lambda)^{-1} = S_\delta\left(\dfrac{\lambda-E_\star}{\delta}\right) +  \OO_{L^2}\left(\delta^{-1/3}\right),
\end{equations}
where:
\begin{align}
S_\delta(z) &= \dfrac{1}{\delta} \matrice{ \phi_+^\star \\ \phi_-^\star  }^\top \cdot    \UU_\delta\ \big( \Di_+- z\big)^{-1}\   \UU_\delta^{-1}  \cdot \ove{\matrice{ \phi_+^\star \\ \phi_-^\star  }}, \label{Sd}\\
S_\delta^D(z) &= \dfrac{1}{\delta} \matrice{ D_x\phi_+^\star \\ D_x\phi_-^\star  }^\top \cdot    \UU_\delta\  \big( \Di_+ - z\big)^{-1} \   \UU_\delta^{-1}\  \cdot \ove{\matrice{ \phi_+^\star \\ \phi_-^\star  }}.
\label{DSd}\end{align}
\end{theorem} 

The proof of Theorem \ref{thm:1} occupies the remainder of the section. We prove it by analyzing the resolvent $\left(P_\delta-\lambda\right)^{-1} : L^2 \rightarrow H^2$ fiberwise: 
\begin{equation}\label{eq:1a}
(P_\delta-\lambda)^{-1}= \frac{1}{2\pi}\ \int^\oplus_{[0,2\pi]}\ d\xi \left(P_\delta(\xi)-\lambda\right)^{-1}.
\end{equation}
In light of \eqref{eq:1a}, the analysis of $(P_\delta-\lambda)^{-1}$  reduces to
\begin{itemize}
\item Estimates on the fiber resolvents $(P_\delta(\xi)-\lambda)^{-1}$ -- \S\ref{sec:6.2}-\ref{sec:6.3};
\item A procedure to integrate these estimates over $[0,2\pi]$ -- \S\ref{sec:6.4}.
\end{itemize}

In \S\ref{sec:6.2}, we show that momenta away from $\pi$ do not contribute to the leading terms in \eqref{eq:1a}. In other words, the dominant contributions in $(P_\delta-\lambda)^{-1}$ must arise from the Dirac point $(\pi,E_\star)$. In \S\ref{sec:6.3}, we study these contributions specifically.

\subsection{Estimates on  $(P_\delta(\xi) - \lambda)^{-1} $ for quasi-momenta away from $\pi$}
\label{sec:6.2}

Let $\Sigma_{L^2_\xi}(P_\delta)$ denote the spectrum of  $P_\delta$ (acting on $L^2_\xi$). This spectrum consists of eigenvalues 
\begin{equation}
\lambda_{\delta,1}(\xi) \leq \cdots \leq \lambda_{\delta,j}(\xi) \leq \cdots, 
\end{equation}
listed with multiplicity. Recall that $(\pi,E_\star)$ is Dirac point of $P_0$ -- see Definition \ref{def:1}.

\begin{prop}\label{prop:1l} Let $\tvar \in (0,|\var_\star|)$. There exist $C > 0$ and $\delta_0 > 0$ such that
\begin{equations}
\delta \in (0,\delta_0), \ \ \lambda \in \Dd(E_\star,\tvar\delta), \ \ |\xi-\pi| \geq \delta^{1/3} \ \  
 \Rightarrow \ \ \left\| (P_\delta(\xi) - \lambda)^{-1}  \right \|_{L^2_\xi \rightarrow H^2_\xi} \leq C \delta^{-1/3}.
\end{equations}
\end{prop}

\begin{proof} 1. Thanks to the spectral theorem,
\begin{equation}\label{eq:2b}
\left \|  (P_\delta(\xi) - \lambda)^{-1} \right \|_{L^2_\xi} = \dfrac{1}{\dist(\Sigma_{L^2_\xi}(P_\delta), \lambda)}.
\end{equation}
Therefore a $L^2_\xi \rightarrow L^2_\xi$ bound on $(P_\delta(\xi) - \lambda)^{-1}$ follows from a lower bound on the distance between $\lambda$  and the eigenvalues of $P_\delta(\xi)$. 

Let $j_\star\ge0$  the index for which $E_\star = \lambda_{0,j_\star}(\pi)= \lambda_{0,j_\star+1}(\pi)$. We recall that 
\begin{equations}
\lambda_{0,j_\star}(\xi) = E_\star - |\nu_\star| \cdot |\xi-\pi| + O(\xi-\pi)^2,
\\
\lambda_{0,j_\star+1}(\xi) = E_\star + |\nu_\star| \cdot |\xi-\pi| + O(\xi-\pi)^2.
\end{equations}
Because of the monotonicity properties of the dispersion curves -- see e.g. \cite[Theorem XII.89]{RS} -- the function $\xi \in [0,\pi] \mapsto \lambda_{0,j_\star}(\xi)$ is increasing and the function $\xi \in [\pi,2\pi] \mapsto \lambda_{0,j_\star}(\xi)$ is decreasing;
 see Figure \ref{fig:1}.
   Assuming that $|\xi-\pi| \geq \delta^{1/3}$,
\begin{equation}\label{eq:1y}
\lambda_{0,j_\star}(\xi) \leq E_\star - |\nu_\star| \cdot \delta^{1/3} + O\left(\delta^{2/3}\right).
\end{equation}
Similarly, using the monotonicity properties of $\xi  \mapsto \lambda_{0,j_\star+1}(\xi)$, we have 
\begin{equation}\label{eq:2a}
\lambda_{0,j_\star+1}(\xi) \geq E_\star + |\nu_\star| \cdot  \delta^{1/3} + O\left(\delta^{2/3}\right).
\end{equation}

Applying the stability estimate Lemma \ref{lem:1m} with $\VV = V$ and $\WW = \delta W$ we have that if $\delta$ is sufficiently small, then
\begin{equation}
|\lambda_{\delta,j}(\xi)-\lambda_{0,j}(\xi)| \leq C(j)  |W|_\infty \cdot \delta, \ \ \ \ j\in \{j_\star, j_\star+1\}.
\end{equation}
Fix $c \in (0,|\nu_\star|)$.
We deduce from \eqref{eq:1y} and \eqref{eq:2a} that for $\delta_0$ sufficiently small,
\begin{equations}
\delta \in (0,\delta_0), \ \ |\xi-\pi| \geq \delta^{1/3}  \ \ \Rightarrow \ \ \lambda_{\delta,j_\star}(\xi) \leq E_\star - c   \delta^{1/3} \text{ and }
\lambda_{\delta,j_\star+1}(\xi) \geq E_\star + c\delta^{1/3}.
\end{equations}
Because the dispersion curves are indexed in increasing order, this extends to all other dispersion curves: for $\delta \in (0,\delta_0)$ and $|\xi - \pi| \geq \delta^{1/3}$,
\begin{equations}
j \leq j_\star \ \Rightarrow \ \lambda_{\delta,j}(\xi) \leq E_\star - c\delta^{1/3}; \ \ \ \ 
j \geq j_\star+1 \ \Rightarrow \ \lambda_{\delta,j}(\xi) \geq E_\star + c \delta^{1/3}.
\end{equations}
This yields that for $|\xi-\pi|\ge\delta^{1/2}$, $\dist(\Sigma_{L^2_\xi}(P_\delta), E_\star) \geq c\delta^{1/3}$. If $\lambda \in \Dd(E_\star,\tvar\delta)$, we deduce $\dist(\Sigma_{L^2_\xi}(P_\delta), \lambda) \geq c\delta^{1/3} - \tvar \delta$. Shrinking $\delta_0$ if necessary, we get thanks to \eqref{eq:2b}:
\begin{equation}\label{eq:2c}
\delta \in (0,\delta_0), \ \ |\lambda-E_\star|<\var_\sharp\delta, \ \ |\xi-\pi|\ge\delta^{1/3} \ \ \Rightarrow  \ \ \left \|  (P_\delta(\xi) - \lambda)^{-1} \right \|_{L^2_\xi} \leq C \delta^{-1/3}.
\end{equation}

2. It remains to improve \eqref{eq:2c} by replacing the $L^2_\xi$-operator norm by the $H^2_\xi$-operator norm. We use an elliptic estimate, based on the inequality
\begin{equation}\label{eq:11h}
|f|_{H^2_\xi} \leq |f|_{L^2_\xi} + \left|D_x^2 f\right|_{L^2_\xi} \leq \big(1+|V|_\infty+\delta|W |_\infty + |\lambda|\big) \cdot|f|_{L^2_\xi} + |(P_\delta-\lambda) f|_{L^2_\xi}.
\end{equation}
Since $\lambda$ and $\delta$ vary within compact sets, we have for $|\xi-\pi|\ge\delta^{1/3}$ and $|\lambda-E|<\var_\sharp\delta$:
\begin{equation}
\left\|(P_\delta-\lambda)^{-1}\right\|_{L^2_\xi \rightarrow H^2_\xi} \leq C\left\|(P_\delta-\lambda)^{-1}\right\|_{L^2_\xi} + 1 \leq 2C \delta^{-1/3}.
\end{equation}
This completes the proof.\end{proof}

\subsection{Estimates on  $(P_\delta(\xi) - \lambda)^{-1} $ for quasi-momenta near $\pi$}\label{sec:6.3}

We now analyze $(P_\delta(\xi) - \lambda)^{-1} $ when $|\xi-\pi| \leq \delta^{1/3}$ and $\lambda\in \Dd(E_\star,\tvar\delta)$, with $0<\tvar<|\var_\star|$. Recalling that $(\pi,E_\star)$ is a Dirac point of $P_0$, we associate eigenpairs $(\lambda_+(\xi),\phi_+(\xi))$ and $(\lambda_-(\xi),\phi_-(\xi))$ guaranteed by Proposition \ref{prop:1w}.

Define $\Pi_0(\xi) : L^2_\xi \rightarrow \C^2$ by
\begin{equation}
\Pi_0(\xi) v \de \matrice{ \blr{\phi_+(\xi),v} \\  \blr{\phi_-(\xi),v}}.
\end{equation}
We introduce the matrix 
\begin{equation}\label{eq:1n}
M_\delta(\xi) \de \matrice{E_\star + \nu_\star(\xi-\pi) & \delta \ove{\var_\star} \ \\ \delta \var_\star & E_\star-\nu_\star(\xi-\pi)} = E_\star + \nu_\star(\xi-\pi) \bst + \delta \bss,
\end{equation}
where $\nu_\star$ and $\bss$ are displayed in \eqref{eq:1l}.
The eigenvalues of the matrix $M_\delta(\xi)$ are
\begin{equation}
E_\star \pm \sqrt{|\var_\star|^2\delta^2 + \nu_\star^2(\xi-\pi)^2}.
\end{equation}
Therefore, $M_\delta(\xi) - \lambda$ is invertible for every $\lambda \in \Dd(E_\star,\tvar\delta)$. Indeed, since $|\lambda-E_\star|<\tvar\delta$, its eigenvalues are   $E_\star\pm \sqrt{ |\var_\star|^2\delta^2 + \nu_\star^2(\xi-\pi)^2 } - \lambda$
and they satisfy
\begin{equations}
\left|\ E_\star \pm \sqrt{ |\var_\star|^2\delta^2 + \nu_\star^2(\xi-\pi)^2 } - \lambda \right| \geq\sqrt{ |\var_\star|^2\delta^2 + \nu_\star^2(\xi-\pi)^2 } - |E_\star-\lambda| 
\\
 > \left(| \var_\star|-\tvar\right)\delta.
\end{equations}
The next proposition shows that for $\delta$ sufficiently small and $(\xi,\lambda)$ near $(\pi,E_\star)$, the operator $P_{\delta}(\xi)-\lambda$ behaves like a $2\times 2$ matrix $M_\delta(\xi)$. 

\begin{prop}\label{prop:1b} Let $\tvar \in (0,|\var_\star|)$. There exists $\delta_0 > 0$ such that
\begin{equations}
\delta \in (0,\delta_0), \ \ |\xi-\pi| \leq 2\delta^{1/3}, \ \ \lambda \in \Dd\left(E_\star,\tvar\delta\right) \ \ \Rightarrow \ \ P_\delta(\xi)-\lambda \text{ is invertible; and }
\end{equations}
\begin{equation}\label{eq:1w}
(P_\delta(\xi)-\lambda)^{-1} = \Pi_0(\xi)^* \cdot (M_\delta(\xi)-\lambda)^{-1} \cdot \Pi_0(\xi) + \OO_{L^2_\xi \rightarrow H^2_\xi}(1).
\end{equation}
\end{prop}

\begin{proof} 1. Let $V(\xi) = \C \phi_+(\xi) \oplus \C \phi_-(\xi) \subset L^2_\xi$.
 We write $P_\delta(\xi)-\lambda$ as a block matrix corresponding to the decomposition $L^2_\xi = V(\xi) \oplus V(\xi)^\perp$:
\begin{equation}\label{eq:1m}
P_\delta(\xi) - \lambda =   \matrice{A_\delta(\xi,\lambda) & B_\delta(\xi) \\ C_\delta(\xi) & D_\delta(\xi,\lambda)}.
\end{equation}
Above, $A_\delta(\xi,\lambda) : V(\xi)\to V(\xi)$,\ $B_\delta(\xi) : V(\xi)^\perp\to V(\xi)$,\  $C_\delta(\xi) : V(\xi)\to V(\xi)^\perp$
 and $D_\delta(\xi,\lambda) : V(\xi)^\perp\to V(\xi)^\perp$. 
 
We will use the  Schur complement Lemma   to  prove that $P_\delta(\xi)-\lambda$ is invertible and to get a useful formula for the inverse.

\begin{lem}\label{schur}
Suppose given an operator in block form
\begin{equation}
\mathscr{M} \de \matrice{A&B\\ C&D } 
\end{equation}
such that the operators $A$ and $E=D - CA^{-1}B$ are invertible (in their appropriate spaces).
Then $\mathscr{M}$ is invertible and its inverse is given by
\begin{equation}
\mathscr{M}^{-1}=
 \matrice{  E^{-1} & -E^{-1}BA^{-1} \\
 - A^{-1}CE^{-1} & A^{-1} + A^{-1}CE^{-1}BA^{-1}}.
 \end{equation}
\end{lem}
\noindent From \eqref{eq:1m}, we conclude that $P_\delta(\xi)-\lambda$ is invertible if under the conditions of \eqref{eq:1w}:
\begin{equations}\label{s-conds}
A_\delta(\xi,\lambda) : \ V(\xi) \rightarrow V(\xi) \ \text{ is invertible; } \\
E_\delta(\xi,\lambda) \de D_\delta(\xi,\lambda) - C_\delta(\xi)A_\delta(\xi,\lambda)^{-1}B_\delta(\xi) : V(\xi)^\perp \rightarrow V(\xi)^\perp \ \text{ is invertible}.
\end{equations}
We next verify these conditions in Steps 2 and 3.

2. Let $J_0(\xi) : V(\xi) \rightarrow \C^2$ be the coordinate map in the basis $\{\phi_+(\xi),\phi_-(\xi)\}$. In other words, $J_0(\xi)$ is the restriction of $\Pi_0(\xi)$ to $L^2_\xi$. Thus, $J_0(\xi)^*$ maps a pair $(\alpha,\beta)\in\C^2$ to $\alpha\phi_+(\xi)+\beta\phi_-(\xi)\in V(\xi)$ and we have
\begin{equation}
A_\delta(\xi,\lambda)= J_0(\xi)^* (P_\delta(\xi)-\lambda)J_0(\xi)= J_0(\xi)^* \matrice{\lambda_+(\xi)-\lambda & \delta \blr{\phi_-(\xi),W \phi_+(\xi)} \\  \delta \blr{\phi_+(\xi),W \phi_-(\xi)} & \lambda_-(\xi)-\lambda } J_0(\xi).
\end{equation}
We show that $A_\delta(\xi,\lambda)$ is invertible. We recall that $\lambda_{\pm}(\xi) = E_\star \pm \nu_\star(\xi-\pi) + O(\xi-\pi)^2$ and we observe that by Propositions \ref{prop:1w},
\begin{equation}
\delta \blr{\phi_+(\xi),W \phi_-(\xi)} = \delta \blr{\phi_+^\star,W \phi_-^\star} + O\big(\delta(\xi-\pi)\big) = \delta \var_\star + O\left(\delta^2+(\xi-\pi)^2\right).
\end{equation}
Similarly, $\delta \blr{\phi_+(\xi),W \phi_-(\xi)} = \delta \ove{\var_\star} + O\left(\delta^2+(\xi-\pi)^2\right)$.
Hence
\begin{equations}
A_\delta(\xi,\lambda) = J_0(\xi)^* (M_\delta(\xi)-\lambda)  J_0(\xi) + \OO_{V(\xi)}\left((\xi-\pi)^2 + \delta^2\right)
\end{equations}
where $M_\delta(\xi)$ is displayed in \eqref{eq:1n}. 

In order to prove that $A_\delta(\xi,\lambda)$ is invertible when the conditions on $\xi$ and $\lambda$ in  \eqref{eq:1w} are satisfied, we first observe that
\begin{equation}\label{eq:11j}
\left\|(M_\delta(\xi)-\lambda)^{-1}\right\|_{\C^2} = O\left( (\delta+|\xi-\pi|)^{-1} \right).
\end{equation}
The matrix $M_\delta(\xi)$ is hermitian. Hence, a bound on $(M_\delta(\xi)-\lambda)^{-1}$ follows from an estimate on the distance between $\lambda$ and the spectrum of $M_\delta(\xi)$. The eigenvalues of $M_\delta(\xi)-\lambda$ are  $E_\star \pm \sqrt{\nu_\star^2(\xi-\pi)^2 + \delta^2 |\var_\star|^2} - \lambda$. Using that for all $t\in(0,1)$, $\sqrt{a^2+b^2}\ge t|a|+(1-t)|b|$,  we have :
\begin{equations}
\left| E_\star \pm \sqrt{\nu_\star^2(\xi-\pi)^2 + \delta^2 |\var_\star|^2} - \lambda \right| 
\geq \sqrt{\nu_\star^2(\xi-\pi)^2 +  |\var_\star|^2\delta^2} - |\lambda-E_\star|\\
\geq (1-t) \nu_\star |\xi-\pi|+ t|\var_\star|\delta\ -\ \tvar\delta 
\geq (1-t) \nu_\star |\xi-\pi|+ |\var_\star|\left(t-\frac{\tvar}{|\var_\star|}\right) \delta.
\end{equations}
Choosing $t$ to satisfy $\frac{\tvar}{|\var_\star|}<t<1$, we have that under the conditions of \eqref{eq:1w}, $M_\delta(\xi)-\lambda$ is invertible with the bound \eqref{eq:11j}.

Because of \eqref{eq:11j} and a perturbative argument (based on a Neumann series), under the conditions of \eqref{eq:1w} $A_\delta(\xi,\lambda)$ is invertible and 
\begin{equations}
A_\delta(\xi,\lambda)^{-1} = J_0(\xi)^* (M_\delta(\xi)-\lambda)^{-1}  J_0(\xi) + \OO_{V(\xi)}\left(\dfrac{(\xi-\pi)^2 + \delta^2}{(\delta + |\xi-\pi|)^2}\right) 
\\
 = J_0(\xi)^* (M_\delta(\xi)-\lambda)^{-1}  J_0(\xi) + \OO_{V(\xi)}(1).
\end{equations}
In particular $A_\delta(\xi,\lambda)^{-1}  = \OO_{V(\xi)}\left((\delta+|\xi-\pi|)^{-1}\right)$.

3.  The operator $B_\delta(\xi) : V(\xi)^\perp \rightarrow V(\xi)$ is the composition of (i) the embedding $V(\xi)^\perp \rightarrow L^2_\xi$;  (ii) the operator $P_\delta = P_0 + \delta W$; and  (iii) the projection $L^2_\xi \rightarrow V(\xi)$. Note that $B_0(\xi) = 0$. This proves that $B_\delta(\xi)$ is more simply the composition of (i); the multiplication operator by $\delta W$; and (iii). In particular $B_\delta(\xi) = \OO_{V(\xi)^\perp \rightarrow V(\xi)}(\delta)$. Since $C_\delta(\xi) = B_\delta(\xi)^*$, we deduce have $C_\delta(\xi) = \OO_{V(\xi) \rightarrow V(\xi)^\perp}(\delta)$. 

 The bounds obtained in Step 2 imply that
\begin{equations}\label{eq:1k}
E(\xi,\lambda)=D_\delta(\xi,\lambda) - C_\delta(\xi)A_\delta(\xi,\lambda)^{-1}B_\delta(\xi)  = D_\delta(\xi,\lambda) + \OO_{{V(\xi)} }(\delta)
\\
= (P_0(\xi)-\lambda)\big|_{V(\xi)^\perp} + \OO_{V(\xi)^\perp}(\delta),
\end{equations}
where  $(P_0(\xi)-\lambda)\big|_{V(\xi)^\perp}$ acts $V(\xi)^\perp \rightarrow V(\xi)^\perp$. In particular, $(P_0(\xi)-\lambda)\big|_{V(\xi)^\perp}$ has spectrum equal to $\Sigma_{L^2_\xi}(P_0) \setminus \{ \lambda_{0,j_\star}(\xi), \lambda_{0,j_\star+1}(\xi)\}$. 

Because of the monotonicity properties of dispersion curves \cite[Theorem XIII.89]{RS}, for every $\xi \in [0,2\pi]$,
\begin{equation}
\lambda_{0,j_\star+2}(\xi) \geq \lambda_{0,j_\star+2}(0) \geq \lambda_{0,j_\star+1}(0) > \lambda_{0,j_\star+1}(\pi) = E_\star.
\end{equation}
Hence, $\lambda_{0,j_\star+2}(\xi)>E_\star$ and similarly we can show that  $\lambda_{0,j_\star-1}(\xi) < E_\star$. It follows that for a constant $c > 0$ and for any $j \notin \{j_\star,j_\star+1\}$, $|\lambda_{0,j}(\xi) -E_\star| \geq c$. Because $\lambda \in \Dd(E_\star,\tvar \delta)$, we deduce that
\begin{equation}
j \notin \{j_\star,j_\star+1\} \ \Rightarrow \ |\lambda_{0,j}(\xi) -\lambda| \geq c - O(\delta).
\end{equation}

Therefore $\left((P_0(\xi)-\lambda)\big|_{V(\xi)^\perp}\right)^{-1}$ exists and has norm $\OO_{V(\xi)^\perp}(1)$. From \eqref{eq:1k}, a  Neumann series argument shows that $E_\delta(\xi,\lambda)$ is invertible with 
\begin{equation}
E_\delta(\xi,\lambda)^{-1}=\left((P_0(\xi)-\lambda)\big|_{V(\xi)^\perp}\right)^{-1} + \OO_{V(\xi)^\perp}(\delta)=\OO_{{V(\xi)^\perp\to V(\xi)^\perp}}(1),
\end{equation}
as long as the conditions on $\xi$ and $\lambda$ of \eqref{eq:1w} are satisfied.

4.  By Lemma \ref{schur} and the previous bounds, when the conditions of \eqref{eq:1w} are fulfilled, the inverse of $P_\delta(\xi)-\lambda$ exists and satisfies
\begin{equations}\label{eq:11x}
(P_\delta(\xi)-\lambda)^{-1} =   \matrice{A_\delta(\xi,\lambda)^{-1} & 0 \\ 0 & 0}   + \OO_{L^2_\xi}(1)\\
 = \Pi_0(\xi)^* \cdot (M_\delta(\xi)-\lambda)^{-1} \cdot \Pi_0(\xi)+ \OO_{L^2_\xi}(1).
 \end{equations}

5. Thanks to the elliptic estimate \eqref{eq:11h}, the expansion \eqref{eq:11x} can be improved to one with an error term which is $\OO_{L^2_\xi \rightarrow H^2_\xi}(1)$. The proof of Proposition \ref{prop:1b} is complete.
\end{proof}

\subsection{Expansion of $(P_\delta-\lambda)^{-1}$, near $\lambda=E_\star$}\label{sec:6.4} By the Floquet--Bloch reconstruction  formula \eqref{eq:2j} we have:
\[(P_\delta-\lambda)^{-1} = \frac{1}{2\pi}\ \int_{[0,2\pi]}^\oplus d\xi\ ( P_\delta(\xi)-\lambda)^{-1}.
\]
Proposition \ref{prop:1b} implies that the dominant contributions to $(P_\delta(\xi)-\lambda)^{-1}$ are encoded in the RHS of \eqref{eq:1w}. Thus we separately consider contributions coming from  quasi-momenta near and away
 from $\xi=\pi$. Consider a function  $\chi \in C_0^\infty(\R)$ such that 
\begin{equation}
\supp(\chi) \subset (-2,2), \ \ \ \ \chi(x) = 1 \text{ for } x \in [-1,1]. 
\end{equation}
Then, 
\begin{align}
(P_\delta-\lambda)^{-1} & = \frac{1}{2\pi}\int_{[0,2\pi]}^\oplus d\xi\ \chi\left( \dfrac{\xi-\pi}{\delta^{1/3}} \right)\cdot \left( P_\delta(\xi)-\lambda \right)^{-1}\\ 
 & + \frac{1}{2\pi} \int_{[0,2\pi]}^\oplus d\xi\ \left(1-\chi\left( \dfrac{\xi-\pi}{\delta^{1/3}} \right)\right) \cdot \left( P_\delta(\xi)-\lambda\right)^{-1}.
\end{align}
By Proposition \ref{prop:1l}, the second integral satisfies the bound:
\begin{equation}
\left\|\ \frac{1}{2\pi}  \int_{[0,2\pi]}^\oplus d\xi\ \left(1-\chi\left( \dfrac{\xi-\pi}{\delta^{1/3}} \right)\right) \cdot (P_\delta(\xi)-\lambda )^{-1}\ \right\|_{{L^2_\xi\to H^2_\xi}} \leq C \delta^{-1/3}.
\end{equation}
By Proposition \ref{prop:1b}, the first integral may be expanded as:
\[\frac{1}{2\pi} \int_{[0,2\pi]}^\oplus d\xi\ \chi\left( \dfrac{\xi-\pi}{\delta^{1/3}} \right)\cdot \left( P_\delta(\xi)-\lambda \right)^{-1}
= R_{\delta}(\lambda) 
+ \OO_{{L^2_\xi\to H^2_\xi}}(1), \]
where $R_{\delta}(\lambda)$ is the operator acting on $L^2$  given by
\begin{equation}
R_{\delta}(\lambda) \de \dfrac{1}{2\pi}\int_{[0,2\pi]}^\oplus d\xi\ \chi\left( \dfrac{\xi-\pi}{\delta^{1/3}} \right) \cdot \Pi_0(\xi)^* \big(M_\delta(\xi)-\lambda \big)^{-1} \Pi_0(\xi).
\end{equation}
Therefore,
\begin{equation}\label{eq:2m}
(P_\delta-\lambda)^{-1}= R_{\delta}(\lambda) +  \OO_{{L^2_\xi\to H^2_\xi}}(\delta^{-1/3}).
\end{equation}
Theorem \ref{thm:1} follows from the following result:

\begin{prop} \label{prop:1}  Let $0 < \tvar < |\var_\star|$. There exists $\delta_0 > 0$ such that if
$
\delta \in (0,\delta_0)$ and $z \in \Dd(0,\tvar)$, 
then the operator $R_\delta(\lambda)$ satisfies
\begin{equation}\label{R_d}
R_\delta(E_\star+z \delta) = S_\delta(z)+  
 \OO_{L^2} (\delta^{-1/3})
\end{equation}
\begin{equation}\label{dR_d}
D_x R_\delta(E_\star+z \delta) = S_\delta^D(z)+   \OO_{L^2} (\delta^{-1/3}),
\end{equation}
where the operators $S_\delta(z)$ and $S_\delta^D(z)$ are displayed in \eqref{Sd} and \eqref{DSd}.
\end{prop}

\begin{proof}[Proof of Proposition \ref{prop:1}] 
\noindent {\bf Step 1.}  We begin with the proof of \eqref{R_d}. The proof of \eqref{dR_d} is very similar; we explain it in Step 4.
 The operator $R_\delta(\lambda)$ has kernel given by
\begin{equation}\label{eq:3f}
R_\delta(x,y;\lambda) \de \dfrac{1}{2\pi}\int_\R  \matrice{ \phi_+(\xi,x) \\ \phi_-(\xi,x)  }^\top \cdot \chi\left( \dfrac{\xi-\pi}{\delta^{1/3}} \right) \big( M_\delta(\xi)-\lambda \big)^{-1} \cdot \ove{\matrice{ \phi_+(\xi,y) \\  \phi_-(\xi,y)}} d\xi.
\end{equation}
Making the change of variables  $\xi \mapsto\pi+\delta\xi $ and hence $d\xi \mapsto \delta d\xi$ in  
\eqref{eq:3f} shows that the kernel $R_\delta(x,y;\lambda)$  equals
\begin{equation}\label{eq:3o}
 \dfrac{\delta}{2\pi}\int_\R \chi\left( \delta^{2/3} \xi \right) \matrice{ \phi_+(\pi+\delta\xi,x) \\ \phi_-(\pi+\delta \xi,x)  }^\top  \cdot \big( M_\delta(\pi+\delta\xi)-\lambda \big)^{-1}  \cdot \ove{\matrice{ \phi_+(\pi+\delta\xi,y) \\  \phi_-(\pi+\delta\xi,y)}} d\xi.
\end{equation}

Introduce the rescaled energy parameter $z = \frac{\lambda-E_\star}{\delta} \in \Dd(0,\tvar)$. By the definition of $M_\delta(\xi)$ -- see \eqref{eq:1n} -- we have:
\begin{equations}
\big(M_\delta(\pi+\delta\xi)-\lambda\big)^{-1}  
 = \dfrac{1}{\delta}\matrice{\nu_\star \xi - z & \ove{\var_\star} \ \\  \var_\star & -\nu_\star \xi - z}^{-1}
   = 
\dfrac{1}{\delta}  \big( \Di_+(\xi)-z \big)^{-1}, \\
\Di_+(\xi) \de\nu_\star \xi\bst  + \bss \in M_2(\C).
\end{equations}

Substituting in \eqref{eq:3o} shows that $R_\delta(x,y;E_\star+z \delta)$ equals
\begin{equation}\label{eq:3f-b}
  \dfrac{1}{2\pi}\int_\R   \matrice{ \phi_+(\pi+\delta\xi,x) \\ \phi_-(\pi+\delta \xi,x)  }^\top \cdot \chi\left( \delta^{2/3} \xi \right)\big( \Di_+(\xi)-z \big)^{-1} \cdot \ove{\matrice{ \phi_+(\pi+\delta\xi,y) \\  \phi_-(\pi+\delta\xi,y)}} d\xi.
\end{equation}
To deal with the integrand in \eqref{eq:3f-b}, we define $p_\pm(\xi)$, $p_\pm^\star$ and $q_\pm(\xi)$ via:
\begin{equation}
\phi_\pm(\xi,x) \de e^{i\xi x} p_\pm(\xi,x) ;  \ \ p_\pm^\star(x) \de p_\pm(\pi,x); \ \  p_\pm(\pi+\xi,x) - p_\pm^\star(x) \de e^{-i\pi x} q_\pm(\xi,x).
\end{equation}
A direct manipulation implies $\phi_\pm(\pi+\delta\xi,x) = e^{i\delta \xi x} \cdot \big(\phi_\pm^\star(x) + q_\pm(\delta \xi,x) \big)$. Hence, the integrand in \eqref{eq:3f-b} becomes
\begin{equation}\label{eq:8h}
e^{i\delta\xi(x-y)} \matrice{ \phi_+^\star (x) + q_+(\delta \xi,x) \\ \phi_-^\star(x) + q_-(\delta \xi,x) }^\top \cdot \chi\left( \delta^{2/3} \xi \right)\big( \Di_+(\xi)-z \big)^{-1} \cdot \ove{\matrice{ \phi_+^\star(y) +q_+(\delta \xi,y) \\  \phi_-^\star(y) + q_-(\delta \xi,y)}}.
\end{equation}
Observe that $\chi(\delta \xi) \chi(\delta^{2/3} \xi) = \chi(\delta^{2/3} \xi)$ -- as long as $\delta$ is sufficiently small. Therefore, inserting cutoffs $\chi(\delta \xi)$ leaves \eqref{eq:8h} unchanged. It now writes as
\begin{equations}\label{eq:8c}
e^{i\delta\xi(x-y)} \sum_{j,k =0}^1  f_j(\delta\xi,x)^\top \cdot \chi\left( \delta^{2/3} \xi \right)\big( \Di_+(\xi)-z \big)^{-1} \cdot \overline{f_k(\delta \xi,y)}, \\ 
f_0(\xi,x) \de \chi(\xi)\matrice{ \phi_+^\star(x)  \\ \phi_-^\star(x)}; \ \ \ \ f_1(\xi,y) \de \chi(\xi)\matrice{ q_+(\xi,y) \\ q_-(\xi,y)}.
\end{equations}

The leading order term in \eqref{eq:8c} corresponds to $(j,k) = (0,0)$. 
Coming back to \eqref{eq:3f-b}, we deduce that
\begin{equation}
R_\delta(x,y;E_\star+ \delta z) = R^0_\delta(x,y; z) + \mathcal{E}_\delta(x,y; z)
\end{equation}
where
\begin{equations}
R^0_\delta(x,y; z) \de \matrice{ \phi_+^\star(x) \\ \phi_-^\star(x)  }^\top \cdot \dfrac{1}{2\pi}\int_\R e^{i\delta \xi(x-y)} \cdot \chi\left( \delta^{2/3} \xi\right) \big(\Di_+(\xi)-z\big)^{-1}  \cdot  d\xi \cdot \ove{\matrice{ \phi_+^\star(y) \\ \phi_-^\star(y)}};
\\
\mathcal{E}_\delta(x,y; z)  \de  \dfrac{1} {2\pi}
 \int_\R e^{i\delta\xi(x-y)}   \sum_{(j,k) \neq (0,0)}  f_j(\delta\xi,x)^\top \cdot \chi\left( \delta^{2/3}\xi \right) \big( \Di_+(\xi)-z \big)^{-1}  \cdot \overline{f_k(\delta \xi,y)} \cdot d\xi.
\end{equations}

 In Step 2 we will prove that $\mathcal{E}_\delta(x,y; z)$ defines an operator $\mathcal{E}_\delta(z)$ with $\mathcal{E}_\delta(z)=\OO_{L^2}(\delta^{-1/3})$ and in Step 3 we 
will show that since the function $\chi(\delta^{2/3}\ \cdot)$ has slow variations, we can to replace it by $1$. These considerations will imply Proposition \ref{prop:1}.

\indent{\bf Step 2.}\ Motivated by the form error term $\mathcal{E}_\delta(z)$ we introduce a class of kernels. Let $\Phi \in C^\infty(\R \times \R \times \R,\C)$ independent of $\delta$ be such that:
\begin{itemize}
\item[(i)] $\Phi$ is bounded together with all its derivatives;
\item[(ii)] $\Phi(0,x,y) = 0$ for every $x, y \in \R$.
\end{itemize}
Let $s_\delta : \R \rightarrow \C$ be a symbol such that 
 \begin{equation}\label{eq:3a}
\supp(s_\delta) \subset \big[-2\delta^{-2/3},2\delta^{-2/3}\big]; \ \ \ \  \textrm{for all}\  \ell, \ \sup_{\xi \in \R, \ 0 < \delta < 1} \big|\blr{\xi}^{\ell+1}D_\xi^\ell s_\delta(\xi)\big| < \infty.
 \end{equation}
 
An operator bound on $\mathcal{E}_\delta(z)$ is a consequence of: 

\begin{lem}\label{lem:1a} Let  $A$ denote the operator with kernel 
\begin{equation}
A(x,y) = \int_\R  e^{i\delta \xi (x-y)} \cdot s_\delta(\xi) \cdot \Phi(\delta \xi,x,y) d\xi,
\end{equation}
where $s_\delta$ and $\Phi$ are as above. Then $A = \OO_{L^2}(\delta^{-1/3})$.
\end{lem}

\begin{proof} By Schur's test,
it suffices to show that 
\[\sup_{y\in\R}\int_\R |A(x,y)| dx+\sup_{x\in\R}\int_\R |A(x,y)| dy = O\left(\delta^{-1/3}\right).\]
We bound the first term; the second term is treated similarly. Note  that 
\begin{equation}
\int_\R |A(x,y)| dx = \int_\R |A(x+y,y)| dx = \int_\R \left|\int_\R  e^{i\delta \xi x} \cdot s_\delta(\xi) \cdot \Phi(\delta \xi,x+y,y) \cdot d\xi\right|dx.
\end{equation}

We first look at the integral for $|x| \leq L$, where $L>0$ is to be determined later. Note that since $\Phi(0,\cdot) \equiv 0$ and $\Phi(\xi, \cdot)$ has bounded derivatives, there exists a constant $C$ such that
\begin{equation}\label{eq:3d}
 (\xi,x,y) \in \R^3 \ \Rightarrow \ |\Phi(\xi,x,y)| \leq C|\xi|.
\end{equation}  Using that $s_\delta \in S^{-1}$, the support condition \eqref{eq:3a}, and the estimate \eqref{eq:3d}, 
\begin{equation}
\int_{|x| \leq L} |A(x,y)| dx \leq \int_{|x| \leq L} \int_{|\xi|\le 2\delta^{-2/3}} C \blr{\xi}^{-1} \delta \xi \cdot d\xi= O(L\delta^{1/3}).
\end{equation} 
To we deal with the integral over $|x| \geq L$ we take advantage of the oscillatory integrand. Integration  by parts gives:
\begin{equations}
\int_{|x| \geq L} \left|\int_\R e^{i\delta \xi x} s_\delta(\xi)  \Phi(\delta \xi,x+y,y) d\xi\right|dx  
= \int_{|x| \geq L} \left|\int_\R  \dfrac{(D_\xi^2 e^{i\delta \xi x})}{\delta^2 x^2}s_\delta(\xi) \Phi(\delta \xi,x+y,y) d\xi\right|dx \\
= \int_{|x| \geq L} \dfrac{1}{\delta^2 x^2} \left|\int_\R  e^{i\delta \xi x} D_\xi^2\big(s_\delta(\xi)\Phi(\delta \xi,x+y,y)\big) d\xi\right|dx \\
\leq \int_{|x| \geq L} \dfrac{1}{\delta^2 x^2} \int_\R  \left| D_\xi^2\big(s_\delta(\xi) \cdot \Phi(\delta \xi,x+y,y)\big) \right| d\xi dx \\
\leq \int_{|x| \geq L} \dfrac{C}{\delta^2 x^2} \int_{[-2\delta^{-2/3},2\delta^{-2/3}]}  \left(\left| \blr{\xi}^{-3} \delta \xi \right| + \left| \blr{\xi}^{-2} \delta \right| + \left| \blr{\xi}^{-1} \delta^2 \right| \right) d\xi dx.
\end{equations}
In the last inequality, we used \eqref{eq:3d}, the support condition \eqref{eq:3a}, the estimates $D_\xi(\Phi(\delta \xi, \cdot)) = O(\delta)$ and $D_\xi^2(\Phi(\delta \xi, \cdot)) = O(\delta^2)$. We deduce that
\begin{equation}
\int_{|x| \geq L} \left|\int_\R  e^{i\delta \xi x} \cdot s_\delta(\xi) \cdot \Phi(\delta \xi,x+y,y) d\xi\right|dx \leq  \dfrac{C}{\delta^2 L}  \cdot \big(\delta - \delta^2 \ln(\delta)\big) = O\left((\delta L)^{-1} \right).
\end{equation}

Optimize the two bounds by picking $L = \delta^{-2/3}$ and get the bound $O(\delta^{-1/3})$.
 This  concludes the proof.
\end{proof}

The operator $\EEE_\delta(z)$ -- with kernel $\mathcal{E}_\delta(x,y;z)$ -- is a finite sum of operators handled by Lemma \ref{lem:1a}. The functions $s_\delta$ are among coefficients of the matrix $\chi(\delta^{2/3 }\xi) (\Di_+(\xi)-z)^{-1}$: they satisfy \eqref{eq:3a}; the $\Phi(\xi,x,y)$ are among the coefficients of the matrix $\ove{f_k(\xi,y)} \cdot f_j(\xi,x)^\top$ for some $(j,k) \neq (0,0)$ -- in particular (ii) holds. We deduce that
$\EE_\delta(z) = \OO_{L^2}(\delta^{-1/3})$.

\indent{\bf Step 3.}\ So far we have shown that $R_\delta(E_\star+\delta z) = R^0_\delta(z)+ \OO_{L^2}(\delta^{-1/3})$ where the kernel of $R^0_\delta(z)$ is given in Step 1. 
We may write
\begin{equation}\label{eq:8v}
R^0_\delta(z)= \matrice{ \phi_+^\star\\ \phi_-^\star  }^\top \cdot \DD_\delta(z) \cdot \ove{\matrice{ \phi_+^\star \\ \phi_-^\star  }} + \OO_{L^2}(\delta^{-1/3})
\end{equation}
where the Fourier multiplier $\DD_\delta(z)$ has kernel
\begin{equation}
\DD_\delta(x,y;z)= \dfrac{1}{2\pi}\int_\R e^{i\delta \xi(x-y)} \cdot  \chi\left( \delta^{2/3} \xi \right) \big( \Di_+(\xi) - z\big)^{-1}  \cdot d\xi.
\end{equation}
In terms of  $\UU_\delta$ -- given by $\UU_\delta f(x) \de \delta^{1/2}f(\delta x)$ -- we obtain
\begin{equations}
\left(\UU_\delta^{-1} \DD_\delta(\cdot,\cdot;z) \UU_\delta\right)(x,y) 
= \dfrac{1}{2\pi \delta}\int_\R e^{i\xi(x-y)}  \chi\left( \delta^{2/3} \xi \right) \cdot   \big( \Di_+(\xi) - z\big)^{-1} \cdot d\xi.
\end{equations}
In particular, $\UU_\delta^{-1} \DD_\delta(z) \UU_\delta$ is a Fourier multiplier with symbol $\delta^{-1} \cdot \chi\left( \delta^{2/3} \xi \right)  \big( \Di_+(\xi) - z\big)^{-1}$. Split this symbol as
\begin{equation}
 \delta^{-1} \cdot \big( \Di_+(\xi) - z\big)^{-1} + \delta^{-1} \cdot \big( \chi\left( \delta^{2/3} \xi \right) -1 \big)  \big( \Di_+(\xi) - z\big)^{-1}.
\end{equation}
The second term has $L^\infty$-norm controlled by $\delta^{-1/3}$, therefore it is associated with a Fourier multiplier that is $\OO_{L^2}(\delta^{-1/3})$. Therefore,
\[ \DD_\delta(z)= \UU_\delta \cdot \big( \Di - z\big)^{-1} \cdot \UU_\delta^{-1}+ \OO_{L^2}(\delta^{-1/3}).\]
Thanks to \eqref{eq:8v}, we deduce \eqref{R_d}.

\textbf{Step 4.} We indicate the steps to prove \eqref{dR_d}. We first write the kernel of $D_xR_\delta$ as in \eqref{eq:3f}:
\begin{equation}
(D_xR_\delta)(x,y;z) = \dfrac{1}{2\pi}\int_\R  \matrice{ (D_x\phi_+)(\xi,x) \\ (D_x\phi_-)(\xi,x)  }^\top \cdot  \chi\left( \dfrac{\xi-\pi}{\delta^{1/3}} \right)(M_\delta(\xi)-\lambda)^{-1} \cdot \ove{\matrice{ \phi_+(\xi,y) \\   \phi_-(\xi,y)}} \cdot d\xi.
\end{equation}
As in Step 1, we deduce (after setting $\xi = \pi+\delta \xi$) that $D_x R_\delta$ has kernel
\begin{equation}\label{eq:8s}
\dfrac{1}{2\pi}  \int_\R  e^{i\delta\xi(x-y)}   \sum_{j,k =0}^1  \big[\delta \xi \cdot f_j(\delta\xi,x) \big]^\top \cdot \chi(\delta^{2/3} \xi)   \big( \Di_+(\xi)-z \big)^{-1} \cdot \overline{f_k(\delta \xi,y)}  \cdot d\xi  
\end{equation}
\begin{equation}\label{eq:8t}
+ \dfrac{1}{2\pi}  \int_\R  e^{i\delta\xi(x-y)}   \sum_{j,k =0}^1  (D_x f_j)(\delta\xi,x)^\top \cdot \chi(\delta^{2/3} \xi)  \big( \Di_+(\xi)-z \big)^{-1} \cdot \overline{f_k(\delta \xi,y)}  \cdot d\xi
\end{equation}

We introduce  $g_j = \xi  \cdot f_j$ so that \eqref{eq:8s} reduces to
\begin{equations}
\dfrac{1}{2\pi}  \int_\R  e^{i\delta\xi(x-y)}  \sum_{j,k =0}^1  g_j(\delta\xi,x)^\top \cdot \chi(\delta^{2/3} \xi) \big( \Di_+(\xi)-z \big)^{-1} \cdot \overline{f_k(\delta \xi,y)}  \cdot d\xi.
\end{equations}
As in Step 2, this kernel is a finite sum of operators handled by Lemma \ref{lem:1a}. The functions $s_\delta$ are among coefficients of the matrix $\chi(\delta^{2/3 }\xi) (\Di_+(\xi)-z)^{-1}$: they satisfy \eqref{eq:3a}; the $\Phi(\xi,x,y)$ are among the coefficients of the matrix $\ove{f_k(\xi,y)} \cdot g_j(\xi,x)^\top$. Note that (i) holds; since $g_j(0,x) = 0$, (ii) holds. We deduce that \eqref{eq:8t} is the kernel of an operator $\OO_{L^2}(\delta^{-1/3})$.

We now look at \eqref{eq:8t}. In order to control the terms associated with $(j,k) \neq (0,0)$, we can apply the same analysis as in the end of Step 2. We deduce that the kernel 
\begin{equation}
\dfrac{1}{2\pi}  \int_\R  e^{i\delta\xi(x-y)} \sum_{(j,k) \neq (0,0)}  (D_x f_j)(\delta\xi,x)^\top \cdot \chi(\delta^{2/3} \xi)   \big( \Di_+(\xi)-z \big)^{-1} \cdot \overline{f_k(\delta \xi,y)}  \cdot d\xi
\end{equation}
represents an operator that is $\OO_{L^2}(\delta^{-1/3})$. Hence, the leading term in \eqref{eq:8t} corresponds to $j=k=0$: it has kernel
\begin{equation}
\matrice{ D_x\phi_+^\star(x) \\ D_x\phi_-^\star(x)  }^\top \cdot \dfrac{1}{2\pi}\int_\R e^{i\delta \xi(x-y)} \cdot \chi\left( \delta^{2/3} \xi\right) \big(\Di_+(\xi)-z\big)^{-1}  \cdot d\xi \cdot \ove{\matrice{ \phi_+^\star(y) \\ \phi_-^\star(y)}}.
\end{equation}
The same procedure as in Step 3 allows to replace $\chi$ by $1$; and identify the resulting kernel with the operator on the RHS of \eqref{dR_d}.
\end{proof}

Proposition \ref{prop:1} together with the estimate \eqref{eq:2m} imply Theorem \ref{thm:1} for the operator $P_\delta$.  This theorem has a corollary for $P_{-\delta}$. We will need the operator $\Di_-$, which is the asymptotic limit of $\Di$ as $x \rightarrow - \infty$:
\begin{equation}
\Di_- \de \nu_\star \bst D_x - \bss.
\end{equation}

\begin{cor}\label{cor:1} There exists $\delta_0 > 0$ such that for $\delta \in (0,\delta_0)$ and $\lambda \in \Dd(E_\star,\tvar \delta)$,
\begin{equations}
(P_{-\delta}-\lambda)^{-1} =  S_{-\delta}\left(\frac{\lambda-E_\star}{\delta}\right) + \OO_{L^2} (\delta^{-1/3}), \ \ \ \
D_x(P_{-\delta}-\lambda)^{-1} =  S_{-\delta}^D\left(\frac{\lambda-E_\star}{\delta}\right) + \OO_{L^2} (\delta^{-1/3})
\end{equations}
where
\begin{align}
S_{-\delta}(z) &= \dfrac{1}{\delta} \matrice{ \phi_+^\star \\ \phi_-^\star  }^\top \cdot    \UU_\delta\  \big( \Di_- - z\big)^{-1}\   \UU_\delta^{-1}  \cdot \ove{\matrice{ \phi_+^\star \\ \phi_-^\star  }}, \\
S_{-\delta}^D(z) &= \dfrac{1}{\delta} \matrice{ D_x\phi_+^\star \\ D_x\phi_-^\star  }^\top \cdot    \UU_\delta\  \big( \Di_- - z\big)^{-1}\   \UU_\delta^{-1}  \cdot \ove{\matrice{ \phi_+^\star \\ \phi_-^\star  }}.
\end{align}
\end{cor}

The proof of Corollary \ref{cor:1} follows from changing $W$ to $-W$ and applying Theorem \ref{thm:1}. The only parameter to change is $\var_\star$, which becomes $-\var_\star$. This produces the operator $\Di_-$ instead of $\Di_+$.

\section{Construction of a parametrix.}\label{sec:5}

We are interested in the eigenvalues of the operator 
\begin{equation}
\PP_\delta = D_x^2 + V + \delta\cdot \kappa(\delta \ \cdot)\cdot W
\end{equation}
that lie in the spectral gap about the Dirac energy $E_\star$. The operator $\PP_\delta$ has both point and essential spectrum. Its essential spectrum is determined by the asymptotic operators $P_{\pm\delta}$, expanded in the previous section. 
It is natural to compose $\PP_\delta$ with a bounded operator  $Q_\delta(\lambda)$ (a {\it parametrix}, constructed using $P_{\pm\delta}$) 
which effectively ``divides out" the essential spectrum. The resulting operator is of the form $\Id+K_\delta(\lambda)$, where $K_\delta(\lambda)$ is compact and depends analytically on $\lambda$. The eigenvalues of $\PP_\delta$ are among the poles of $(\Id+K_\delta(\lambda))^{-1}$.

Our strategy in \S\ref{sec:5} and \S\ref{sec:6} is:
\begin{itemize}
\item Construct explicitly $Q_\delta(\lambda)$ and $K_\delta(\lambda)$ from the asymptotic operators $P_{\pm \delta}$;
\item Apply Theorem \ref{thm:1} to derive an expansion of  $K_\delta(\lambda)$ for small $\delta$;
\item Use the expansion of $K_\delta(\lambda)$  to locate and count the poles of $(\Id+K_\delta(\lambda))^{-1}$.
\end{itemize}

\subsection{Analytic Fredholm theory}\label{fred-thy} We briefly review analytic Fredholm theory. We refer to \cite[Appendix C]{DZ}. Let $U \subset \C$ be open and $\HH$ be a Hilbert space. Let $\BB(\HH)$ be the space of bounded operators on $\HH$. A family $\zeta \in U \mapsto T(\zeta) \in \BB(\HH)$ is holomorphic (resp. meromorphic) on $U$ if for every $u, v \in \HH$, the function 
\begin{equation}
\zeta \in U \ \mapsto \blr{u, T(\zeta) v}_\HH \in \C
\end{equation}
is holomorphic (resp. meromorphic). 

We recall that a bounded operator $T \in \BB(\HH)$ is Fredholm if both its kernel and cokernel are finite dimensional. We will use the following result:

\begin{lem}\cite[Theorem C.5]{DZ}\label{lem:1k} Let $\zeta \in U \mapsto T(\zeta) \in \BB(\HH)$ be a holomorphic family of Fredholm operators. If there exists $\zeta_0 \in U$ such that $T(\zeta_0)$ is invertible, then $\zeta \in U \mapsto T(\zeta)^{-1}$ (initially defined in a small neighborhood of $\zeta_0$) extends to a meromorphic family of operators with poles of finite rank.
\end{lem}

\subsection{Parametrix}\label{sec:5.1}
We define resonances for the operator $\PP_\delta$. We look for a parametrix of $\PP_\delta-\lambda$, i.e. an operator $Q_\delta(\lambda) : L^2 \rightarrow H^2$ such that for $\lambda \in \Dd(E_\star,\tvar \delta)$,
\begin{equation}\label{eq:3i}
(\PP_\delta-\lambda)Q_\delta(\lambda) - \Id \ \de \ K_\delta(\lambda)  \text{ is a compact operator on } L^2.
\end{equation}
This requires $Q_\delta(\lambda)$ to be equal to the asymptotic resolvents $(P_{\pm\delta}-\lambda)^{-1}$ near $x=\pm \infty$. Since $\kappa(x) \to \pm 1$ as $x\to\pm\infty$, a natural parametrix is
\begin{equation}
Q_\delta(\lambda) = \dfrac{1+\kappa_\delta}{2} (P_\delta-\lambda)^{-1} + \dfrac{1-\kappa_\delta}{2} (P_{-\delta}-\lambda)^{-1}, \ \ \ 
\kappa_\delta(x) = \kappa(\delta x).
\end{equation}

For $\tvar\in(0,|\var_\star|)$, 
 and for sufficiently small positive $\delta$, $Q_\delta(\lambda)$ is well-defined for $\lambda \in \C^+ \cup \Dd(E_\star,\tvar\delta)$ because $P_\delta-\lambda$ is invertible  for $\lambda$ in this set -- see Theorem \ref{thm:1}.
  We next verify that  with this choice of $Q_\delta(\lambda)$, assertion \eqref{eq:3i} holds for $\lambda \in \C^+ \cup \Dd(E_\star,\tvar\delta)$:
\begin{equations}
(\PP_\delta-\lambda) Q_\delta(\lambda) = \dfrac{1}{2} \sum_{\pm} (\PP_\delta-\lambda) (1 \pm \kappa_\delta) (P_{\pm\delta}-\lambda)^{-1} \\
 = \dfrac{1}{2} \sum_{\pm} (P_{\pm \delta}-\lambda) (1 \pm \kappa_\delta) (P_{\pm\delta}-\lambda)^{-1} + \dfrac{1}{2} \sum_{\pm} (\PP_\delta-P_{\pm\delta}) (1 \pm \kappa_\delta) (P_{\pm\delta}-\lambda)^{-1} \\
 = \Id + \dfrac{1}{2} \sum_{\pm} \big[D_x^2, \pm \kappa_\delta\big] (P_{\pm\delta}-\lambda)^{-1} + \dfrac{\delta}{2} \sum_{\pm} ( \kappa_\delta  \mp  1) W (1 \pm \kappa_\delta) (P_{\pm\delta}-\lambda)^{-1} \\
 = \Id + \dfrac{1}{2} \sum_{\pm} \pm \big[D_x^2, \kappa_\delta\big] (P_{\pm\delta}-\lambda)^{-1} + \dfrac{\delta}{2} \sum_{\pm} \pm \left(\kappa_\delta^2-1\right) W  (P_{\pm\delta}-\lambda)^{-1}.
 \end{equations}
 
 Thus we have $(\PP_\delta-\lambda) Q_\delta(\lambda) =  \Id + K_\delta(\lambda)$ where
 \begin{equation}\label{Kd-def}
K_\delta(\lambda) \de \dfrac{1}{2} \left( \big[D_x^2,\kappa_\delta\big] + \delta \left(\kappa_\delta^2-1\right) W  \right) \cdot \left((P_\delta-\lambda)^{-1} - (P_{-\delta}-\lambda)^{-1}\right).
\end{equation}

\begin{prop}\label{prop:1m} There exists $\delta_0 > 0$ such that for every $\delta \in (0,\delta_0)$ and all
  $\lambda \in  \C^+ \cup \Dd(E_\star,\tvar\delta)$, the operator $K_\delta(\lambda)$ is compact and $\Id + K_\delta(\lambda)$ forms a holomorphic family of Fredholm operators. In addition, $\Id+K_\delta(i)$ is invertible on $L^2$. 
\end{prop}

\begin{proof}[Proof of Proposition \ref{prop:1m}]\ We show that (a)  $K_\delta(\lambda)$ forms a holomorphic family of compact operators on $L^2$
 and (b) $\|K_\delta(i)\|_{L^2} \leq 1/2$ for $\delta$ sufficiently small.
Assertion (a) implies that $\Id + K_\delta(\lambda)$ forms a holomorphic family of Fredholm operators and assertion (b) implies that $\Id + K_\delta(i)$ is invertible for $\delta$ sufficiently small (by a Neumann series). 

We start with (a). The holomorphic dependence on $\lambda$ is clear. To show that  $K_\delta(\lambda)$ is compact on $L^2$, we observe that $D_x\kappa_\delta, D_x^2 \kappa_\delta$ and $\kappa_\delta^2-1$ are all smooth, compactly supported functions (with support contained in that of $\kappa_\delta^2-1$); and that the operators $(P_{\pm \delta} - \lambda)^{-1}$ map $L^2$ to $H^2$ when $\lambda$ is outside the spectrum of $P_\delta$. Therefore, the range of $K_\delta(\lambda)$ consists of $H^2$-functions with support contained in $\supp(\kappa_\delta^2-1)$. This is a compact subset of $L^2$ by the Rellich--Kondrachov theorem.

We now prove (b).  We first show that $(P_\delta+i)^{-1}$ is bounded from $L^2$ to $H^2$, with uniform bound in $\delta \in (0,1)$. By the spectral theorem, $\|(P_\delta+i)^{-1}\|_{L^2} \leq 1$. We now use the elliptic estimate: if $f \in H^2$ then
\begin{equation}
|f|_{H^2} \leq \left|D_x^2 f\right|_{L^2} + |f|_{L^2} \leq  |(P_\delta+i)f|_{L^2} + C |f|_{L^2}, 
\end{equation}
where $C$ does not depend on $\delta$ when $0 < \delta < 1$. This shows that 
\begin{equation}
\left|(P_\delta+i)^{-1}f\right|_{H^2} \leq  |f|_{L^2} + C |(P_\delta+i)^{-1}f|_{L^2} \leq C |f|_{L^2}.
\end{equation}
Thus $\|(P_\delta+i)^{-1}\|_{L^2 \rightarrow H^2} = O(1)$ uniformly in $\delta \in (0,1)$ as claimed. Now, we note that $D_x\kappa_\delta, D_x^2 \kappa_\delta$ and $\delta(\kappa_\delta^2-1)$ are all $O_{L^\infty}(\delta)$. This together with the uniform bound on $(P_\delta+i)^{-1} : L^2 \rightarrow H^2$ implies that $K_\delta(i) = \OO_{L^2}(\delta)$. Taking $\delta$ sufficiently small  concludes the proof. 
\end{proof}

Proposition \ref{prop:1m} implies assertion \eqref{eq:3i}. It relates the eigenvalues of $\PP_\delta$ with the poles of $(\Id + K_\delta(\lambda))^{-1}$.  Indeed, Lemma \ref{lem:1k} shows that $(\Id + K_\delta(\lambda))^{-1} : L^2 \rightarrow L^2$, initially defined for $\lambda$ in a neighborhood of $i$, extends analytically to a family of bounded operators on $L^2$ for $\lambda \in  \C^+ \cup \Dd(E_\star,\tvar\delta)$. From \eqref{eq:3i}, we deduce that $\PP_\delta-\lambda$ is invertible away from the poles of $(\Id + K_\delta(\lambda))^{-1}$; and
\begin{equation}
(\PP_\delta - \lambda)^{-1} = Q_\delta(\lambda) \cdot \big(\Id + K_\delta(\lambda)\big)^{-1}.
\end{equation}
In particular, the eigenvalues of $\PP_\delta$ in $[E_\star-\tvar \delta, E_\star+\tvar\delta]$ are among the poles of $(\Id + K_\delta(\lambda))^{-1}$.

\begin{rmk}\label{degeneracy} The converse inclusion does not necessarily hold. This is because $Q_\delta(\lambda)$ may cancel poles of $(\Id + K_\delta(\lambda))^{-1}$. This phenomena unfortunately happens here. We will fix this issue in \S\ref{sec:5.3}.
\end{rmk}

\subsection{Expansion of $\Id + K_\delta(\lambda)$}\label{sec:5.2}
To expand the expression for $K_\delta(\lambda)$ in \eqref{Kd-def}, we first study the resolvent difference $(P_\delta-\lambda)^{-1} - (P_{-\delta}-\lambda)^{-1}$. We apply Theorem \ref{thm:1} and Corollary \ref{cor:1}. Observe that
\begin{equations}
\big( \Di_+ - z\big)^{-1} - \big( \Di_- - z\big)^{-1} = \dfrac{2 \bss}{\nu_\star^2 D_x^2+|\var_\star|^2- z^2}.
\end{equations}
Use $\lambda=E_\star+z \delta$, $z \in \Dd(0,\tvar)$. Theorem \ref{thm:1} and Corollary \ref{cor:1} imply
\begin{equations}
(P_\delta-\lambda)^{-1} - (P_{-\delta}-\lambda)^{-1} = \dfrac{2}{\delta} \matrice{ \phi_+^\star \\ \phi_-^\star  }^\top \UU_\delta \cdot \dfrac{\bss}{\nu_\star^2 D_x^2+|\var_\star|^2- z^2} \cdot \UU_\delta^{-1}\ove{\matrice{ \phi_+^\star \\ \phi_-^\star  }} + \OO_{L^2}(\delta^{-1/3}), \\
D_x\left((P_\delta-\lambda)^{-1} - (P_{-\delta}-\lambda)^{-1}\right) = \dfrac{2}{\delta} \matrice{ D_x\phi_+^\star \\ D_x\phi_-^\star  }^\top \UU_\delta \cdot  \dfrac{\bss}{\nu_\star^2 D_x^2+|\var_\star|^2- z^2}  \cdot \UU_\delta^{-1}\ove{\matrice{ \phi_+^\star \\ \phi_-^\star  }} + \OO_{L^2}(\delta^{-1/3}).
\label{Pdiff}
\end{equations}

From \eqref{Kd-def} and \eqref{Pdiff}, for $z \in \Dd(0,\tvar)$, the operator $\Id + K_\delta(E_\star+z \delta) $ splits as
\begin{equation}\label{eq:1j}
\Id + K_\delta(E_\star+z \delta) = \Id + \EE_\delta(z) + \KK_\delta(z),
\end{equation}
where $\EE_\delta(z) = \OO_{L^2}(\delta^{1/3})$, $z \in \Dd(0,\tvar)$ and $\KK_\delta(z)$ equals
\begin{equation}\label{eq:2p}
\left( 2(D_x\kappa)_\delta \matrice{ D_x\phi_+^\star \\ D_x\phi_-^\star  }^\top + (\kappa_\delta^2-1) W  \matrice{ \phi_+^\star \\ \phi_-^\star  }^\top \right) \UU_\delta \cdot  \dfrac{\bss}{\nu_\star^2 D_x^2+|\var_\star|^2- z^2}  \cdot \UU_\delta^{-1}\ove{\matrice{ \phi_+^\star \\ \phi_-^\star  }}.
\end{equation}
The operator $\KK_\delta(z)$ is a trace-class operator.
The trace-class property holds because $\KK_\delta(z)$ maps $L^2$ to $ H^2$ functions with fixed compact support. This is sufficient in dimension 1; see for instance \cite[(B.3.9)]{DZ}.

\section{Eigenvalues of $\mathscr{P}_\delta$ and the effective Dirac operator $\Di$}\label{sec:6}

In this section we show that the expansion of $\KK_\delta(z)$ given in \eqref{eq:2p} helps locate the eigenvalues of $\PP_\delta$. We will see that $\KK_\delta(z)$ is related to a matrix Schr\"odinger operator with localized, highly oscillatory potential. The weak -- or homogenized -- limit of this operator is $\Di{}^2-z^2$, which is not invertible precisely when $z$ or $-z$ is an eigenvalue~of~$\Di$.

This yields only a weak form of Theorem \ref{thm:2} -- see Proposition \ref{prop:1q} below. Indeed, the non-zero eigenvalues of $\Di{}^2$ are double -- because the spectrum of $\Di$ is symmetric about $0$ -- while Theorem \ref{thm:2} predicts that the point spectrum of $\PP_\delta$ is simple.

This paradox is illusive; zeros of $Q_\delta(\lambda)$ cancel poles of $(\Id + K_\delta(\lambda))^{-1}$. This phenomena is annoying, but it is very unstable. In the last step of the proof of Theorem \ref{thm:2}, we add a specifically-designed rank-one operator to the natural parametrix $Q_\delta(\lambda)$. This twists the parametrix and  deforms away spurious modes.

\subsection{Fredholm determinants}\label{fred-det} We refer to \cite[Appendix B]{DZ}. 
Let $\HH, \HH'$ be Hilbert spaces; let $\BB(\HH,\HH')$ be the space of bounded operators $\HH \rightarrow \HH'$. A compact operator $T \in \BB(\HH,\HH')$ is trace-class if the eigenvalues $\{\mu_j\}_{j=1}^\infty$ of $(TT^*)^{1/2} \in \BB(\HH')$ (called singular values of $T$) are summable. The space of trace-class operators on $\HH$ is denoted $\LL(\HH,\HH')$; when $\HH = \HH'$ we more simply write $\LL(\HH)$. The trace-class norm is
\begin{equation}
\|T\|_{\LL(\HH,\HH')} \de \sum_{j=1}^\infty \mu_j.
\end{equation}
The space $\LL(\HH,\HH')$ is a two-sided ideal: if $A  \in \BB(\HH',\HH)$ and $T \in \LL(\HH,\HH')$ then $AT \in \LL(\HH)$ and $TA \in \LL(\HH')$. See \cite[(B.4.6)]{DZ}.

If $T$ is trace-class (hence compact), we can approximate $T$ by a sequence $T_n$ of finite-rank operators in $\LL(\HH,\HH')$. When $\HH=\HH'$, the trace-class property  shows that $\Det(\Id + T_n)$ (well-defined because $T_n$ has finite rank) converges; and the limit depends on $T$ only. The limit is the \textit{Fredholm determinant} $\Det(\Id + T)$. See \cite[\S B.5]{DZ}. 

Fredholm determinants inherit most properties of their finite-dimensional analogs. Three of them are particularly relevant here:
\begin{itemize}
\item If $T \in \LL(\HH)$ then $\Det(\Id+T) = 0$ if and only if $\Id + T$ is not invertible -- see \cite[Proposition B.25]{DZ}.
\item The Fredholm determinant is cyclic -- see \cite[(B.5.13)]{DZ}: 
\begin{equation}\label{eq:2o}
A \in \BB(\HH',\HH), \ \ T \in \LL(\HH,\HH') \ \ \Rightarrow \ \  \Det(\Id + AT) = \Det(\Id+TA).
\end{equation}
\item If $z \in U \mapsto T(z)$ is holomorphic family of trace-class operators, then $z \in U \mapsto \Det(\Id+T(z))$ is holomorphic.
\end{itemize}
To prove the last property, we remark that it holds for holomorphic families of finite-rank operators. Hence $z \in U \mapsto \Det(\Id + T(z))$ is a uniform limit of holomorphic functions, thus holomorphic.

We will use the following lemma, proved in Appendix \ref{sec:A}:

\begin{lem}\label{lem:1v} For any $s \in (0,1/2)$, there exists $C$ with the following. Let $A, A' : L^2 \rightarrow L^2$ and $B : L^2 \rightarrow H^2$ such that $B$ is continuous from $H^{-s}$ to $H^{2-s}$. Then for any $\chi \in C_0^\infty(\R,[0,1])$, $A \chi B$ and $A'\chi B$ are trace-class and
\begin{equation}
\big|\Det(\Id+A \chi B) - \Det(\Id+A' \chi B)\big| \leq C \|A-A'\|_{H^s \rightarrow H^{-s}} \cdot \|\chi B\|_{H^{-s} \rightarrow H^{2-s}}.
\end{equation}
\end{lem}

\subsection{The Dirac operator}\label{sec:7.3}

We introduce the Dirac operator $\Di \de \nu_\star \bst D_x + \bss \kappa$, where we recall that
\begin{equation}
\nu_\star = 2\blr{ \phi_+^\star,D_x\phi_+^\star},  \ \ \ \  \var_\star = \blr{\phi_+^\star,W\phi_-^\star}, \ \ \ \ \bss = \matrice{0 & \ove{\var_\star} \ \\ \var_\star & 0}.
\end{equation}
It is a selfadjoint operator on $L^2$ with domain $H^1$. It will emerge naturally in the study of the spectrum of $\PP_\delta$. 

\begin{thm}\label{lem:1b} The following holds:
\begin{itemize}
\item[(1)]  $\Sigma(\Di)$ is independent of the choice of $(\phi_+^\star, \phi_-^\star)$ in Proposition \ref{prop:1w}.
\item[(2)] The essential spectrum of $\Di$ is $\R \setminus (-|\var_\star|, |\var_\star|)$.
\item[(3)] The pure point spectrum of $\Di$ consists of finitely many simple eigenvalues
\begin{equation}
-|\var_\star| < \var_{-N} < \dots < \var_{-1} < \var_0 = 0 < \var_1 < \dots < \var_N < |\var_\star|,
\end{equation}
satisfying $\var_{-j} = -\var_j$.
\end{itemize}
\end{thm}

\begin{proof} We start with the first point. If we choose another Dirac eigenbasis satisfying Proposition \ref{prop:1w} then $\nu_\star$ remains the same; and $\var_\star$ transforms to $\w^2 \var_\star$, where $\w \in \Ss^1$. The operator $\Di$ tranforms to
\begin{equation}\label{eq:1v}
\nu_\star \bst D_x + \matrice{0 & \ove{\w^2 \var_\star} \ \\ \w^2 \var_\star & 0} \kappa.
\end{equation}
We observe that
\begin{equations}
\matrice{\ow & 0 \\ 0 & \w} \bst \matrice{\w & 0 \\ 0 & \ow} = |\w|^2 \bst = \bst;
\\
\matrice{\ow & 0 \\ 0 & \w} \bss \matrice{\w & 0 \\ 0 & \ow} = \matrice{\ow & 0 \\ 0 & \w} \matrice{0 & \ove{\var_\star} \ \\ \var_\star & 0} \matrice{\w & 0 \\ 0 & \ow} = \matrice{0 & \ove{\w^2 \var_\star} \ \\ \w^2 \var_\star & 0}.
\end{equations}
This implies that $\Di$ and the operator \eqref{eq:1v} are conjugated. Therefore they have the same spectrum. This proves (1).

The assertions regarding the essential spectrum of $\Di$ and the zero mode were proved in \cite[\S 4.1]{FLTW2} when $\var_\star \in \R$. The same proofs apply to the more general case $\var_\star \in \C$. We next show that eigenvalues of $\Di$ do not accumulate at the edges $\pm |\var_\star|$ of the essential spectrum. A computation using $\bst \bss = - \bst \bss$ and $\bst^2 = \Id_2$, $\bss^2 = |\var_\star|^2 \Id_2$ yields
\begin{equation}\label{eq:1p}
\Di^2 - |\var_\star|^2  = \nu_\star^2 D_x^2 + \MM_0, \ \ \ \ \MM_0 \de |\var_\star|^2 \left(\kappa^2-1\right) + \nu_\star  (D_x\kappa)\bst \bss.
\end{equation}
In particular,  $\Di^2 - |\var_\star|^2$ is a Schr\"odinger operator with compactly supported potential. Therefore, its resonances (in particular, its eigenvalues) are isolated -- see e.g. \cite[Theorem 2.2]{DZ}. It follows that $\Di$ has only isolated eigenvalues.  

To show the symmetry of the spectrum, we note the identity
\begin{equation}\label{eq:1o}
\bst \bss \cdot \Di \cdot (\bss \bst)^{-1} = -\Di,
\end{equation}
which follows from the relations $\bst \bss = - \bst \bss$, $\bst^2 = \Id_2$ and $\bss^2 = |\var_\star|^2 \Id_2$. The identity \eqref{eq:1o} implies that $\Di$ and $-\Di$ have same spectra; in particular, $\Sigma(\Di)$ must be symmetric about $0$.

The proof that all eigenvalues of $\Di$ are simple is presented in Appendix \ref{app:3}.
\end{proof}

\begin{rmk} The Dirac operator $\Di$ may have arbitrarily many eigenvalues. Fix $N \in \N$. The paper \cite{LWW} studies the point spectrum of  the Dirac operator $\Di_a=\nu_\star\bst D_x +  \var_\star \bsu  \kappa_a$, where $\kappa_a$ is constructed by gluing together $2N+1$ \textit{reference domain walls} $\kappa$ separated by a distance $a$ -- see Figure \ref{fig:4}. If $\epsi > 0$ is given and $\nu_\star\bst D_x +  \var_\star \bsu  \kappa$ has a single eigenvalue then for $a$ sufficiently large, \cite{LWW} shows that $\Di_a$ has precisely $2N+1$ eigenvalues in $(-\var_\sharp+\epsi,\var_\sharp-\epsi)$. 
\end{rmk}

 \begin{center}
\begin{figure}
\begin{tikzpicture}
      \draw[thick,->] (-7.5,0) -- (7.5,0);
      \draw[thick,->] (0,-2) -- (0,2);
      \draw[blue,thick,->] (1,.8) -- (3.3,.8);
      \draw[blue,thick,<-] (-.3,.8) -- (1,.8);
      \node at (7.2,-.3) {$x$};
      \node at (1.5,1.1) {$\sim a$};
      
      \node[red] at (7,1.8) {$\kappa_a$};
   \draw[thick,domain=-1.5:1.5,smooth,variable=\x,red] plot ({\x},{ -1.5* (e^(4*\x)-1)/(e^(4*\x)+1)  });
   \draw[thick,domain=-4.5:1.5,smooth,variable=\x,red] plot ({\x-3},{ 1.5* (e^(4*\x)-1)/(e^(4*\x)+1)  });
   \draw[thick,domain=-1.5:4.5,smooth,variable=\x,red] plot ({\x+3},{ -1.5* (e^(-4*\x)-1)/(e^(-4*\x)+1)  });
\end{tikzpicture} 
\label{fig:4}{\caption{The function $\kappa_a$ is the concatenation of $3$ ``single-bump" domain walls separated by $a$.}}
\end{figure}
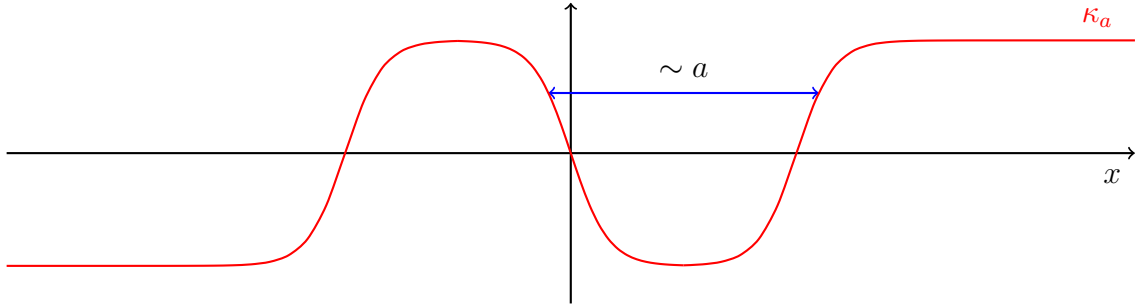
\end{center}
 
\vspace*{-9mm}

\subsection{Localization of eigenvalues}\label{ev-loc}

\begin{prop}\label{prop:1q} Let $\var_N < \tvar < |\var_\star|$
and define
\begin{equation}
\ZZ_\delta \de  \big\{ z \in \Dd(0,\tvar) : E_\star+z \delta \in \Sigma(\PP_\delta) \big\}.
\end{equation}
Then $\dist(\ZZ_\delta,\Sigma_\pp(\Di))$ tends to $0$ as $\delta$ tends to $0$. 
\end{prop}

Proposition \ref{prop:1q} shows that the eigenvalues of $\PP_\delta$ in $[E_\star-\tvar\delta,E_\star+\tvar\delta]$ must necessarily be of the form $E_\star + \var_j \delta + o(\delta)$, where the $\var_j$'s are the eigenvalues of the effective Dirac operator $\Di$. This proves the assertion (A) of Theorem \ref{thm:2}.

\begin{proof} 1. If $\lambda$ is an eigenvalue of $\PP_\delta$, then $\Id+K_\delta(\lambda)$ is not invertible. We recall that $\Id + K_\delta(E_\star+z \delta)$ decomposes as
\begin{equation} 
\Id + \EE_\delta(z) + \KK_\delta(z),
\end{equation}
where $\EE_\delta(z) = \OO_{L^2}(\delta^{1/3})$ and $\KK_\delta(z)$ is trace-class, see \eqref{eq:1j}. Hence, for $\delta$ sufficiently small, $\Id + \EE_\delta(z)$ is invertible and
\begin{equation}
 \lambda=E_\star+z \delta \in \Sigma(\PP_\delta) \cap \Dd(E_\star,\var_\sharp \delta) \ \  \Rightarrow  \ \ \Id + \big(\Id + \EE_\delta(z)\big)^{-1} \KK_\delta(z) \text{ is not invertible.}
\end{equation}
Since $\KK_\delta(z)$ is a trace-class operator we deduce that
\begin{equations}\label{eq:1q}
 \lambda=E_\star+z \delta \in \Sigma(\PP_\delta) \cap \Dd(E_\star,\var_\sharp \delta) \\ 
 \Rightarrow \ \ \DDD_\delta(z) = 0, \ \ \ \
\DDD_\delta(z) \de \Det\left(\Id+\big(\Id+\EE_\delta(z)\big)^{-1} \KK_\delta(z)\right).
\end{equations}
Therefore we get $\ZZ_\delta \subset \DDD_\delta^{-1}(0) \cap \Dd(0,\tvar)$.

2. We have reduced the proof of Proposition \ref{prop:1q} to a problem in complex analysis: locate the zeros of $\DDD_\delta$. Note first that since $\EE_\delta(z) = \OO_{L^2}(\delta^{1/3})$, uniformly in $z\in\ove{\Dd(0,\tvar)}$, we have
\begin{equation}
\DDD_\delta(z) = D_\delta(z) + O\left(\delta^{1/3}\right), \ \ \   D_\delta(z) \de \Det\big(\Id+\KK_\delta(z)\big).
\end{equation}

3. Using the cyclicity of the determinant \eqref{eq:2o}, we rewrite the Fredholm determinant $D_\delta(z)$ as the determinant of an operator acting on $L^2(\R,\C^2)$.
Specifically  permuting the matrix term on the right implies that $D_\delta(z)$ equals
\begin{equations}
\Det\left( \Id +   \left( 2(D_x\kappa)_\delta \matrice{ D_x\phi_+^\star \\ D_x\phi_-^\star  }^\top + (\kappa_\delta^2-1) W  \matrice{ \phi_+^\star \\ \phi_-^\star  }^\top \right) \UU_\delta \cdot  \dfrac{\bss}{\nu_\star^2 D_x^2+|\var_\star|^2- z^2}  \cdot \UU_\delta^{-1}\ove{\matrice{ \phi_+^\star \\ \phi_-^\star  }} \right)
\end{equations}
\begin{equations}
 = \Det\left( \Id + \UU_\delta^{-1}\ove{\matrice{ \phi_+^\star \\ \phi_-^\star  } } \left( 2(D_x\kappa)_\delta \matrice{ D_x\phi_+^\star \\ D_x\phi_-^\star  }^\top + (\kappa_\delta^2-1) W  \matrice{ \phi_+^\star \\ \phi_-^\star  }^\top \right) \UU_\delta \cdot  \dfrac{\bss}{\nu_\star^2 D_x^2+|\var_\star|^2- z^2}   \right)
\end{equations}
\begin{equations}
 = \Det\left( \Id + \MM_\delta \left(\nu_\star^2 D_x^2+|\var_\star|^2- z^2\right)^{-1} \right).
\end{equations}

\noindent Here, $\MM_\delta(x)= \MM(x/\delta,x)$ is a matrix-valued multiscale highly oscillatory potential. Specifically,  $\MM(x,y)$ is $1$-periodic in $x$, compactly supported in $y$, smooth in both variables equal to
\begin{equation}\label{eq:3l}
 \ove{\matrice{ \phi_+^\star(x) \\ \phi_-^\star(x)  } } \left( 2(D_x\kappa)(y) \matrice{ D_x\phi_+^\star(x) \\ D_x\phi_-^\star(x)  }^\top + \left(\kappa(y)^2-1 \right) W (x) \matrice{ \phi_+^\star(x) \\ \phi_-^\star(x)  }^\top \right)  \bss.
\end{equation}

4. We would like to expand $D_\delta(z)$.  Due to the oscillations of $\MM_\delta$,  the operator $\Id + \MM_\delta \left(\nu_\star^2 D_x^2+|\var_\star|^2- z^2\right)^{-1} $ does not converge
in the operator norm. The $L^2$-weak limit of $\MM_\delta$ (corresponding to the average with respect to $x$) is computed using that as $\delta \to 0$:
\begin{equation}
2\ove{\matrice{ \phi_+^\star(\cdot/\delta) \\ \phi_-^\star(\cdot/\delta)  } }\ \matrice{ D_x\phi_+^\star(\cdot/\delta) \\ D_x\phi_-^\star(\cdot/\delta)  }^\top\ \rightharpoonup\ \nu_\star \bst\ \ \textrm{and}\ \  \ove{\matrice{ \phi_+^\star(\cdot/\delta) \\ \phi_-^\star(\cdot/\delta)  } }\ W \ \matrice{\phi_+^\star(\cdot/\delta) \\ \phi_-^\star(\cdot/\delta)  }^\top\  
\rightharpoonup\ \bss.
\label{w-limits}
\end{equation}
We find $\MM_\delta \rightharpoonup \MM_0$, where $\MM_0$ already appeared in \eqref{eq:1p}:
\begin{equation}
\MM_\delta \rightharpoonup \left((D_x \kappa) \nu_\star \bst + \left(\kappa^2-1\right) \bss\right) \bss = |\var_\star|^2 \left(\kappa^2-1\right) + \nu_\star  (D_x\kappa)\bst \bss = \MM_0.
\label{W0-def}
\end{equation}
We now show that $D_\delta(z)$ can be expanded about the formal weak limit 
\begin{equation}
D_0(z) \de \Det\left(\Id + \MM_0 \left(\nu_\star^2 D_x^2+|\var_\star|^2-z^2\right)^{-1}\right).
\end{equation}
Specifically, using Lemma \ref{lem:1u}, we have $|\MM_\delta-\MM_0|_{H^{-1/4}} = O(\delta^{1/4})$. As a multiplication operator, $\MM_\delta-\MM_0$ is bounded from $H^{1/4}$ to $H^{-1/4}$ with bound
\begin{equation}
\| \MM_\delta-\MM_0 \|_{H^{1/4} \rightarrow H^{-1/4}} = O\left(\delta^{1/4}\right),
\end{equation}
see \cite[Lemma 2.1]{Dr4}. Because of Lemma \ref{lem:1v} we deduce that
\begin{equation}
D_\delta(z) = D_0(z) + O\left(\delta^{1/4}\right).
\label{DzD0z}
\end{equation}

5. A complex number $z \in \Dd(0,\tvar)$ is a zero of $D_0$ if and only if the operator  $\nu_\star^2 D_x^2+|\var_\star|^2 + \MM_0 - z^2$ is not invertible. Because of \eqref{eq:1p}, this operator is precisely $(\Di-z)(\Di+z)$. Hence, the zeros of $D_0(z)$ in $\Dd(0,\tvar)$ are exactly the eigenvalues of $\Di$ or $-\Di$. Since $\Di$ and $-\Di$ have the same spectrum -- see Theorem \ref{lem:1b} -- we deduce that the zeros of $D_0(z)$ are exactly the eigenvalues of $\Di$, with twice their multiplicities.

6. From \eqref{DzD0z}, $|\DDD_\delta(z) - D_0(z)| = O(\delta^{1/4})$ for $z\in\ove{\Dd(0,\tvar\delta)}$.  Hurwitz's theorem implies that the zeros of $\DDD_\delta$ converge to those of $D_0$: $\dist\left( \DDD_\delta^{-1}(0), \ D_0^{-1}(0) \right) \rightarrow 0$. Proposition \ref{prop:1q} now follows from \eqref{eq:1q} and~Step~5.
\end{proof}

\subsection{The twisted parametrix}\label{sec:5.3} In order to conclude the proof of Theorem \ref{thm:2}, we must show that the eigenvalues of $\Di$ each seed \textit{at most} one eigenvalue of $\PP_\delta$. However, a deeper look at the proof of Proposition \ref{prop:1q} shows that the poles of $(\Id+K_\delta(\lambda))^{-1}$ have twice the multiplicity of the eigenvalues of $\Di$.

If one believes in Theorem \ref{thm:2}, there is only one explanation: $Q_\delta(\lambda)$ must cancel out half the multiplicity of the poles of $(\Id+K_\delta(\lambda))^{-1}$. This surprising phenomena is however very unstable. We take advantage of this instability to produce a quick fix.

Given a family of rank-one operators $\Pi_\delta(\lambda) : L^2 \rightarrow H^2$ (depending on $\delta > 0$ and holomorphically on $\lambda$), we define a {\it twisted parametrix} $Q^\Pi_\delta(\lambda)$ by
\begin{equation}
Q^\Pi_\delta(\lambda) = Q_\delta(\lambda) + \Pi_\delta(\lambda).
\end{equation}
To this twisted parametrix corresponds a family of operators $K^\Pi_\delta(\lambda)$ such that 
\begin{equation}\label{eq:3s}
(\PP_\delta-\lambda)Q_\delta^\Pi(\lambda) = \Id + K_\delta(\lambda) + (\PP_\delta-\lambda)\Pi_\delta(\lambda) \de  \Id + K^\Pi_\delta(\lambda).
\end{equation}
Since $K^\Pi_\delta(\lambda)$ is a rank-one perturbation of $K_\delta(\lambda)$, it still forms a holomorphic family of compact operators on $L^2$. 

We show that $\Id + K_\delta^\Pi(\lambda_\delta)$ is invertible for some  $\lambda_\delta\in\C$; \cite[Theorem C.5]{DZ}  provides the meromorphic continuation of $(\Id + K_\delta^\Pi(\lambda_\delta))^{-1}$. We deduce that the resolvent of $\PP_\delta$ is expressed in terms of the twisted parametrix:
\begin{equation}
(\PP_\delta-\lambda)^{-1} = Q_\delta^\Pi(\lambda) \cdot \left( \Id + K^\Pi_\delta(\lambda) \right)^{-1}.
\end{equation}
In particular, the eigenvalues of $\PP_\delta$ must be within the poles of $\big(\Id + K^\Pi_\delta(\lambda)\big)^{-1}$. 

In the rest of the section, we consider operators $\Pi_\delta(\lambda)$ in the form
\begin{equation}\label{eq:3z}
\Pi_\delta(\lambda) = \dfrac{1}{\delta} \matrice{ \phi_+^\star \\ \phi_-^\star  }^\top \cdot \UU_\delta \cdot  f  \otimes g \cdot \UU_\delta^{-1} \cdot \ove{\matrice{ \phi_+^\star \\ \phi_-^\star  }}.
\end{equation}
Above, $f, g \in H^2(\R,\C^2)$ may depend on $\delta \in (0,1)$ and $\lambda \in \Dd(E_\star,\tvar \delta)$. They will be explicitly determined. 
 
\begin{prop}\label{lem:1c} Let $\tPi_\delta(z) \de (\PP_\delta-E_\star-z \delta) \cdot \Pi_\delta(E_\star+z \delta)$. There exists $C > 0$ such that for any $\delta > 0$, 
\begin{equation} 
\left\|\tPi_\delta(z)\right\|_{L^2} \leq C(1+|z|) \cdot |f|_{H^2} \cdot |g|_{L^2}.
\end{equation}
\end{prop}

\begin{proof}  Fix $f, g \in H^2(\R, \C^2)$. Then, using that $(D_x^2+V-E_\star)\phi_\pm^\star=0$ and that $D_x \UU_\delta(f)(x)=\delta \UU_\delta(D_x f)(x)$, we find
\begin{equations}
\big(\PP_\delta-E_\star-z \delta\big) \cdot \dfrac{1}{\delta} \matrice{ \phi_+^\star \\ \phi_-^\star  }^\top \cdot \UU_\delta\ \left( f \otimes g\right)\  \UU_\delta^{-1} \cdot \ove{\matrice{ \phi_+^\star \\ \phi_-^\star  }}\\
  = \dfrac{1}{\delta} \left(\matrice{ (\PP_\delta-E_\star-z \delta)\phi_+^\star \\ (\PP_\delta-E_\star-z \delta) \phi_-^\star  }^\top  \cdot \UU_\delta f  + 2\matrice{D_x \phi_+^\star \\ D_x \phi_-^\star  }^\top \cdot D_x \UU_\delta f  + \matrice{\phi_+^\star \\ \phi_-^\star  }^\top \cdot D_x^2\ \UU_\delta f  \right)\otimes\ g\ \UU_\delta^{-1} \cdot \ove{\matrice{ \phi_+^\star \\ \phi_-^\star  }}\\
  = \left( (\kappa_\delta W  - z) \matrice{ \phi_+^\star \\ \phi_-^\star  }^\top  \cdot \UU_\delta f  + 2\matrice{D_x \phi_+^\star \\ D_x \phi_-^\star  }^\top \cdot  \UU_\delta D_x f + \delta \UU_\delta D_x^2 f  \right)\otimes g\  \UU_\delta^{-1} \cdot \ove{\matrice{ \phi_+^\star \\ \phi_-^\star} }.
\end{equations}
A direct estimate shows that the $L^2$-operator norm of this projector is controlled by  
$C (1+|z|) \cdot |f|_{H^2} \cdot |g|_{L^2}$, 
as claimed. 
\end{proof}

The family $\lambda \in \Dd(E_\star,\tvar \delta) \mapsto K_\delta^\Pi(\lambda)$ is analytic. To adapt the arguments of \S\ref{ev-loc}, we first show that $\left(\Id + K_\delta^\Pi(\lambda)\right)^{-1}$ defines a meromorphic family of operators. Because of \cite[Theorem C.5]{DZ}, it suffices to prove that for $\delta$ sufficiently small, there exists $\lambda_\delta \in \C^+ \cup \Dd(E_\star,\tvar \delta)$ with $\Id+K_\delta^\Pi(\lambda_\delta)$ invertible. We set 
\begin{equation}
\lambda_\delta \de E_\star + i\delta\var',  \ \ \var' \de \dfrac{\tvar+\var_N}{2} \ \ \  (\text{note } \var_N<\var'<|\var_\star|).
\end{equation}

\begin{prop}\label{prop:1n} There exist $\delta_0 > 0$, $\theta_0 > 0$ such that if
\begin{equation}\label{eq:3k}
|f|_{H^2} \cdot |g|_{L^2} \leq \te_0, \ \ \delta \in (0,\delta_0)
\end{equation}
then $\Id+K_\delta^\Pi(\lambda_\delta)$ is invertible.
\end{prop}

\begin{proof}[Proof of Proposition \ref{prop:1n}] 1. Recall that $\KK_\delta(z)$ is given by \eqref{eq:2p}. We first show that $\Id+\KK_\delta(i\var')$ is invertible, with bounds on its inverse. In the proof of 
Proposition \ref{prop:1q} (see \eqref{DzD0z})  we showed that
\begin{equation}
\Det\big(\Id+\KK_\delta(i\var')\big) = D_\delta(i\var') = D_0(i\var')+O\left(\delta^{1/4}\right). 
\end{equation}
In addition, we know that $ D_0(i\var') \neq 0$ since $i\var'$ is not an eigenvalue of $\Di$. Therefore for $\delta$ sufficiently small, $\Det(\Id+\KK_\delta(i\var'))$ is non-zero. This implies that $\Id+\KK_\delta(i\var')$ is invertible. Furthermore, bounds on the inverse are obtained using Cramer's  rule, see e.g. \cite[Lemma 5.1]{Dr2}:
\begin{equation}
\left\| ( \Id+\KK_\delta(i\var') )^{-1} \right\|_{L^2}\  \leq\ 
 \dfrac{ e^{2\|\KK_\delta(i\var')\|_{\LL(L^2)}}}{\big|\Det\big(\Id+\KK_\delta(i\var')\big)\big|}
  \leq C e^{2\|\KK_\delta(i\var')\|_{\LL(L^2)}}.
\end{equation}
The trace-class norm of $\KK_\delta(\lambda_\delta)$ is uniformly bounded because it is the product of a uniformly bounded multiplication operator with a $\delta$-independent Fourier multiplier mapping $L^2$ to $H^2$.

2. By \eqref{eq:1j} we have that 
\begin{equations}
\Id+K_\delta^\Pi(E_\star + z\delta) = \Id+\KK_\delta(z) + \EE_\delta(z) + \tPi_\delta(z),
\end{equations}
where $\tPi_\delta(z)$ was defined in Proposition \ref{lem:1c}.
The operator $\Id+\KK_\delta(i\var')$ is invertible, with uniform bounds on its inverse. Therefore for $\Id+K_\delta^\Pi(\lambda_\delta)$ to be invertible it suffices that $\|\EE_\delta(i\var') + \tPi(i\var)\|_{L^2}$ gets arbitrarily small as $\delta \to 0$.  By \eqref{eq:2p} and Proposition \ref{lem:1c},
\begin{equation}
\left\| \EE_\delta(\lambda_\delta) + \tPi_\delta(i\var') \right\|_{L^2} \leq C\delta^{1/3} + C(1+\var') \cdot  |f|_{{H^2}} \cdot |g|_{{L^2}}.
\end{equation}
Therefore, if $\delta$ and 
$|f|_{{H^2}} \cdot |g|_{{L^2}}$ are sufficiently small then $\Id+K_\delta^\Pi(\lambda_\delta)$ is invertible. This completes the proof of Proposition \ref{prop:1n}. \end{proof}

By Proposition \ref{prop:1n} and analytic Fredholm theory (\S\ref{fred-thy}), we know that whenever \eqref{eq:3k} is satisfied, the eigenvalues of $\PP_\delta$ in the disk $\Dd(E_\star,\tvar \delta)$ must be among the poles of $\left(\Id+K_\delta^\Pi(\lambda)\right)^{-1}$. We now have all the ingredients to prove Theorem \ref{thm:2}.

\begin{proof}[Proof of Theorem \ref{thm:2}] The first point of Theorem \ref{thm:2} was proved in Propopsition \ref{prop:1q}. It remains to eliminate the spurious modes discussed above. We show that for every $j \in [-N,N]$, for some $\epsi > 0$ and $\delta_0 > 0$ and for $\delta \in (0,\delta_0)$, $\PP_\delta$ has at most one eigenvalue in $\big[ E_\star+\delta(\var_j-\epsi),  E_\star+\delta(\var_j+ \epsi) \big]$.

 1. Assume that $\var_j$ is an eigenvalue of $\Di$. We now construct $\Pi_\delta(\lambda)$ of the 
form \eqref{eq:3z} by choosing $f \in \ker(\Di-\var_j)$ with $|f|_{L^2} = 1$.  We next set 
 \begin{equation}
  g \de t \cdot \left(\nu_\star^2 D_x^2 + |\var_\star|^2-\oz^2 \right)^{-1} \chi f.
  \end{equation}
  Above $\chi$ is a function with compact support which takes values in $[0,1]$.  We now fix
\begin{equation}
t \de \dfrac{\te_0 \cdot \left(|\var_\star|^2 - |\tvar|^2 \right)}{2 |f|_{H^2}}.
\label{t-pick}\end{equation}
The $L^2$-operator norm of $\left(\nu_\star^2 D_x^2 + |\var_\star|^2-\oz^2 \right)^{-1}$ is bounded by $\left(|\var_\star|^2-\oz^2 \right)^{-1}$. 
With the choice \eqref{t-pick} of $t$,  the functions $f$ and $g$ satisfy \eqref{eq:3k}. Therefore $\lambda \in \Dd(E_\star,\tvar \delta) \mapsto (\Id+K_\delta^\Pi(\lambda))^{-1}$ forms a meromorphic family.

The eigenvalues of $\PP_\delta$ are among the poles of $\left(\Id+K_\delta^\Pi(\lambda)\right)^{-1}$. In other words, if $\lambda$ is an eigenvalue of $\PP_\delta$ then $\Id+K_\delta^\Pi(\lambda)$ is not invertible. The decomposition \eqref{eq:1j} implies that
 \begin{equation}
 \Id + K_\delta^\Pi(E_\star+\delta z) = \Id + \EE_\delta(z) + \KK_\delta(z) + \tPi_\delta(z)
 \end{equation}
 where $\EE_\delta(\lambda) = \OO_{L^2}(\delta^{1/3})$ and $\KK_\delta(z) + \tPi_\delta(z)$ is a trace-class operator. Hence
\begin{equations}\label{eq:1r}
\lambda=E_\star+z \delta \text{ is an eigenvalue of }  \PP_\delta \ \Rightarrow \  \DDD_\delta^\Pi(z) = 0,
\\
\DDD_\delta^\Pi(z) \de \Det\left(\Id+\big(\Id+\EE_\delta(z)\big)^{-1} \KK^\Pi_\delta(z)\right),\ \ \ \ \KK_\delta^\Pi(z) \de \KK_\delta(z) + \tPi_\delta(z).
\end{equations}

2. The proof of Theorem \ref{thm:2} reduces now to a problem in complex analysis; we must locate the zeros of $\DDD_\delta^\Pi$. Similarly to Step 2 in the proof of Proposition \ref{prop:1q}, we have
\begin{equation}
\DDD^\Pi_\delta(z) = D^\Pi_\delta(z) + O\left(\delta^{1/3}\right), \ \ D^\Pi_\delta(z) = \Det\left(\Id+\KK_\delta^\Pi(z)\right).
\end{equation}

3. Applying the same calculation as in Step 3 of Proposition \ref{prop:1q}, we get
\begin{equation}
D^\Pi_\delta(z)= \Det\left( \Id + (\MM_\delta + X_\delta) \left(\nu_\star^2 D_x^2+|\var_\star|^2- z^2\right)^{-1} \right).
\end{equation}
Above, $\MM_\delta$ is the multiscale matrix-valued highly oscillatory potential \eqref{eq:3l},
  and $X_\delta$ is a multiscale highly oscillatory rank-one operator, associated with $\Pi_\delta(\lambda)$:
\begin{equation}
X_\delta  \de t\cdot   \UU_\delta^{-1} \cdot \ove{\matrice{ \phi_+^\star \\ \phi_-^\star} } \left( (\kappa_\delta W  - z) \matrice{ \phi_+^\star \\ \phi_-^\star  }^\top  \cdot \UU_\delta f  + 2\matrice{D_x \phi_+^\star \\ D_x \phi_-^\star  }^\top \cdot  \UU_\delta D_x f \right)\otimes f \chi.
\end{equation}

4. To anticipate the limiting behavior of $D^\Pi_\delta(z)$ we calculate the weak limit of 
the operator $\MM_\delta + X_\delta$ as $\delta$ tends to zero. The weak limit $\MM_0$ of $\MM_\delta$ is displayed in \eqref{W0-def}. We now study the convergence of $X_\delta$ as $\delta$ tends to $0$. We observe that
\begin{equation}
X_\delta = t \cdot (F_{\delta} + G_{\delta}) \otimes \chi f,
\end{equation}
where $F_{\delta}(y) = F(y/\delta, y)$ and $G_{\delta}(y) = G(y/\delta, y)$. The functions $F$, $G$ are both one-periodic in the first variable and  exponentially decaying in the second:
\begin{equations}
F(x,y) \de \big(\kappa(y) W (x) - z \big) \cdot f(y) \cdot \ove{\matrice{ \phi_+^\star(x) \\ \phi_-^\star(x)} } 
 \matrice{ \phi_+^\star(x) \\ \phi_-^\star(x)  }^\top , \\ 
 G(x,y) \de 2(D_yf)(y) \cdot \ove{\matrice{ \phi_+^\star(x) \\ \phi_-^\star(x)} } \matrice{D_x \phi_+^\star(x) \\ D_x \phi_-^\star(x)  }^\top.
\end{equations}
Using the relations \eqref{w-limits} we have 
\begin{equations}
X_\delta\  \rightharpoonup\ t\cdot \left(\kappa\ \bss - z\right) f\otimes f\chi +\ t\cdot \bst \nu_\star D_y f\otimes f\chi
\\ 
 = t \cdot (\Di-z) \cdot f \otimes f \chi \de X_0.
\end{equations}

Using Lemma \ref{lem:1u} we can obtain quantitative bounds in $H^{1/4}$ for the convergence of $F_{\delta}$ and $G_{\delta}$ to their weak limit as $\delta$ goes to zero -- in the same spirit as Step 4  in the proof of Proposition \ref{prop:1q}. We deduce that $\|X_\delta  - X_0\|_{H^{1/4} \rightarrow H^{-1/4}} = O(\delta^{1/4})$ because of \cite[Lemma 2.1]{Dr4}.

 Thanks to Lemma \ref{lem:1v},  $D_\delta^\Pi(\lambda) = D_0^\Pi(\lambda) + O(\delta)$ with
\begin{equation}\label{eq:2k}
D_0^\Pi(\lambda) \de \Det\left(\Id + \left(\MM_0+X_0\right) \left(\nu_\star^2 D_x^2+\var_\star^2-z^2\right)^{-1}  \right).
\end{equation}

5. Hence $z$ is a zero of $D_0^\Pi(z)$ if and only the operator
\begin{equation}
 \Di^2-z^2 +   t\left(\Di-z\right)\cdot f \otimes f\chi = (\Di-z) \left( \Di+z +  t \cdot f \otimes f \chi \right)
\end{equation}
is not invertible. Therefore, the zeros of $D_0^\Pi(z)$ are exactly the eigenvalues of $\Di$ and of $\Di + t f \otimes f \chi$.

6. With $t$ as in \eqref{t-pick},  we fix $\epsi \in (0,t)$ small enough, so that the only eigenvalue of $\Di$ in the interval $[\var_j-\epsi,\var_j+\epsi]$ is $\var_j$.
\footnote{Here the eigenvalues of $\Di$ are simple -- see Theorem  \ref{thm:5}. This makes the analysis a bit simpler: the parametrix twist is rank-one. That being said, the techniques developed in this section only rely on moving  eigenvalues simultaneously.  They would still work in more general situations where the parametrix cancels poles of \textit{higher} multiplicity. The parametrix twist   $\Pi_\delta(\lambda)$ would simply need to be of higher rank -- equal to the number of spurious modes.} 
The eigenvalues of $\Di + t \cdot f \otimes f$ are given by 
\begin{equation}
\mu_k \de \var_k \text{ if } k \neq j; \ \ \ \ \mu_j = \var_j+t \ \ (\text{when } k=j). 
\end{equation}
In particular, they all lie outside $[\var_j-\epsi,\var_j+\epsi]$. By taking $\chi$ equal to $1$ on a sufficiently large set we can arrange for all the eigenvalues of 
\begin{equation}
\Di + t \cdot f \otimes f \chi
\end{equation}
to also lie outside $[\var_j-\epsi,\var_j+\epsi]$. Hence the only zero of $D_0^\Pi$ in the disc $\ove{\Dd(\var_j,\epsi)}$ is the eigenvalue $\var_j$ of $\Di$. 

Using \eqref{eq:2k}, Hurwitz's theorem implies that for $\delta$ sufficiently small, $D_\delta^\Pi$ has precisely one zero in $\Dd(\var_j, \epsi)$. From \eqref{eq:1r}, this completes the proof of Theorem \ref{thm:2}.
\end{proof}

\appendix

\section{From quasimodes to eigenvalues in gaped selfadjoint problems}\label{sec:B}

\begin{proof}[Proof of Lemma \ref{lem:1r}] If $E$ is an eigenvalue of $T$, the lemma is proved. Otherwise $T-E : \HH \rightarrow \HH$ is invertible; from the assumption, $E$ is neither in the pure point spectrum nor in the essential spectrum. From \eqref{q-mode}, $|(T-E)^{-1}|_{\HH}\ge\epsilon^{-1}$. But because of the spectral theorem,
\begin{equation}
\left\|(T-E)^{-1}\right\|_{\HH} = \dfrac{1}{\dist(E, \Sigma(T))}.
\end{equation}
We deduce $\dist(E,\Sigma(T)) \leq \epsilon$, therefore $E$ is $\epsilon$-close to the spectrum of $T$. Since $0 < \epsilon < \epsilon_0$, and $T$ has no essential spectrum in $[E-\epsilon_0,E+\epsilon_0]$, $T$ must have an eigenvalue $\epsilon$ in $[E-\epsilon,E+\epsilon]$.

Assume now that $T$ has only one eigenvalue -- say $\lambda$ --  in $[E-C\epsilon,E+C\epsilon]$. Let $u$ be a corresponding normalized eigenvector. Define $T'$ formally as $T$, but acting on the space $\HH' = \textrm{span}\{u\}^\perp$. The spectral theorem implies that
\begin{equation}
\left\|(T'-E)^{-1}\right\|_{\HH'}= \frac{1}{\dist(E,\Sigma(T'))}\le \dfrac{1}{C\epsilon}.
\end{equation}
Therefore we have 
\begin{equation}
w\in\HH' \ \Rightarrow \
C\epsilon \cdot |w|_{\HH'} \leq |(T'-E)w|_{\HH'}.
\label{lb-w}\end{equation}
If $v$ satisfies \eqref{q-mode}, write  $v = \az u + w$ where $w \in \HH'$ and $\az \in \C$. Then,
\begin{equation}
(T-E) v = \az (T-E)u + (T-E)w = \az (\lambda-E)u + (T'-E)w.
\end{equation}
This is because $Tu = \lambda u$ and $w\in\HH'$, hence $Tw = T' w$. The RHS is an orthogonal decomposition of $(T-E) v$. 
Thus $|(T-E)v|_{\HH} \geq |(T'-E)w|_{\HH}=|(T'-E)w|_{\HH'}$. From \eqref{q-mode} the lower bound \eqref{lb-w} and the relation $w=v-\alpha u$, we deduce
\begin{equation}
\epsilon \geq |(T-E)v|_{\HH} \geq |(T'-E)w|_{\HH'} \geq C\epsilon |w|_{\HH'}=C\epsilon |v-\az u|_{\HH'}.
\end{equation}
Thus, $|v-\az u|_{\HH}  \leq  C^{-1}$; since  $\az u$ is an eigenvector of $T$, the proof is complete. \end{proof}

\section{Spectral estimates for highly oscillatory operators}\label{sec:A}

\begin{lem}\label{lem:1u} Let $F \in C^\infty(\R \times \R)$ be $1$-periodic in the first variable and such that there exists $\te > 0$ with
\begin{equation}
\textrm{for all}\quad  \az\ge0, \ \ \sup_{(x,y) \in \R^2} e^{\te |y|} \cdot \left| \p_y^\az F(x,y) \right| < \infty.
\end{equation}
Set $f_\delta(y) = F(y/\delta,y)$. Then there exists $C > 0$ such that for $\delta \in (0,1)$,
\begin{equation}
|f_\delta - g|_{H^{-1/4}} \leq C\delta^{1/4}, \ \ \ \ g(y) = \int_0^1 F(x,y) dx.
\end{equation}
\end{lem}

\begin{proof} This is mostly contained in the proof of \cite[Theorem 1]{Dr2}. Write a Fourier decomposition of $F$:
\begin{equation}
F(x,y) = \sum_{m \in 2\pi\Z} a_m(y) e^{imx}, \ \ \ \ a_m(y) = \int_0^1 F(x,y) e^{-imx} dx.
\end{equation}
Then,
\begin{equations}
\left|f_\delta - g\right|_{H^{-{1/4}}} = \left|\blr{D}^{-{1/4}}(f_\delta-g)\right|_{L^2} \leq  \sum_{m \in 2\pi\Z \setminus \{0\} } \left|\blr{\ \cdot \ }^{-{1/4}} \widehat{a_m}(\ \cdot-m/\delta)\right|_{L^2} \\
=  \sum_{m \in 2\pi\Z \setminus \{0\}} \left|\blr{\ \cdot \ }^{-{1/4}} \blr{\ \cdot-  m/\delta}^{-{1/4}} \blr{\ \cdot-m/\delta}^{{1/4}} \widehat{a_m}(\ \cdot-m/\delta)\right|_{L^2} \\
\leq \sum_{m \in 2\pi\Z \setminus \{0\}} \left( \sup_{\zeta \in \R} \blr{\zeta}^{-{1/4}} \blr{\zeta-m/\delta}^{-{1/4}} \right) \cdot \left|\blr{\ \cdot-m/\delta}^{{1/4}} \widehat{a_m}(\ \cdot-m/\delta)\right|_{L^2}.
\end{equations}
We now apply Peetre's inequality to the first factor and deduce that
\begin{equations}
\left|f_\delta - g\right|_{H^{-{1/4}}} \leq\sum_{m \in 2\pi\Z \setminus \{0\}} \blr{m/\delta}^{-{1/4}} \cdot   \left|\blr{\ \cdot \ }^{{1/4}} \widehat{a_m}\right|_{L^2} \leq C\delta^{1/4} \sum_{m \in 2\pi\Z \setminus \{0\}} \dfrac{|a_m|_{H^{1/4}}}{m^{1/4}}.
\end{equations}
Because $F \in C^\infty(\R \times \R)$ decays sufficiently fast, the sum on the RHS is absolutely convergent. This completes the proof.\end{proof}

\begin{proof}[Proof of Lemma \ref{lem:1v}] 1. We first use the cyclicity of the determinant to write
\begin{equation}
\Det(\Id + A \chi B) = \Det\left(\Id + \blr{D}^{-s} A \blr{D}^{-s} \cdot \blr{D}^s \chi B \blr{D}^{s}\right).
\end{equation}
The trace-class norm controls the determinant difference -- see e.g. \cite[Proposition B.26]{DZ}. Hence, 
\begin{equations}
\big| \Det(\Id+A \chi B) - \Det(\Id+A'\chi B) \big| \leq C \left\| \blr{D}^{-s} (A-A') \blr{D}^{-s} \cdot \blr{D}^s \chi B \blr{D}^{s} \right\|_{\LL(L^2)}
\\
\leq C \left\| \blr{D}^{-s} (A-A') \blr{D}^{-s} \right\|_{L^2} \cdot  \left\| \blr{D}^s \chi B \blr{D}^{s} \right\|_{\LL(L^2)}.  
\end{equations}
The space of bounded operators from $L^2$ to $H^{2-2s}$ with range in function with fixed compact support embeds continuously in the space of trace-class operators on $L^2$, because $2-2s > 1$ -- see e.g. \cite[Equation (B.3.9)]{DZ}. Therefore, 
\begin{equations}
\big| \Det(\Id+A \chi B) - \Det(\Id+A'\chi B) \big| \leq C \left\| \blr{D}^{-s} (A-A') \blr{D}^{-s} \right\|_{L^2} \cdot  \left\| \chi B \right\|_{H^{-s} \rightarrow H^{2-s}}.
\end{equations}
This completes the proof of Lemma \ref{lem:1v}. \end{proof}

\section{The point spectrum of $\Di$ is simple}\label{app:3}

\begin{theorem}\label{thm:5} The eigenvalues of $\Di$ are simple.
\end{theorem}

 \subsection{Eigenvalues as of $\Di$ as zeros of Evans' function}\label{Evans-fn} Recall that
\begin{equation}
\Di = \nu_\star \bst D_x + \bss \kappa, \ \ \ \ \bss  = \matrice{0 & \ove{\var_\star} \ \\ \var_\star & 0}.
\end{equation}
Without loss of generality, we assume that $\nu_\star = 1$ (this only amounts to a change of variables $x \mapsto \nu_\star x$). We first rewrite the equation $(\Di - z) \az = 0$ as
\begin{equation}
D_x \az  = A(\cdot,z) \cdot \az, \ \ \ \ A(x,z) \de \bss \bst\kappa(x) + z \bst.
\label{direct}
\end{equation}
We used above that $\bst\bss=-\bss\bst$. 
We use the framework  in \cite{PW:92}. We introduce the adjoint system
\begin{equation}\label{adjoint}
D_x\beta = -\beta \cdot A(\cdot,z).
\end{equation}
We observe that if $\az$ solves \eqref{direct} and $\beta$ solves \eqref{adjoint}, then $D_x(\beta \cdot \az) = 0$ -- where $\beta \cdot \az$ is the scalar obtained by matrix multiplication. In particular, the quantity $\beta(x) \cdot \az(x)$ does not depend on $x$.

We are interested in solutions of \eqref{direct} and \eqref{adjoint} determined by the eigenvectors  of the asymptotic matrices:
\begin{equation}
 A_\pm(z) \de \lim_{x \to \pm\infty}A(x,z) = \matrice{ z & \mp \ove{\var_\star} \ \\ \pm \var_\star  & -z}.
\end{equation}
Introduce 
\begin{equation}
\mu(z) \de i\sqrt{|\var_\star|^2-z^2}, \ \ \ \ z \in \Omega \de \C \setminus \big( (-\infty,-|\var_\star|\ ]\cup[\ |\var_\star|,+\infty) \big).
\end{equation}
The complex square root on $\C \setminus (-\infty,0]$ has range $\{ z : \Re(z )> 0\}$. The matrices  $A_\pm(z)$ have distinct eigenvalues, given by $\pm\mu(z)$. Therefore, they admit left-and-right   eigenvectors $f^\pm(z)$ and $g^\pm(z)$ depending analytically on $z\in\Omega$: 
\begin{equations}
\left( A_+(z) - \mu(z) \right) \cdot f^+(z) = 0, \ \ \ \ g^-(z)\cdot \left( A_-(z) - \mu(z) \right) = 0, \\
\left( A_-(z) + \mu(z) \right) \cdot f^-(z) = 0, \ \ \ \ g^+(z)\cdot \left( A_+(z) + \mu(z) \right) = 0.
\end{equations}
We normalize these eigenvectors so that $g^\pm(z) f^\pm(z)=1$.

We focus on $z \in (-|\var_\star|,|\var_\star|)$, where $\Di$ may have point spectrum. For such $z$, $i\mu(z) \in (-\infty,0)$. We now introduce two pairs of solutions of \eqref{direct} and \eqref{adjoint}: 
\begin{itemize}
\item The system \eqref{direct} has a one-dimensional subspace of solutions which decays as $x$ tends to $+\infty$. We can choose such a solution $\zeta^+(x,z)$  which satisfies
 \begin{equation}
 \zeta^+(x,z) =  f^+(z) e^{i\mu(z) x} \quad \textrm{as}\ \ x \gg 1.
\end{equation}
 \item The adjoint system  \eqref{adjoint} has a  one-dimensional subspace of solutions
 which decays as $x\to-\infty$. We can choose such a solution $\eta^-(x,z)$  which satisfies
 \begin{equation}
 \eta^-(x,z) = g^-(z) e^{-i\mu(z) x} \quad \textrm{as}\ \ x \ll -1.
 \end{equation}
 \item The system \eqref{direct} has a one-dimensional subspace of solutions which decays as $x$ tends to $-\infty$. We can choose such a solution $\zeta^-(x,z)$  which satisfies
\begin{equation}
\zeta^-(x,z) = f^-(z) e^{-i\mu(z) x} \quad \textrm{as}\ \ x \ll -1.
\end{equation}
\item The adjoint system  \eqref{adjoint} has a  one-dimensional subspace of solutions
 which decay as $x\to+\infty$. We can choose such a solution $\eta^+(x,z)$  which satisfies
 \begin{equation}
 \eta^+(x,z) =  g^+(z) e^{i\mu(z) x} \quad \textrm{as}\ \ x \gg 1.
\end{equation}
\end{itemize}
These all have analytic continuations to $\Omega$ because the eigenvalues of $A_\pm(z)$ are simple.
 
\begin{definition}\label{evans}
For $z\in \Omega$, we define Evans' functions
   \begin{equation}
   D(z) \de \eta^-(x,z) \cdot \zeta^+(x,z),\ \ \ \ 
    \tD(z) \de \eta^+(x,z)\cdot \zeta^-(x,z).
\end{equation}
(The above expressions are independent of $x$ because of the conservation law).
 \end{definition}

We have the following result:

\begin{theorem}\cite[Theorem 1.9]{PW:92}  \label{thm:4}
\begin{enumerate}
\item There exists an analytic function $\psi$ on $\Omega$ such that $\tD(z) = \psi(z) \cdot  D(z)$.
\item Let $z_0\in(-|\var_\star|,|\var_\star|)$. The solution $\zeta^+(x,z_0)$ of  \eqref{direct}  is exponentially decaying as $|x|\to \infty$ 
  if and only if $D(z_0)=0$ (and $ \tD(z_0)=0)$. When this is satisfied,  there exists $C_\zeta \neq 0$ with
   $\zeta^-(x,z_0)=C_\zeta\cdot \zeta^+(x,z_0)$. 
\item Let $z_0\in(-|\var_\star|,|\var_\star|)$. The solution $\eta^-(x,z_0)$ of  \eqref{adjoint}  is exponentially decaying as $|x|\to \infty$ 
  if and only if $ \tD(z_0)=0$ (and $D(z_0)=0)$. When this is satisfied,  there exists $C_\eta \neq 0$ with
   $\eta^+(x,z_0)=C_\eta \cdot \eta^-(x,z_0)$. \end{enumerate}
\end{theorem}

\begin{cor}\label{pt-spec}
A point $z_0 \in (-|\var_\star|, |\var_
\star|)$ is an eigenvalue of $\Di$ if and only if 
$D(z_0)=\tD(z_0)=0$.
\end{cor}

\subsection{Resolvent kernel; projector; and proof of Theorem \ref{thm:5}}\label{resolvent}

We prove a formula for the resolvent of $\Di$ using $\zeta^\pm$ and $\eta^\pm$. We study the equation $\left(\Di -z\right)\az = \vp$, where $\vp \in C^\infty_0(\R,\C^2)$ and $z \in (-|\var_\star|, |\var_\star|)$. This is equivalent to
   \begin{equation}
   D_x \az(x)= A(x,z)\cdot \az(x)+ \Phi(x), \ \ \ \  \Phi(x) 
   \de \bst \vp(x).
   \label{inhom}
   \end{equation}
   We proceed by variation of constants. We seek $\az$ in the form $\az(x)=a(x) \cdot \zeta^+(x,z)+b(x) \cdot \zeta^-(x,z)$, where 
   $\zeta^\pm$ were introduced in \S\ref{Evans-fn}. Substituting into \eqref{inhom} we find that $a$ and $b$ must satisfy
   \begin{equation}
   D_x a(x) \cdot \zeta^+(x,z) + D_x b(x) \cdot \zeta^-(x,z) = \Phi(x).
   \label{apbp}
   \end{equation}
 
From Definition \ref{evans}, $\eta^-(x,z) \cdot \zeta^+(x,z)=D(z)$ and $\eta^+(x,z) \cdot \zeta^-(x,z)= \tD(z)$. Moreover, 
    we have  $\eta^-\cdot \zeta^- \equiv 0$ and $\eta^+ \cdot \zeta^+ \equiv 0$ (by evaluation as $x\to \pm \infty$ and use of the conservation law).
    Therefore, left multiplication of \eqref{apbp} by $\eta^-(\cdot,z)$ gives $D(z) D_x a(x)=\eta^-(x,z) \Phi(x)$ and 
     left multiplication of \eqref{apbp} by $\eta^+(\cdot,z)$ gives $ \tD(z) D_x b(x)=\eta^+(x,z) \Phi(x)$. The identity $\Phi=\bst\varphi$ implies
      \begin{equation}
      u(x) =  \int_{-\infty}^x i\frac{\eta^-(y,z)}{D(z)}\cdot \bst \varphi(y) dy \cdot \zeta^+(x,z) - \int_x^{+\infty}i\frac{\eta^+(y,z)}{ \tD(z)} \cdot \bst \varphi(y) dy\cdot \zeta^-(x,z)
      \end{equation}
      is a solution of \eqref{inhom} which decays as $x\to\pm\infty$. Since $\Phi=\bst \vp$, we proved:
      
\begin{prop}\label{prop:1a} For $z\notin\Sigma(\Di)$, $\left(\Di - z\right)^{-1}$  has kernel
\begin{equation}
 H(x,y,z) \de \systeme{      \dfrac{i}{D(z)} \cdot \eta^-(y,z) \cdot \bst \zeta^+(x,z), & y<x, \\
   -\dfrac{i}{\tD(z)} \cdot \eta^+(y,z) \cdot  \bst  \zeta^-(x,z), &  x<y. }
\end{equation}
\end{prop}

 The following proposition is a consequence of  \cite[Theorem 1.11]{PW:92} -- the proof there is written for $D'(z_0)$ but applies  to  $\widetilde{D}'(z_0)$ as well. 
 
    \begin{prop}\label{Dprime}
     Let $z_0\in(-|\var_\star|,|\var_\star|)$ be an eigenvalue of $\Di$.       Then, 
      \begin{equations} 
  D(z_0)=0, \ \ \ \   D'(z_0)  =  -  i  \int_\R  \eta^-(x,E_0) \cdot \bst  \zeta^+(x,z_0)  dx \\
      \tD(z_0)=0, \ \ \ \   \widetilde{D}'(z_0)   =  i  \int_\R  \eta^+(x,z_0)\cdot \bst  \zeta^-(x,z_0)  dx.
      \end{equations}
      \end{prop}

    Assume that $z_0 \in(-|\var_\star|,|\var_\star|)$ is an eigenvalue of $\Di$. We show that  $\overline{\bst\zeta^+(\cdot,z_0)}^\top$ solves the adjoint equation \eqref{adjoint} and decays as $x\to-\infty$. Indeed, suppose 
 \begin{equation}        D_x\az  =  A(\cdot,z) \cdot \az  = (\bss \bst\kappa   + z\bst ) \az. 
   \end{equation}
Multiplying by $\bst$ and using $\bst \bss = -\bss \bst$, we deduce
  \begin{equation}
 D_x(\bst \az) =  (\bss \bst\kappa +z\bst) \cdot \bst \az.
   \end{equation}
   Next take the transpose, then the complex conjugate to obtain
   \begin{equations}
  D_x(\bst\az)^\top = (\bst\az)^\top \cdot (\bst \bss \kappa + z\bst),
  \\
  D_x \ove{(\bst\az)}^\top = -\ove{(\bst\az)}^\top \cdot (\bst \bss \kappa + z\bst) = - \ove{(\bst\az)}^\top \cdot A(\cdot,z).
  \end{equations}
Since  $\az =\zeta^+(\cdot,z_0)$ decays exponentially as $x\to\pm\infty$ (because $z_0$ is an eigenvalue), the above calculation shows that $[\overline{\bst\zeta^+(\cdot,z_0)}]^\top$ solves the adjoint equation \eqref{adjoint} and decays exponentially as $x\to-\infty$.
       Therefore, for some constant $\gamma\in\C$, 
        \begin{equation}
         \eta^-(x,z_0)= \gamma \cdot  \overline{\bst\zeta^+(x,z_0)}^\top .\label{adj-eta}
         \end{equation}
        Because of Theorem \ref{thm:4} and \eqref{adj-eta}, we deduce that
        \begin{equation}
        \zeta^-(x,z_0)= C_\zeta\cdot \zeta^+(x,z_0)\ \textrm{ and } \
         \eta^+(x,z_0)=C_\eta\cdot \eta^-(x,z_0)= 
         C_\eta \gamma \cdot  \overline{\bst\zeta^+(x,z_0)}^\top.
        \label{zepm} \end{equation}
Substituting \eqref{zepm} in Proposition \ref{Dprime}, we deduce that

\begin{prop}\label{simple} Let $z_0$ be an eigenvalue of $\Di$. Then, 
$D(z_0)= \tD(z_0) =0$ and 
\begin{equations} 
 D'(z_0)=-i  \gamma \cdot |\zeta^+(\cdot,z_0)|_{L^2}^2\ne0, \ \ \ \ 
 \tD'(z_0)  =i  C_\eta  C_\zeta  \gamma \cdot  |\zeta^+(\cdot,z_0)|_{L^2}^2 \neq 0.
  \end{equations}
\end{prop}

This shows that if $z_0$ is an eigenvalue of $\Di$ then it is a simple zero of $z\mapsto D(z)$, i.e.  $D(z_0)=0$ and $D'(z_0) \neq 0$. Let $\Pi_{z_0}$ be the eigenprojector associated to $z_0$. We compute its kernel: from the Cauchy formula,
\begin{equation}
\Pi_{z_0}(x,y) = -\dfrac{1}{2\pi i} \oint_{z \in \Dd(z_0,r)} H(x,y,z) dz. 
\end{equation}
Proposition \ref{prop:1a} implies
\begin{align}
\Pi_{z_0}(x,y) &=
\begin{cases}
  -\dfrac{i}{D'(z_0)} \cdot \eta^-(y,z_0) \cdot  \bst \zeta^+(x,z_0), & x > y\\
  \ \ \dfrac{i}{\tD'(z_0)} \cdot \eta^+(y,z_0) \cdot  \bst  \zeta^-(x,z_0), &x<y.
  \end{cases}
  \label{eq:3b}
  \end{align}

Furthermore, using \eqref{adj-eta} and \eqref{zepm} we have
 \begin{equations}\label{emzp}
i  \eta^-(y,z_0)\cdot  \bst \zeta^+(x,z_0) 
= i  \gamma\cdot  \overline{\zeta^+(y,z_0)} ^\top\cdot  \zeta^+(x,z_0), \\
\eta^+(y,z_0)\  \bst \zeta^-(x,z_0) =i \gamma C_\eta C_\zeta\cdot  \overline{\zeta^+(y,z_0)} ^\top\cdot  \zeta^+(x,z_0).
\end{equations}
Substitution of \eqref{emzp} into \eqref{eq:3b} and using Proposition \ref{simple} yields 
\begin{equation}
\Pi_{z_0}(x,y) = \overline{\Psi_0(y)}^\top \cdot \Psi_0(x), \ \ \ {\rm where}\quad \Psi_0(y)= \dfrac{\zeta^+(y,z_0)}{|\zeta^+(\cdot,z_0)|_{L^2}}.
\end{equation}
 This is the kernel of a rank-one operator.
This concludes the proof. 

\end{document}